\documentclass[microtype]{gtpart}     
\usepackage{mathtools}
\usepackage{amscd}
\usepackage{interval}
\usepackage{bm}
\usepackage{float}
\usepackage{pinlabel}
\usepackage{enumerate}
\usepackage{subfig} 
\usepackage{color}

\title[Scl in graphs of groups]{Scl in graphs of groups}

\author{Lvzhou Chen}
\address{Department of Mathematics, University of Chicago, Chicago, Illinois, 60637}
\email{lzchen@math.uchicago.edu}
\givenname{Lvzhou}
\surname{Chen}


\newcommand{\defeq}{\vcentcolon=}

\newcommand{\scl}{\mathrm{scl}}
\newcommand{\BS}{\mathrm{BS}}
\newcommand{\conv}{\mathrm{conv}}
\newcommand{\Dom}{\mathrm{Dom}}
\renewcommand{\Im}{\mathrm{Im}}
\newcommand{\supp}{{\rm supp}}
\newcommand{\bdry}{\partial}
\newcommand{\ug}{\underline{g}}
\newcommand{\uga}{\underline{\gamma}}

\newtheorem{thm}{Theorem}[section]
\newtheorem{lemma}[thm]{Lemma}
\newtheorem{prop}[thm]{Proposition}
\theoremstyle{definition}
\newtheorem{definition}[thm]{Definition}
\newtheorem{remark}[thm]{Remark}

\newtheorem{example}[thm]{Example}
\newtheorem{cor}[thm]{Corollary}

\numberwithin{equation}{section}

\begin{document}

\begin{abstract}    
	Let $G$ be a group acting on a tree with cyclic edge and vertex stabilizers. Then stable commutator length (scl) is rational in $G$. Furthermore, scl varies predictably and converges to rational limits in so-called ``surgery'' families. This is a homological analog of the phenomenon of geometric convergence in hyperbolic Dehn surgery.
\end{abstract}

\maketitle

\tableofcontents

\section{Introduction}\label{sec: intro}
If $X$ is a space, there is a (pseudo-)norm on $H_2(X)$ whose value on an integral class $\alpha$ is the infimum of $-\chi(S)/2n$ over all $n>0$ and all aspherical surfaces $S$ representing $n\alpha$. This is called the Gromov--Thurston norm. If $X$ is a $3$-manifold the unit ball is a rational polyhedron \cite{Thurston norm}.

There is a relative Gromov norm called \emph{stable commutator length} (hereafter scl), which measures the minimal complexity of a surface with prescribed boundary. More precisely, scl of a null-homologous loop $\gamma$ is the infimum of $-\chi(S)/2n$ over all $n>0$ and all surfaces $S$ that contains no sphere or disk components and whose boundary represents $n\gamma$. It only depends on $\pi_1$, and can therefore be thought of as a function on groups, where it can be given a purely algebraic definition. Unfortunately exact calculations are very hard and it is not known how to compute scl even in a closed surface group.

Calegari \cite{Cal:freegroup PQL} gave an algorithm to compute scl in free groups. The nature of the algorithm shows the unit ball is a rational polyhedron, and Calegari asked if the same might be true for wider classes of groups, eg hyperbolic groups or groups acting nicely on trees. In this paper, we give a positive answer for groups acting on trees when edge and vertex stabilizers are cyclic. This is a large and interesting class of groups, including eg the Baumslag--Solitar groups (some authors prefer to use $t^{-1}$ as the generator)
$$\BS(M,L)\defeq\left<a,t\ |\ a^M=ta^Lt^{-1}\right>.$$

The method of proof is very geometric, and our techniques give us a way of understanding all (relative) maps of surface pairs $(S,\bdry S)$ to $2$-complexes presenting such groups. In the case of free groups this is analogous to eg the normal forms obtained by Wicks \cite{Wicks} and Culler \cite{Culler}. 

There is a tension here. The set of all homotopy classes of maps of all surfaces to a fixed $2$-complex has a lot of information but very little structure. More structure arises when we allow more operations: cut-and-paste, compression, and passing to finite covers. The last operation is rather analogous to working over $\Q$ instead of $\Z$, and indeed mirrors the relative complexity of linear and integer programming respectively.

A group acting on a tree is encoded by a graph of groups. Our groups are those whose edge and vertex groups are cyclic. For a fixed graph one obtains families of groups, by varying the edge homomorphisms. This gives rise to a linear structure on this set, and it makes sense to study how scl varies with respect to these parameters.

For certain families---so called surgery families---we are able to show that scl varies \emph{quasirationally}, ie as a ratio of quasipolynomials. Analogous phenomena arise in Ehrhart theory \cite{Ehrhart} and in parametric families of integer programming problems \cite{CW:Inthull}.


\subsection{Statement of results}
\subsubsection{Rationality}
In this paper we develop a method to compute scl in many graphs of groups using linear programming. The nature of the algorithm shows that scl is piecewise rational linear on chains (See Subsection \ref{subsec: scl} for definitions).
\newtheorem*{thm: rational}{Theorem \ref{thm: rational} (rationality)}

The main result of the paper is:
\begin{thm: rational}
	Let $G$ be a graph of groups with vertex groups $\{G_v\}$ and edge groups $\{G_e\}$ where
	\begin{enumerate}
		\item $\scl_{G_v}\equiv0$; and \label{item: ratthm cond 1}
		\item the images of the edge groups in each vertex group are \emph{central} and \emph{mutually commensurable}.
	\end{enumerate}
	Then $\scl_G$ is piecewise rational linear, and $\scl_G(c)$ can be computed via linear programming for each rational chain $c\in B_1^H(G)$.
\end{thm: rational}

The theorem above applies to many interesting groups. Bullet (\ref{item: ratthm cond 1}) holds true for amenable groups \cite[Theorem 2.47]{Cal:sclbook}, irreducible lattices in higher rank Lie groups \cite{BM1,BM2} (see also \cite[Theorem 5.26]{Cal:sclbook}), and certain transformation groups like $\rm{Homeo}^+(S^1)$ \cite[Proposition 5.11]{Ghys} and subgroups of $\rm{PL}^+(I)$ \cite[Theorem A]{Cal:sclinPL}. In particular, Theorem \ref{thm: rational} applies to all graphs of groups with vertex and edge groups isomorphic to $\Z$ (also known as generalized Baumslag--Solitar groups). See Example \ref{ex: groups} for more groups covered by this theorem.

A map of a surface group to a graph of groups is represented geometrically by a map of a surface to a graph of spaces. The surface can be cut into pieces along curves mapping to the edge spaces. If we understand the edge groups we can give conditions under which such pieces may be reassembled. When edge groups are infinite, such gluing conditions depend a priori on an infinite amount of data which we refer to as ``winding numbers''.

The key to our method is to keep track of a \emph{sufficient} but \emph{finite} amount of information about winding numbers. What makes this approach possible is a method to solve gluing conditions \emph{asymptotically}. This is a geometric covering space technique and depends on residual properties of the fundamental group of the underlying graph. It also relies on an elementary but crucial observation about \emph{stability} of \emph{virtual isomorphisms} of abelian groups (see Subsection \ref{subsec: stability}).

\subsubsection{Extremal surfaces}
\begin{definition}
	If $-\chi(S)/2n$ achieves the infimum in the definition of scl, we say $S$ is \emph{extremal}.
\end{definition}

Extremal surfaces are $\pi_1$-injective. Techniques developed to find them have lead to significant progress on Gromov's question of finding surface subgroups in hyperbolic groups \cite{Cal:sshom,CW: RandSurf,CW:sslp,Wilton}. A necessary condition for an extremal surface to exist is that $\scl\in\Q$. Thus a natural question in view of Theorem \ref{thm: rational} is whether extremal surfaces exist for rational chains. 

The answer is negative in general, but we have an algorithmic criterion (Theorem \ref{thm: extremal surfaces}) in the special case of Baumslag--Solitar groups. For any \emph{reduced} rational chain $c$, the criterion for the existence of extremal surfaces is expressed in terms of branched surfaces canonically built from the result of the linear programming problem computing $\scl_{\BS(M,L)}(c)$. Here a rational chain $c=\sum r_i g_i$ with $r_i\in \Q_{>0}$ is \emph{reduced} if no positive powers of $g_i$ and $g_j$ ``almost'' cobound an annulus (See Definition \ref{def: reduced chains} and Lemma \ref{lemma: reduced chains}). 

To describe a special case, let $h:\BS(M,L)\to \Z$ be the homomorphism with $h(a)=0$ and $h(t)=1$ in terms of the standard presentation. Its kernel $\ker h=\langle a_k,k\in \Z\ |\ a_k^M=a_{k+1}^L\rangle$ sits inside $\BS(M,L)$ via $a_k\mapsto t^k a t^{-k}$.
\newtheorem*{cor: extremal surfaces}{Corollary \ref{cor: extremal surfaces} (extremal surfaces)}
\begin{cor: extremal surfaces}
	Let $M\neq\pm L$. For a reduced rational chain $c=\sum r_i g_i$, an extremal surface exists if each $g_i$ either has $h(g_i)\neq0$ or is of the form $a_{k_1}^{u_1}\ldots a_{k_n}^{u_n}\in \ker h$ with 
	$$\sum_j u_j\left(\frac{M}{L}\right)^{k_j}=0.$$
\end{cor: extremal surfaces}

\subsubsection{Comparison theorems}
There is a striking relation between scl in $\BS(M,L)$ and in $\Z/M\Z*\Z/L\Z$. A word $w\in\BS(M,L) =\left<a,t\ |\ a^M=ta^Lt^{-1}\right>$ is \emph{$t$-alternating} if it can be written as $w=a^{u_1}ta^{v_1}t^{-1}a^{u_2}ta^{v_2}t^{-1}\cdots a^{u_n}ta^{v_n}t^{-1}$, where the generator $t$ alternates between $t$ and $t^{-1}$.
 
\newtheorem*{cor: t-alt}{Corollary \ref{cor: t-alt}}
\begin{cor: t-alt}
	Let $\Z/M\Z*\Z/L\Z = \left<x,y\ |\ x^M=y^L=1\right>$. For any $t$-alternating word in $\BS(M,L)$, we have
	$$\scl_{\BS(M,L)}(a^{u_1}ta^{v_1}t^{-1}a^{u_2}ta^{v_2}t^{-1}\cdots a^{u_n}ta^{v_n}t^{-1})=\scl_{\Z/M\Z*\Z/L\Z}(x^{u_1}y^{v_1}\cdots x^{u_n}y^{v_n}).$$
\end{cor: t-alt}

For any group $G$, the \emph{scl spectrum} is the image of $\scl_G$ in $\R$. Corollary \ref{cor: t-alt} implies that the spectrum of $\BS(M,L)$ contains the spectrum of $\Z/M\Z*\Z/L\Z$. Proposition \ref{prop: t-alt} is the more general statement with $\Z$ replaced by an arbitrary abelian group. This is a special case of the isometric embedding theorems we prove in Section \ref{sec: surf in gog}.

A different kind of relationship is expressed in the next theorem. The family of Baumslag--Solitar groups $\BS(M,L)$ with the obvious generators converges to the free group $F_2$ as marked groups \cite{Marked groups} when $M,L\to \infty$. This convergence is reflected in the behavior of scl.

\newtheorem*{thm: surgery}{Theorem \ref{thm: surgery} (convergence)}
\begin{thm: surgery}
	For any chain $c\in B_1^H(F_2)$, let $\bar{c}$ be its image in $\BS(M,L)$. If $\gcd(M,L)\to \infty$, then $\scl_{\BS(M,L)}(\bar{c})$ converges to $\scl_{F_2}(c)$. Moreover, for $M,L$ fixed, the sequence $\scl_{\BS(dM,dL)}(\bar{c})$ is eventually quasirational in $d$ and converges to $\scl_{F_2}(c)$.
\end{thm: surgery}
This resembles the geometric convergence witnessed in hyperbolic Dehn surgery. However in Theorem \ref{thm: surgery}, it is important for the parameters $(M,L)$ to go to infinity in a specific way. An example (Proposition \ref{prop: eg1}) shows that such an assumption is necessary, and scl is not a continuous function on the space of marked groups.

\subsection{Comparison with previous results}
Here is a list of groups where scl is previously shown to be piecewise rational linear.
\begin{enumerate}
	\item Free groups, by Calegari \cite{Cal:freegroup PQL}.\label{item: prev results 1}
	\item Free products of cyclic groups, by Walker \cite{Wal:scylla}.\label{item: prev results 2}
	\item Free products of free abelian groups, by Calegari \cite{Cal:sss}.\label{item: prev results 3}
	\item Free products $*_\lambda G_\lambda$ with $\scl_{G_\lambda}\equiv 0$ for all $\lambda$, by the author \cite{Chen:sclfp}.\label{item: prev results 4}
	\item Amalgams of free abelian groups, by Susse \cite{Susse}.\label{item: prev results 5}
	\item $t$-alternating words in Baumslag--Solitar groups, by Clay--Forester--Louwsma \cite{CFL}.\label{item: prev results 6}
\end{enumerate}

Theorem \ref{thm: rational} is a generalization of all the rationality results above. Corollary \ref{cor: t-alt} provides an easier way to understand and compute scl of $t$-alternating words in Baumslag--Solitar groups.

Regarding extremal surfaces, they exist for any rational chains in bullet (\ref{item: prev results 1}) and (\ref{item: prev results 2}). In all other results, extremal surfaces do not exist in general. In bullet (\ref{item: prev results 6}), however, a criterion \cite[Theorem 5.7]{CFL} is provided for $t$-alternating words that bound extremal surfaces. Our criterion (Theorem \ref{thm: extremal surfaces}) extends this to general words and rational chains.

A few previous results are similar to the convergence Theorem \ref{thm: surgery}, expressing scl in a group $G$ as the limit of scl in a certain family of quotients of $G$.
\begin{enumerate}
	\item \cite[Theorem 4.13]{Cal:sss} deals with $G=A*B\to A'*B$ where $A,A',B$ are free abelian with $\rm{rank}(A)=\rm{rank}(A')+1$ and the maps are induced by a linear family of surjective homomorphisms $A\to A'$ and $id:B\to B$;
	\item \cite[Corollary 4.12]{Susse} handles the case $G=F_2=\left<x,y\right>\to \Z*_{\Z}\Z$ by adding the relation $x^M=y^L$ with $M,L\to \infty$; and
	\item \cite[Theorem 6.4]{Chen:sclfp} proves the case $G=F_2=\left<x,y\right>\to \Z/M\Z*\Z/L\Z$ by adding the relation $x^M=y^L=1$ with $M,L\to \infty$.
\end{enumerate}
The convergences in these results are all quasirational in the parameters. Note that how $M,L$ go to infinity matters in Theorem \ref{thm: surgery} but not in the results above. These results together suggest a more general phenomenon to be discovered.

\subsection{Organization of the paper}
To ease into our discussion, we introduce the notion of \emph{relative stable commutator length} and develop its basic properties in Section \ref{sec: background}, where we also review some elements of scl and graphs of groups. In Section \ref{sec: surf in gog} we study surfaces in graphs of groups and their (simple) normal forms, from which we obtain two isometric embedding theorems. In Section \ref{sec: asym prom} we define disk-like pieces. These hold a finite amount of information about winding numbers, which turns out to be sufficient to solve the gluing conditions asymptotically. Finally in Section \ref{sec: sclBS}, we focus on the case of Baumslag--Solitar groups, investigating when the asymptotic realization can terminate at a finite stage resulting in extremal surfaces. Then we establish explicit formulas, prove the convergence theorem, and give an implementation with low time complexity.

\subsection{Acknowledgment}
I would like to thank Danny Calegari for his consistent encouragements and guidance. I also thank Matt Clay, Max Forester, Joel Louwsma and Tim Susse for helpful conversations on their related studies.
Finally I would like to thank Benson Farb, Martin Kassabov, Jason Manning, and Alden Walker for useful discussions, and thank the anonymous referee for good suggestions improving the paper.

\section{Background}\label{sec: background}
\subsection{Stable commutator length}\label{subsec: scl}
We review some basics of scl (stable commutator length) and set up notation. For a group $G$, let $C_1(G)$ be the $1$-chains in $G$, namely the $\R$-vector space with basis $G$. Taking group homology is a linear map $h_G:C_1(G)\to H_1(G;\R)$ whose kernel contains $H(G)$, the subspace of $C_1(G)$ spanned by elements of forms $g^n-ng$ or $hgh^{-1}-g$ for some $g,h\in G$ and $n\in\Z$. Thus, on the quotient $C_1^H(G):=C_1(G)/H(G)$, we have a well-defined linear map $\bar{h}_G:C_1^H(G)\to H_1(G;\R)$. Denote the kernel of $\bar{h}_G$ by $B_1^H(G)$.

\begin{definition}
	Let $X$ be a $K(G,1)$ space. For an integral chain $c=\sum_{i\in I} g_i$, let $\gamma_i$ be a loop in $X$ representing $g_i$. An admissible surface for $c$ is a map $f:(S,\bdry S)\to (X,\sqcup \gamma_i)$ from a compact oriented surface $S$, such that the following diagram commutes and $\bdry f_*[\bdry S]=n(S)[\sqcup S^1]$.
	\[
	\begin{CD}
	\bdry S @>{i}>> S\\
	@V{\bdry f}VV @V{f}VV\\
	\sqcup S^1 @>{\sqcup \gamma_j}>> X
	\end{CD}
	\]
\end{definition}

For a rational chain $c$, a certain multiple $mc$ with $m\in\Z_+$ is integral. A surface $S$ is admissible for $c$ of degree $n$ if it is admissible for some $mc$ of degree $k$ such that $n=mk$.

Admissible surfaces exist when the rational chain $c$ is null-homologous, ie $c\in B_1^H(G)$. We measure the complexity of a surface $S$ by $-\chi^{-}(S)$, where $\chi^-(\Sigma)=\min(0,\chi(\Sigma))$ for each connected component $\Sigma$, and $\chi^-(S)$ is the sum of $\chi^-(\Sigma)$ over all its components. Equivalently, $\chi^{-}(S)$ is the Euler characteristic of $S$ neglecting sphere and disk components.

\begin{definition}
	The \emph{stable commutator length} of a null-homologous rational chain $c$, denoted $\scl_G(c)$, is defined as
	$$\scl_G(c)\defeq\inf_{S} \frac{-\chi^{-}(S)}{2n(S)},$$
	where the infimum is taken over all admissible surfaces $S$ of degree $n(S)\ge 1$.
\end{definition}

This agrees with the algebraic definition of scl as the limit of commutator lengths. See \cite[Chapter 2]{Cal:sclbook}.

Throughout this paper, we only consider admissible surfaces where each boundary component is an orientation preserving covering map of some loop in the chain. This does not affect the computation of scl \cite[Proposition 2.13]{Cal:sclbook}.

It immediately follows by pushing forward admissible surfaces that scl is \emph{monotone}: For any homomorphism $\phi: G\to H$, we have $\scl_G(c)\ge \scl_H(\phi(c))$ for any chain $c$. The equality holds if $\phi$ is an isomorphism, or more generally, if $\phi$ admits a retract. In particular, scl is \emph{invariant under conjugation}.

Scl extends continuously in a unique way to all null-homologous real chains and induces a pseudo-norm on $B_1^H(G)$ \cite[Chapter 2]{Cal:sclbook}. As a consequence, removing finite order elements from the chain $c$ does not affect $\scl_G(c)$. We make the convention that $\scl_G(c)=+\infty$ if $c$ has non-trivial homology. 

We say $\scl_G$ is \emph{piecewise rational linear} if it is a piecewise rational linear function on each rational finite-dimensional subspace of $B_1^H(G)$.


\begin{definition}
	We say a homomorphism $\phi: G\to H$ is an \emph{isometric embedding}, if $\phi$ is injective and
	$$\scl_G(c)=\scl_H(\phi(c)),$$
	for all $c\in B_1^H(G)$.
\end{definition}

\begin{example}\label{ex: isom emb}
	Here are some occasions where we know an injection $\phi: G\to H$ is an isometric embedding.
	\begin{enumerate}
		\item $\phi$ admits a retract $r: H\to G$, that is $r\circ \phi=id$. This follows from monotonicity.
		\item $G$ is abelian. In this case, we have $B_1^H(G)=0$.
		\item $\scl_G\equiv0$.
	\end{enumerate}
\end{example}

Isometric embeddings allow us to pull back admissible surfaces at arbitrarily small cost.
\begin{lemma}\label{lemma: pull back surf}
	Let $\phi: G\to H$ be an isometric embedding, realized by a map $\phi: X_G\to X_H$ between $K(G,1)$ and $K(H,1)$ spaces. Suppose $f':S'\to X_H$ is a surface without sphere components in $X_H$ such that $\bdry f'\bdry S'=\phi(\uga)$ for a collection of loops $\uga$ in $X_G$ whose sum is a null-homologous chain $c$. Then for any $\epsilon>0$, there is a surface $f:S\to X_G$ satisfying the following properties:
	\begin{enumerate}
		\item $S$ has no sphere components;
		\item There is a (disconnected) covering map $\pi: \bdry S\to \bdry S'$ of a certain degree $n>0$ such that the following diagram commutes;
		\[
		\begin{CD}
		\bdry S @>{\pi}>> \bdry S'\\
		@V{\bdry f}VV @V{\bdry f'}VV\\
		X_G @>{\phi}>> X_H
		\end{CD}
		\]
		\item We have $$\frac{-\chi(S)}{n}\le -\chi(S')+\epsilon.$$		
	\end{enumerate}
\end{lemma}
\begin{proof}
	Let $D'$ be the union of disk components in $S'$. Since $\phi$ is an isometric embedding and thus $\pi_1$-injective, the loops in $\uga$ corresponding to $\bdry D'$ bound a collection of disks $D$ in $X_G$ accordingly. Let $S_0'$ be the remaining components of $S'$. Then $\bdry S_0'$ represents a null-homologous chain $c_0$ equivalent to $c$ in $B_1^H(G)$ and $\chi(S'_0)=\chi^{-}(S'_0)$. Since $\phi$ preserves scl, there is an admissible surface $S_0$ for $c_0$ of a certain degree $n>0$ without disk or sphere components such that
	$$\frac{-\chi(S_0)}{n}\le 2\scl_G(c_0)+\epsilon=2\scl_H(\phi(c_0))+\epsilon\le -\chi(S_0')+\epsilon.$$
	Then the surface $S=S_0\sqcup nD$ has the desired properties.
\end{proof}


\subsection{Relative stable commutator length}\label{subsec: relscl}
It is natural and convenient for our discussion to introduce \emph{relative stable commutator length}. Several previous work \cite{CF,CFL,Chen:sclfp,IK} contains similar thoughts without formally formulating this notion.
\begin{definition}
	Let $\mathcal{G}=\{G_\lambda\}_{\lambda\in\Lambda}$ be a collection of subgroups of a group $G$. For a chain $c\in C_1(G)$, its \emph{stable commutator length relative to} $\mathcal{G}$ is 
	$$\scl_{(G,\mathcal{G})}(c)\defeq\inf\left\{ \scl\left(c+\sum c_\lambda\right): c_\lambda \in C_1(G_{\lambda})\right\},$$
	where each summation contains only finitely many non-zero $c_\lambda$.
\end{definition}

Recall that we have a linear map $\bar{h}_G:C_1^H(G)\to H_1(G;\R)$ whose kernel is denoted as $B_1^H(G)$ in Subsection \ref{subsec: scl}. Let $H_1(\mathcal{G})\le H_1(G)$ be the image of the inclusion $\oplus_{\lambda}H_1(G_\lambda)\to H_1(G)$. Denote $\bar{h}_G^{-1}H_1(\mathcal{G})$ by $B_1^H(G,\mathcal{G})$, which contains $B_1^H(G)$ as a subspace.

We summarize some basic properties of relative scl in the following lemma.
\begin{lemma}
	With notation as above, we have:
	\begin{enumerate}
		\item A chain $c$ has finite $\scl_{(G,\mathcal{G})}(c)$ if and only if its homology class $[c]$ lies in $H_1(\mathcal{G})$.
		\item $\scl_{(G,\mathcal{G})}$ is a well defined pseudo norm on $B_1^H(G,\mathcal{G})$.
		\item $\scl_{G}(c)\ge \scl_{(G,\mathcal{G})}(c)$.
		\item For any homomorphism $\phi: G\to H$, we have 
		$$\scl_{(G,\mathcal{G})}(c)\ge\scl_{(H,\phi(\mathcal{G}))}(c).$$ If $\phi$ is an isomorphism, then equality holds.
		\item If $g$ conjugates into some $G_\lambda$, then $\scl_{(G,\mathcal{G})}(g)=0$.
	\end{enumerate}
\end{lemma}

When $c$ is a rational chain, its relative scl can be described in terms of \emph{relative admissible surfaces} as follows. 

\begin{definition}
	Let $X$ be a topological space with subspaces $X_\lambda$ such that $\pi_1(X)\cong G$ and under this isomorphism $\pi_1(X_\lambda)$ represents the conjugacy class of $G_\lambda$. Then a \emph{relative admissible surface} for a rational chain $c$ of \emph{degree} $n>0$ is a continuous map $f:S\to X$ from a compact oriented surface $S$ with a specified collection of boundary components $\bdry_0\subset\bdry S$ such that $f(\bdry_0)$ represents $nc$ and $f(C)\subset X_\lambda$ for some $\lambda=\lambda(C)$ for each boundary component $C\subset \bdry S$ outside of $\bdry_0$.
\end{definition}

\begin{prop}\label{prop: relative admissible surf}
	For any rational chain $c\in C_1(G)$, we have 
	$$\scl_{(G,\mathcal{G})}(c)=\inf \frac{-\chi^-(S)}{2n(S)},$$ where the infimum is taken over all relative admissible surfaces for $c$.
\end{prop}
\begin{proof}
	On the one hand, any relative admissible surface $S$ for $c$ of degree $n$ is admissible for a chain $c+\sum c_\lambda$ of degree $n$, where $n\sum c_\lambda$ is the chain represented by the boundary components of $S$ outside of the specified components $\bdry_0$. This proves the ``$\le$'' direction.
	
	On the other hand, consider a chain $c+\sum c_\lambda$ with all $c_\lambda$ rational. Any admissible surface $S$ for $c+\sum c_\lambda$ of degree $n$ is a relative admissible surface for $c$ of degree $n$ by taking the boundary components representing $nc$ to be the specified components $\bdry_0$. Thus for such a rational chain, $\scl(c+\sum c_\lambda)$ is no less than the right-hand side of the desired equality. Then the ``$\ge$'' direction follows by continuity. 
\end{proof}

With appropriate homology conditions, we have $\scl_{(G,\mathcal{G})}(c)=\scl_G(c)$ for $c\in C_1(G)$ if $\scl_G$ vanishes on each $H\in \mathcal{G}$. In this case, one can compute $\scl_G$ using relative admissible surfaces, which has the advantage of not closing up additional boundary components representing chains with trivial scl. Proposition \ref{prop: rel vertex} handles the case we need for graphs of groups.

A relative admissible surface is called \emph{extremal} if it obtains the infimum in Proposition \ref{prop: relative admissible surf}. For scl in the absolute sense, extremal surfaces are $\pi_1$-injective \cite[Proposition 2.104]{Cal:sclbook}. This generalizes to the relative case with the same proof using the fact that free groups are LERF.

\begin{prop}
	If $f: S\to X$ is an extremal relative admissible surface for a rational chain $c\in B_1^H(G,\mathcal{G})$, then $f_*:\pi_1(S)\to \pi_1(X)$ is injective.
\end{prop}



\subsection{Graphs of groups}\label{subsec: gog}
Throughout this paper, edges on graphs are oriented. 
In this subsection, we consider graphs $\Gamma=(V,E)$ where $E$ includes both orientations of edges. So we have an involution $e\mapsto \bar{e}$ on $E$ without fixed point by reversing edge orientations, and maps $o,t:E\to V$ taking origin and terminus vertices of edges respectively, such that $t(e)=o(\bar{e})$.

Suppose we have a graph $\Gamma=(V,E)$ and two collections of groups $\{G_v\}_{v\in V}$ and $\{G_e\}_{e\in E}$ indexed by vertices and edges such that $G_e=G_{\bar{e}}$. Suppose we also have injections $o_e:G_e\to G_{o(e)}$ and $ t_e: G_e\to G_{t(e)}$ for each oriented edge $e$ satisfying $o_{\bar{e}}=t_e$. Let $X_v$ and $X_e=X_{\bar{e}}$ be $K(G_v,1)$ and $K(G_e,1)$ spaces with base points $b_v$ and $b_e=b_{\bar{e}}$ respectively. For each edge $e$, the injections $o_e$ and $t_e$ determine (up to homotopy) maps $o_{e}: (X_e,b_e)\to (X_{o(e)},b_{o(e)})$ and $t_{e}: (X_e,b_e)\to (X_{t(e)},b_{t(e)})$ respectively. Let $X$ be the space obtained from the disjoint union $(\sqcup_{v\in V} X_v)\sqcup (\sqcup_{e\in E} X_e\times[-1,1])$ by identifying $X_e\times \{t\}$ with $X_{\bar{e}\times \{-t\}}$ and gluing $X_e\times\{-1\}$ and $X_e\times\{1\}$ to $X_{o(e)}$ and $X_{t(e)}$ via $o_{e}$ and $t_{e}$, respectively.

We call $X$ the \emph{graph of spaces} associated to the given data. Identify $X_v$ with its image in $X$, referred to as the \emph{vertex space}. For each vertex $v$, the image of $X_v\cup (\cup_{t(e)=v} X_e\times  [0,1))$ is homotopic to $X_v$, so its completion is too; we refer to both as the \emph{thickened vertex space} $N(X_v)$. Identify $X_e$ with the image of $X_e\times \{0\}$, called the \emph{edge space}.

When $\Gamma$ is connected, we call $G=\pi_1(X)$ the (fundamental group of) \emph{graph of groups} and $X$ the \emph{standard realization} of $G$. We use the notation $G=\mathcal{G}(\Gamma, \{G_v\}, \{G_e\})$ to specify the underlying data. It is a fact that $G_v$ and $G_e$ sit inside $G$ as subgroups via the inclusions, referred to as \emph{vertex groups} and \emph{edge groups}. Graphs of groups are natural generalizations of amalgams and HNN extensions. For more details, especially their actions on trees, see \cite{Serre}.

We say an element $g\in G$ and its conjugacy class are \emph{elliptic} if $g$ conjugates into some vertex group, otherwise they are \emph{hyperbolic}. Geometrically, $g$ is elliptic if and only if it is represented by a loop supported in some vertex space.

Let $\gamma$ be a loop in $X$ representing the conjugacy class of an element $g\in G$. We can homotope $\gamma$ so that it is either disjoint from $X_e\times\{t\}$ or intersects it only at $b_e\times\{t\}$ transversely, for all $t\in(-1,1)$ and any edge $e$. Then the edge spaces cut $\gamma$ into finitely many arcs, unless $\gamma$ is supported in a vertex space. 

Each arc $a$ is supported in some thickened vertex space $N(X_v)$ and decomposes into three parts (see Figure \ref{fig: simplify backtrack}): an arc parameterizing $b_e\times[0,1]$, a based loop in $X_v$, and an arc parameterizing $b_{e'}\times[-1,0]$, where $t(e)=v=o(e')$. We refer to the element $w(a)\in G_v$ represented by the based loop as \emph{the winding number} of $a$, and denote $e,e'$ by $e_{in}(a),e_{out}(a)$ respectively. We say $\gamma$ \emph{trivially backtracks} if for some arc $a$ as above $\overline{e_{out}(a)}=e_{in}(a)$ and $w(\alpha)$ lies in $t_e(G_e)$. In this case, $\gamma$ can be simplified by a homotopy reducing the number of arcs. See Figure \ref{fig: simplify backtrack}. After finitely many simplifications, the loop $\gamma$ does not trivially backtrack, which we call a \emph{tight} loop. 

\begin{figure} 
	\centering
	
	\subfloat[]{
		\labellist
		\small \hair 2pt
		\pinlabel $X_v$ at 120 190
		\pinlabel \textcolor{red}{$a\subset\gamma$} at 188 95
		\pinlabel $X_e\times[0,1]$ at 188 15
		\endlabellist
		\includegraphics[scale=0.6]{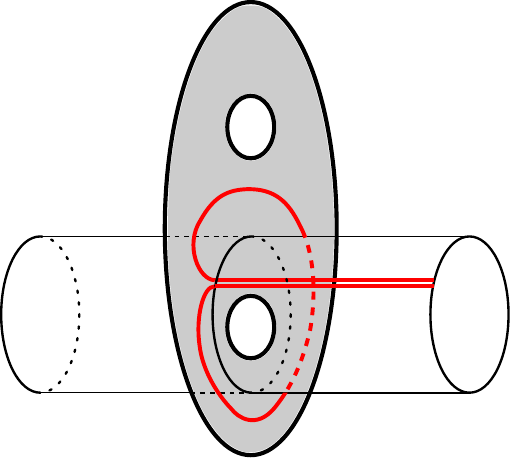}
		\label{fig: before pull}
	}
	\hfill
	\subfloat[]{
		\labellist
		\small \hair 2pt
		\pinlabel $X_v$ at 120 190
		\pinlabel \textcolor{red}{$a\subset\gamma$} at 195 120
		\pinlabel $X_e\times[0,1]$ at 188 15
		\endlabellist
		\includegraphics[scale=0.6]{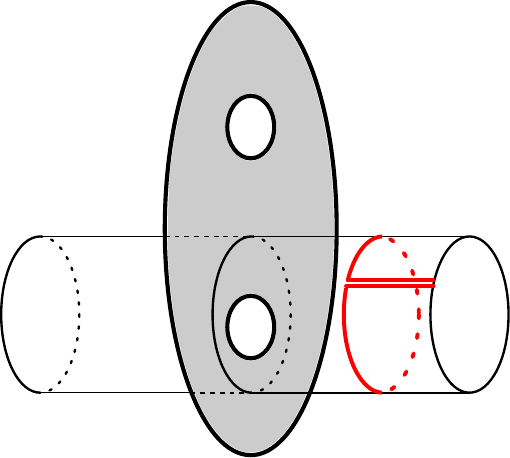}
		\label{fig: after pull}
	}

	\caption{A loop $\gamma$ trivially backtracks at an arc $a$ supported in the thickened vertex space $N(X_v)$ as in (\ref{fig: before pull}). It can be pushed off the vertex space $X_v$ by a homotopy as in (\ref{fig: after pull}).}
	\label{fig: simplify backtrack}
\end{figure}
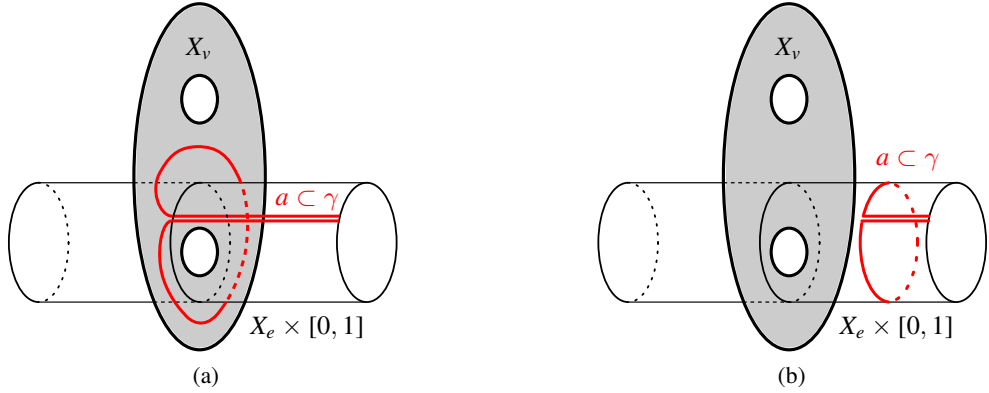

In the case where $g$ is elliptic, $\gamma$ is tight if and only if it is supported in some vertex space. Moreover, instead of a collection of arcs, we have a single loop $\gamma$, whose winding number $w(\gamma)$ is only well-defined up to conjugacy.

The following proposition allows us to use relative admissible surfaces to compute scl.

\begin{prop}\label{prop: rel vertex}
	Let $G=\mathcal{G}(\Gamma, \{G_v\}, \{G_e\})$ be a a graph of groups with underlying graph $\Gamma=(V,E)$. Suppose $\scl_{G}(c)=0$ for any $c\in B_1^H(G_v)$ and vertex $v$. Then for any null-homologous chain $c\in B_1^H(G)$, we have
	$$\scl_{(G,\{G_v\})}(c)=\scl_G(c).$$
\end{prop}
\begin{proof}
	A simple calculation shows
	\begin{eqnarray}
		H_1(G;\R)&\cong& H_1(\Gamma;\R)\nonumber\\
		&\oplus&
		\left(\bigoplus_v H_1(G_v;\R)\right)\left/\left<o_{e*}(c_e)-t_{e*}(c_e):c_e\in H_1(G_e;\R)\right>\right.  .\label{eqn: homology of G}
	\end{eqnarray}
	Consider any null-homologous chain $c+\sum c_i$ where $g_ic_ig_i^{-1}\in C_1(G_{v_i})$ for some $g_i\in G$ and vertex $v_i$. It follows that $\sum c_i$ is null-homologous since $c$ is. Thus by the homology calculation (\ref{eqn: homology of G}), there exist chains $c_e=-c_{\bar{e}}\in C_1(G_e)$ such that, for each vertex $v$, the chain
	$$c_v\defeq\sum_{i: v_i=v} c_i+\sum_{e: t(e)=v} t_{e}(c_e)\quad \in C_1(G_v)$$
	is null-homologous in $H_1(G_v;\R)$. Clearly, for each $e$ the chains $t_{\bar{e}}(c_{\bar{e}})=o_{e}(-c_e)$ and $-t_{e}(c_e)$ are equivalent to each other as they are geometrically homotopic. Hence $c+\sum c_i=c+\sum_{v\in V} c_v\in C_1^H(G)$, and
	$$\scl_G\left(c+\sum c_i\right)=\scl_G\left(c+\sum c_v\right)=\scl_G(c),$$
	where the last equality holds since $\scl_G(c_v)=0$ for all $v$ by assumption. Then the result follows from the definition of relative scl.
\end{proof}

In the sequel, all relative admissible surfaces will be understood to be relative to vertex groups unless stated otherwise.

\section{Surfaces in graphs of groups}\label{sec: surf in gog}
Throughout this section, let $G=\mathcal{G}(\Gamma, \{G_v\}, \{G_e\})$ be a graph of groups and let $X$ be its standard realization as in Subsection \ref{subsec: gog}. Let $\ug=\{g_1,\ldots,g_m\}$ be a set of infinite order elements in $G$ represented by tight loops $\uga=\{\gamma_1,\ldots,\gamma_m\}$ in $X$. 

Recall that edge spaces cut hyperbolic tight loops into arcs, while each elliptic tight loop is already supported in some vertex space. For each vertex $v$, denote by $A_v$ the collection of arcs supported in $N(X_v)$ obtained by cutting hyperbolic loops in $\uga$. Arcs from different loops or different parts of the same loop are considered as distinct elements. Let $L_v\subset\uga$ be the set of elliptic loops supported in $X_v$.

For each $\alpha\in A_v\cup L_v$, let $i(\alpha)$ be the index such that $\alpha$ sits on $\gamma_{i(\alpha)}$. Denote by $[w(\alpha)]$ the homology class of the winding number $w(\alpha)$ in $H_1(G_v;\R)$.

\begin{lemma}\label{lemma: homological computation}
	Let $c=\sum_i r_ig_i$ be a null-homologous chain. Then there exists $\beta_e\in H_1(G_e;\R)$ for each edge $e$ such that $\beta_e=-\beta_{\bar{e}}$ and
	$$\sum_{\alpha\in A_v\cup L_v} r_{i(\alpha)}[w(\alpha)]=\sum_{t(e)=v} t_{e*}\beta_e$$
	holds for each vertex $v$.
\end{lemma}
\begin{proof}
	This directly follows from equation (\ref{eqn: homology of G}).
\end{proof}

\subsection{Normal form}\label{subsec: normal form}
With the above setup, fix a rational homologically trivial chain $c=\sum_i r_ig_i$ with $r_i\in\Q$. Let $f:S\to X$ be an arbitrary admissible surface for $c$ without sphere components. Put $S$ in general position so that it is transverse to all edge spaces. Then $F=f^{-1}(\cup_e X_e)$ is a proper submanifold of codimension $1$, that is, a union of embedded loops and proper arcs. Eliminate all trivial loops in $F$ (innermost first) by homotopy and then compress $S$ along a loop in $F$ whose image is trivial in $X$ if any. Since this process decreases $-\chi(S)$, all loops in $F$ are non-trivial in $X$ after finitely many repetitions. All proper arcs in $F$ are essential since loops in $\uga$ are tight.

Now cut $S$ along $F$ into surfaces with corners, each component mapped into $N(X_v)$ for some vertex $v$. Let $S_v$ be the union of components mapped into $N(X_v)$. The boundary components of $S_v$ fall into two types (See Figure \ref{fig: S_v}).
\begin{enumerate}
	\item Polygonal boundary: these are the boundary components divided by corners of $S_v$ into segments alternating between proper arcs in $F$ and arcs in $A_v$.
	\item Loop boundary: these are the components disjoint from corners of $S_v$, and thus each is a loop in either $F$ or $L_v$.
\end{enumerate}

\begin{figure}
	\labellist
	\small \hair 2pt
	\pinlabel $\beta_1$ at 120 250
	\pinlabel $\beta_2$ at 180 120
	\pinlabel $\beta_3$ at 165 25
	\endlabellist
	\centering
	\includegraphics[scale=0.5]{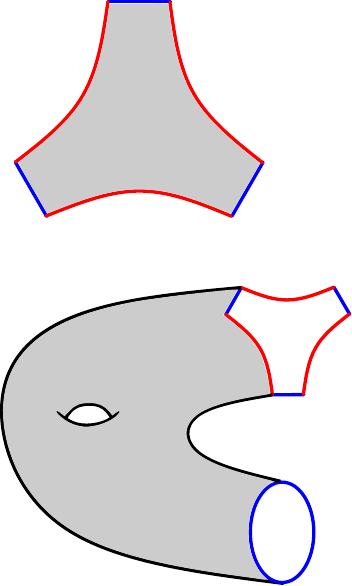}
	\caption{An example of $S_v$ consisting of two components and three boundary components. On the boundary, the blue parts are supported in edge spaces and the red in the interior of $N(X_v)$. The boundaries $\beta_1$ and $\beta_2$ are polygonal, with arcs curly and turns straight. The boundary $\beta_3$ is a loop boundary.}\label{fig: S_v}
\end{figure}

Note that any disk component in $S_v$ must bound a polygonal boundary since loop boundaries are non-trivial in $X$ by the simplification process above and the assumption that $\ug$ consists of infinite order elements.

Let $\alpha$ be any proper arc in $F$ that appears on a polygonal boundary of $S_v$. If $\alpha$ with the orientation induced from $S_v$ starts from an arc $a_v\in A_v$ and ends in $a'_v\in A_v$, we say $\alpha$ is a \emph{turn} from $a_v$ to $a'_v$. With the induced orientation, $\alpha$ represents an element $w\in G_e$, called the \emph{winding number} of this turn, where $e=e_{out}(a_v)=\overline{e_{in}(a'_v)}$. The triple $(a_v,w,a'_v)$ determines the \emph{type} of the turn $\alpha$. There are possibly many turns in $S_v$ of the same type.

Suppose $u=t(e)$ is the other end of the edge $e$ above. Then $\alpha$, viewed from $S_u$, gives rise to a turn in $S_u$ from $a_u$ to $a'_u$, where $a_u,a'_u\in A_u$ are arcs such that $a_v$ is followed by $a'_u$ and $a_u$ is followed by $a'_v$ on $\uga$. See Figure \ref{fig: pair}. Since the two sides induce opposite orientations on $\alpha$, the winding numbers of the two turns are inverses. We refer to such two types of turns $(a_v,w,a'_v)$ and $(a_u,w^{-1},a'_u)$ as \emph{paired turns}. 

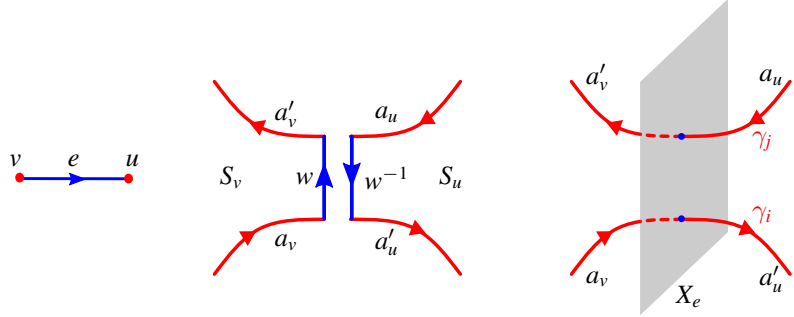
\begin{figure}
	\labellist
	\small \hair 2pt
	\pinlabel $v$ at 0 85
	\pinlabel $u$ at 64 85
	\pinlabel $e$ at 32 85
	
	\pinlabel $a_v$ at 148 40
	\pinlabel $a'_v$ at 148 110
	\pinlabel $a_u$ at 203 110
	\pinlabel $a'_u$ at 203 40
	\pinlabel $w$ at 158 75
	\pinlabel $w^{-1}$ at 203 75
	\pinlabel $S_v$ at 118 75
	\pinlabel $S_u$ at 238 75
	
	\pinlabel $a_v$ at 318 20
	\pinlabel $a'_v$ at 318 130
	\pinlabel $a_u$ at 413 130
	\pinlabel $a'_u$ at 413 20
	\pinlabel \textcolor{red}{$\gamma_j$} at 408 95
	\pinlabel \textcolor{red}{$\gamma_i$} at 408 55
	
	\pinlabel $X_e$ at 368 10
	
	\endlabellist
	\centering
	\includegraphics[scale=0.7]{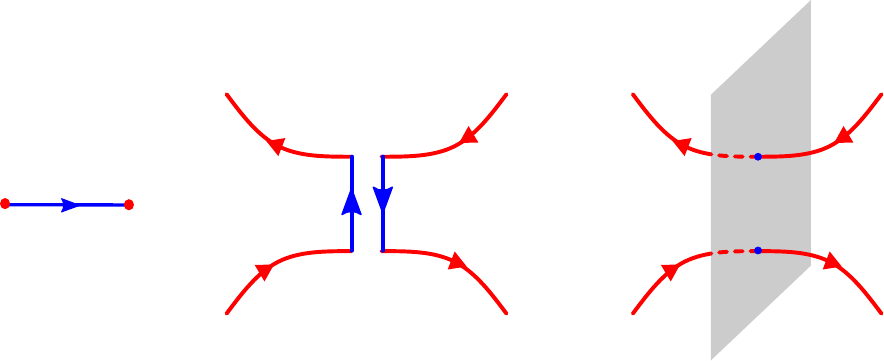}
	\caption{Two paired turns on an edge space $X_e$, with arcs $a_v,a'_u\subset\gamma_i$ and $a'_v,a_u\subset\gamma_j$, $e=e_{out}(a_v)$}\label{fig: pair}
\end{figure}

Similar analysis can be done for loops in $F$ except that they represent conjugacy classes in $G_e$ instead of elements.

Then the collection of boundaries $\{\bdry S_v\}$ satisfies the following
\begin{equation}\label{eqn: gluing condition}
	\emph{gluing condition}: \text{Turns of paired types have equal numbers of instances.}
\end{equation}
Moreover, the loop boundaries together represent a trivial chain in $B_1^H(G)$.


\begin{definition}
	Such a decomposition into subsurfaces $S_v$, one for each vertex $v$, is called the \emph{normal form} of $S$.
\end{definition}

Recall the orbifold Euler characteristic of a surface $\Sigma$ with corners is
$$\chi_o(\Sigma)\defeq\chi(\Sigma)-\frac{1}{4}\#\text{corners}.$$
Then we obtain a surface $S'$ admissible for a chain equivalent to $c$ in $B_1^H(G)$ of the same degree as $S$ by gluing up turns of paired types arbitrarily on polygonal boundaries in the normal form of $S$, and
$$-\chi(S')=-\sum \chi_o(S_v)\le-\chi(S).$$

We summarize the discussion above as the following lemma.

\begin{lemma}\label{lemma: normal form}
	Every admissible surface $S$ can be decomposed into the normal form after throwing away sphere components, compressions, homotopy and cutting along edge spaces. By gluing up paired turns arbitrarily on polygonal boundaries in the normal form of $S$, we obtain a surface $S'$ relative admissible for $c$ of the same degree as $S$ and
	$$-\chi(S')=-\sum_v \chi_o(S_v)\le-\chi(S).$$
\end{lemma}

The following theorem pieces isometric embeddings of vertex groups into a global one.
\begin{thm}[first isometric embedding]\label{thm: isomemb1}
	Let $\mathcal{G}(\Gamma, \{G_v\}, \{G_e\})$ and $\mathcal{G}(\Gamma', \{G'_v\}, \{G'_e\})$ be graphs of groups. Suppose there is a graph isomorphism $h:\Gamma\to\Gamma'$, homomorphisms $h_v:G_v\to G'_{h(v)}$ and isomorphisms $h_e:G_e\to G'_{h(e)}$ such that 
	\begin{enumerate}
		\item each $h_v$ is an isometric embedding;\label{item: isom}
		\item for each $v$, the map induced by $h_{v}$ on homology is injective on $H_v$, where $H_v$ is the sum of $\mathrm{Im}t_{e*}$ over all edges $e$ with $t(e)=v$; and\label{item: homology}
		\item the following diagram commutes for all edge pairs $(e,e')$ with $e'=h(e)$.\label{item: diagram}
		\[
		\begin{CD}
		G_{o(e)} @<{o_e}<< G_e @>{t_e}>> G_{t(e)}\\
		@V{h_{o(e)}}VV @V{h_e}V{\cong}V @V{h_{t(e)}}VV\\
		G'_{o(e')} @<{o_{e'}}<< G'_{e'} @>{t_{e'}}>> G'_{t(e')}\\
		\end{CD}
		\]
	\end{enumerate}
	Then the induced homomorphism $h:\mathcal{G}(\Gamma, \{G_v\}, \{G_e\})\to\mathcal{G}(\Gamma', \{G'_v\}, \{G'_e\})$ is an isometric embedding.
\end{thm}
\begin{proof}
	It easily follows from assumptions (\ref{item: isom}) and (\ref{item: diagram}) that $h$ is well defined and injective. Now we show $h$ preserves scl.
	
	Consider an arbitrary rational homologically trivial chain $c=\sum_i r_ig_i$ with $r_i\in\Q$ and $g_i\in G$ of infinite order. Represent each $g_i$ by a tight loop $\gamma_i$ in the standard realization $X$ of $G$. Define $\uga,A_v,L_v$ as in the discussion above.
	
	Let $X'$ be the standard realization of $G'$ and let $\eta: X\to X'$ be the map representing $h$. Then the arcs obtained by cutting $\sqcup_i \eta(\gamma_i)$ along the edge spaces of $X'$ are exactly $\sqcup_v \eta(A_v)$.
	
	Let $S'$ be an admissible surface for $c'=h(c)$ of degree $n'$. It suffices to show, for any $\epsilon>0$, there is an admissible surface $S$ for $c$ of degree $n$ such that
	$$\frac{-\chi^{-}(S)}{n}\le \frac{-\chi^{-}(S')}{n'}+C(S')\epsilon$$
	for some constant $C(S')$.
	
	By Lemma \ref{lemma: normal form}, we may assume $S'$ to be in its normal form and obtained by gluing surfaces with corner $S'_{v'}$, where $S'_{v'}$ is supported in the thickened vertex space $N(X'_{v'})$. Let $v=h^{-1}(v')$. Now we pull back $\bdry S'_{v'}$ in a natural way and show that the pull back is homologically trivial in $G_v$. First consider each polygonal boundary of $S'_{v'}$. Each arc is in $\eta(A_v)$ and pulls back via $h_v$ to a unique arc in $A_v$. Each turn is in $\mathrm{Im}t_{e'}$ for some $e'$ with $t(e')=v'$, and thus, by assumption (\ref{item: diagram}), pulls back via $h_v$ to a unique turn in $\mathrm{Im}t_{e}$, where $e=h^{-1}(e')$. Similarly consider each loop boundary of $S'_{v'}$. If such a loop lies in $\eta(L_v)$, then it uniquely pulls back to a loop in $L_v$. Otherwise, it is obtained by cutting along the preimage of edge spaces, thus uniquely pulls back to a loop represented by an element in $\mathrm{Im}t_{e}$. By construction, the homology of the pull back of $\bdry S'_{v'}$ lies in 
	$$n'\sum_{\alpha\in A_v\cup L_v} r_{i(\alpha)}[w(\alpha)]+H_v,$$
	which is exactly $H_v$ by Lemma \ref{lemma: homological computation} since $c$ is homologically trivial. Thus assumption (\ref{item: homology}) implies the pull back of $\bdry S'_{v'}$ is homologically trivial in $G_v$. By assumption (\ref{item: isom}) and Lemma \ref{lemma: pull back surf}, there is a surface $S_v$ (with corners induced by those on $S'_{v'}$) mapped into $X_v$ such that $h_v(\bdry S_v)=n_v \bdry S'_{v'}$ for some integer $n_v>0$ and $$\frac{-\chi(S_v)}{n_v}\le -\chi(S_v')+\epsilon.$$
	The analogous inequality holds with $\chi$ replaced by $\chi_o$.
	
	Now let $N=\prod_v n_v$ and take $N/n_v$ copies of $S_v$ for every vertex $v$. These pieces satisfy the gluing condition (\ref{eqn: gluing condition}) and can be glued to form an admissible surface for $c$ of degree $n=N\cdot n'$ since $h_v(\bdry S_v)=n_v \bdry S'_{v'}$ and $\{S'_{v'}\}_{v'}$ glues up to $S'$, which is admissible of degree $n'$. Note that $S_v$ has no disk components with loop boundary, and thus $S$ has no sphere components. Each $g_i$ is of infinite order, so $S$ has no disk components either. Hence $\chi^{-}(S)=\chi(S)$ and
	\begin{eqnarray*}
		\frac{-\chi(S)}{n}&=&\frac{\sum_v -\chi_o(S_v)\cdot\frac{N}{n_v}}{N\cdot n'}\\
		&\le & \frac{\sum_{v'} [-N\chi_o(S'_{v'})+N\epsilon]}{N\cdot n'}\\
		&\le& \frac{-\chi^{-}(S')}{n'}+\frac{\#\{v'\}}{n'}\cdot \epsilon,
	\end{eqnarray*}
	
	where the summations are taken over vertices $v'$ where $S'$ has nonempty intersection with $X_{v'}$ and $\#\{v'\}$ is the number of such vertices, which is finite by compactness of $S'$.
\end{proof}

In the special case of free products, this is \cite[Theorem B]{Chen:sclfp}, whose applications can be found in \cite[Section 5]{Chen:sclfp}.

\begin{cor}
	Let $w$ be an element in a group $G$ representing a non-trivial rational homology class. Let $G(w,M,L)$ be the HNN extension $G*_\Z$ given by the inclusions $o_e,t_e:\Z\to G$ sending the generator $1$ to $w^M$ and $w^L$ respectively, where $M,L\neq 0$. Then the homomorphism $h:\BS(M,L)\to G(w,M,L)$ is an isometric embedding.
\end{cor}
\begin{proof}
	Consider both groups as graphs of groups where the underlying graph has one vertex and one edge. The inclusion of the vertex group $\Z=\langle a\rangle\to G$ with $a\mapsto w$ is injective since $[w]\neq 0\in H_1(G;\Q)$, and is an isometric embedding since $\Z$ is abelian (See Example \ref{ex: isom emb}). The non-triviality of $[w]$ also implies that condition (\ref{item: homology}) of Theorem \ref{thm: isomemb1} holds. Then it is easy to see that Theorem \ref{thm: isomemb1} applies.
\end{proof}

\subsection{Simple normal form}
From now on, we assume $\scl_{G}(c)=0$ for all $c\in B_1^H(G_v)$ and all vertices $v$, which holds for example when $\scl_{G_v}\equiv0$ by monotonicity of scl. In view of Proposition \ref{prop: rel vertex} we are interested in admissible surfaces relative to vertex groups. Each such a surface is admissible (in the absolute sense) for some rational chain, and thus has a normal form possibly after simplifications. We now further simplify the normal form using the triviality of scl in the vertex groups.

\begin{definition}
	A normal form of an admissible surface $S$ relative to vertex groups is called the \emph{simple normal form} if each component in $S_v$ has exactly one polygonal boundary and is either a disk or an annulus with the former case happening if and only if the polygonal boundary is null-homotopic in $S_v$. For simplicity, we refer to such a surface as a \emph{simple relative admissible surface}.
\end{definition}

Assume each $g_i\in \ug$ to be hyperbolic. Let $c=\sum r_ig_i$ be a rational chain whose homology class $[c]$ is in the kernel of the projection $H_1(G;\R)\to H_1(\Gamma;\R)$.

\begin{lemma}\label{lemma: simple normal form}
	For any relative admissible surface $S$ for $c$, there is another $S'$ of the same degree in simple normal form, such that
	$$-\chi(S')\le -\chi(S).$$
\end{lemma}
\begin{proof}
	By Lemma \ref{lemma: normal form}, we may assume $S$ to be in its normal form. For each subsurface $S_v$, in any of its components other than disks, take out a small collar neighborhood of each polygonal boundary. Let $S'_v$ be the disjoint union of these collar neighborhoods and disk components of $S_v$. See Figure \ref{fig: obtain simple}. Then each component of the subsurfaces in $S_v$ that we throw away to obtain $S'_v$ has at least one boundary component, has no corners, and cannot be a disk. Thus
	$$-\chi_o(S'_v)\le-\chi_o(S_v).$$
	Since $S'_v$ has all its polygonal boundaries taken from $S_v$, it satisfies the gluing condition (\ref{eqn: gluing condition}). Recall that disk components of $S_v$ must have polygonal boundaries. It follows that each component of $S'_v$ has exactly one polygonal boundary. Moreover, each loop boundary in $S'_v$ is homotopic to a polygonal boundary by construction, and thus to a loop in $X_v$. If the loop boundary is null-homotopic, replace the interior by a disk realizing the null-homotopy. Since each $g_i\in \ug$ is hyperbolic, gluing up paired turns in $S'_v$ produces a relative admissible surface $S'$ for $c$ of the same degree as $S$ with the desired properties.
\end{proof}

\begin{figure} 
	\centering
	\subfloat[]{
		\labellist
		\small \hair 2pt
		\pinlabel $S_v$ at 0 150
		\pinlabel $\beta_1$ at 120 250
		\pinlabel $\beta_2$ at 180 120
		\pinlabel $\beta_3$ at 165 25
		\endlabellist
		\centering
		\includegraphics[scale=0.5]{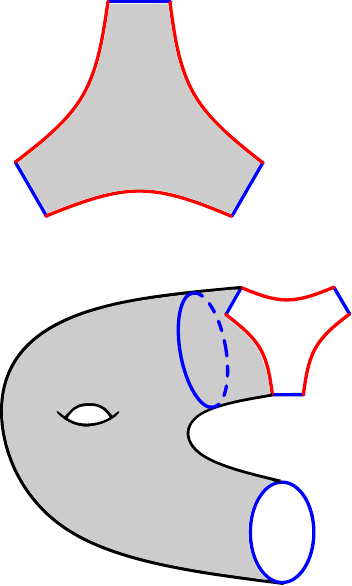}
		\label{fig: cut}
	}
	\hspace{60 pt}
	\subfloat[]{
		\labellist
		\small \hair 2pt
		\pinlabel $S'_v$ at 0 150
		\pinlabel $\beta_1$ at 120 250
		\pinlabel $\beta_2$ at 120 30
		\endlabellist
		\includegraphics[scale=0.5]{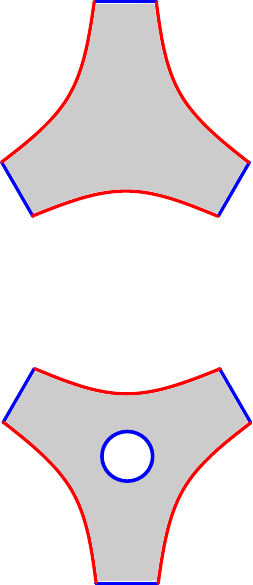}
		\label{fig: simple}
	}
	
	\caption{With the $S_v$ in Figure \ref{fig: S_v}, we cut out a neighborhood of $\beta_2$, shown on the left, and throw away the component without polygonal boundary to obtain $S'_v$ on the right.}
	\label{fig: obtain simple}
\end{figure}
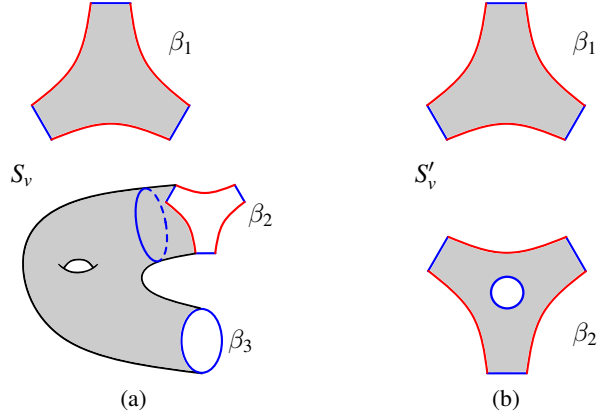

\begin{thm}[second isometric embedding]\label{thm: isomemb2}
	Let $\mathcal{G}(\Gamma, \{G_v\}, \{G_e\})$ and \linebreak $\mathcal{G}(\Gamma', \{G'_v\}, \{G'_e\})$ be graphs of groups where scl vanishes on vertex groups. Suppose there is a graph homomorphism $h:\Gamma\to\Gamma'$, injective homomorphisms $h_v:G_v\to G'_{h(v)}$ and isomorphisms $h_e:G_e\to G'_{h(e)}$ such that 
	\begin{enumerate}
		\item $h$ is injective on the set of edges;
		\item the map induced by $h_{v}$ on homology is injective on $H_v$, where $H_v$ is the sum of $\mathrm{Im}t_{e*}$ over all edges $e$ with $t(e)=v$; and
		\item the following diagram commutes for all edge pairs $(e,e')$ with $e'=h(e)$.
		\[
		\begin{CD}
		G_{o(e)} @<{o_e}<< G_e @>{t_e}>> G_{t(e)}\\
		@V{h_{o(e)}}VV @V{h_e}V{\cong}V @V{h_{t(e)}}VV\\
		G'_{o(e')} @<{o_{e'}}<< G'_{e'} @>{t_{e'}}>> G'_{t(e')}\\
		\end{CD}
		\]
	\end{enumerate}
	Then the induced homomorphism $h:\mathcal{G}(\Gamma, \{G_v\}, \{G_e\})\to\mathcal{G}(\Gamma', \{G'_v\}, \{G'_e\})$ is an isometric embedding.
\end{thm}

\begin{proof}
	The proof is similar to that of Theorem \ref{thm: isomemb1} by considering an arbitrary relative admissible surface $S'$ for $c'=h(c)$ in simple normal form instead of normal form, where $c=\sum r_ig_i$ is a rational chain in $G$ with all $g_i$ hyperbolic. Elliptic elements are ignored since it suffices to show that $h$ preserves scl relative to vertex groups according to Proposition \ref{prop: rel vertex}. The difference is that, the graph homomorphism $h$ is now allowed to collapse vertices, so there might be several vertices of $\Gamma$ mapped to the same vertex $v'$ in $\Gamma'$. However, each component of $S'_{v'}$ has a polygonal boundary on which each turn is supported in some $X'_{e'}$ connecting two arcs in $A'_{v'}$, where $t(e')=v'$. Then the two arcs must be the image of two arcs in $A_v$ under $h_v$ for the vertex $v=t(e)\in h^{-1}(v)$, where $e=h^{-1}(e')$. It follows that all the arcs on each polygonal boundary come from the same vertex $v$ and thus each polygonal boundary of $S'_{v'}$ can be pulled back to a polygonal boundary in a unique $X_v$ with $h(v)=v'$. Then the rest of the proof is the same as that of Theorem \ref{thm: isomemb1}.
\end{proof}

An immediate corollary is the following proposition, which is useful to simplify the computation of scl.
\begin{prop}[restriction of domain]\label{prop: restriction of domain}
	Let $\mathcal{G}(\Gamma, \{G_v\}, \{G_e\})$ be a graph of groups where scl vanishes on vertex groups. Let $\Gamma'$ be a subgraph of $\Gamma$, then the inclusion $i: \mathcal{G}(\Gamma', \{G_v\}, \{G_e\})\to\mathcal{G}(\Gamma, \{G_v\}, \{G_e\})$ is an isometric embedding.
\end{prop}


We can use Theorem \ref{thm: isomemb2} to compute the scl of $t$-alternating words in HNN extensions of abelian groups.

\begin{definition}\label{def: t-alt}
	Given injective homomorphisms $i,j:E\to V$, let $G=V*_E=\left<G,t\ |\ i(e)=t\ j(e)t^{-1}\right>$ be the HNN extension. An element $g$ is \emph{$t$-alternating} if it is conjugate to a cyclically reduced word of the form $a_1 t b_1 t^{-1}\ldots a_n t b_n t^{-1}$ for some $n\ge 0$.
\end{definition}

Let $H=V_1*_E V_2$ be the amalgam with $V_1=V_2=V$ and injections $i,j$ above, which has a natural inclusion into $G$, whose image is the set of $t$-alternating words.

\begin{prop}[$t$-alternating words]\label{prop: t-alt}
	Let $V$ and $E$ be abelian groups with inclusions $i,j:E\to V$, then we have an isometric embedding $h:V_1*_E V_2\to V*_E$ where $V_1=V_2=V$. In particular, with $T=t^{-1}$, we have
	\begin{eqnarray*}
		\scl_{(V*_E,V)}(a_1tb_1Ta_2tb_2T\cdots a_ntb_nT)&=&\scl_{(V_1*_E V_2,\{V_1,V_2\})}(a_1b_1\cdots a_nb_n)\\
		&=&\scl_{(V_1/i(E)*V_2/j(E),\{V_1/i(E), V_2/j(E)\})}(\bar{a}_1\bar{b}_1\cdots\bar{a}_n\bar{b}_n).
	\end{eqnarray*}
\end{prop}
\begin{proof}
	Consider $V_1*_E V_2$ and $V*_E$ as graphs of groups where the underlying graphs are a segment and a loop respectively. Let $h$ be the graph homomorphism taking the two end points of the segment to the vertex on the loop. Let $h_v$ and $h_e$ be identities. It is easy to check that the assumptions of Theorem \ref{thm: isomemb2} are satisfied since $V$ and $E$ are abelian. Thus we obtain an induced isometric embedding $h:V_1*_E V_2\to V*_E$ with $h(a_1b_1\cdots a_nb_n)=a_1tb_1Ta_2tb_2T\cdots a_ntb_nT$.
	
	Let $a_0\in V_1$ and $b_0\in V_2$ be elements such that the chain $c=a_0+b_0+a_1b_1\cdots a_nb_n$ is null-homologous. Then by Proposition \ref{prop: rel vertex}, the isometric embedding $h$ provides
	\begin{eqnarray*}
	\scl_{(V_1*_E V_2, \{V_1,V_2\})}(a_1b_1\cdots a_nb_n)&=&\scl_{V_1*_E V_2}(a_0+b_0+a_1b_1\cdots a_nb_n) \\
	&=& \scl_{V*_E}(a_0+b_0+a_1tb_1Ta_2tb_2T\cdots a_ntb_nT)\\
	&=& \scl_{(V*_E,V)}(a_1tb_1Ta_2tb_2T\cdots a_ntb_nT)
	\end{eqnarray*}
	The other equality can be proved similarly using the scl-preserving projection $\pi$ \cite[Proposition 4.3]{Susse} in the central extension
	$$1\to E\to V_1*_E V_2 \overset{\pi}{\to} V_1/i(E)*V_2/j(E)\to 1.$$
\end{proof}

In particular, scl of $t$-alternating words in Baumslag--Solitar groups can be computed as scl in free products of cyclic groups, which is better understood and easier to compute.

\begin{cor}\label{cor: t-alt}
	Let $\Z/M\Z*\Z/L\Z = \left<x,y\ |\ x^M=y^L=1\right>$. For any $t$-alternating word in $\BS(M,L)$, we have
	$$\scl_{\BS(M,L)}(a^{u_1}ta^{v_1}t^{-1}a^{u_2}ta^{v_2}t^{-1}\cdots a^{u_n}ta^{v_n}t^{-1})=\scl_{\Z/M\Z*\Z/L\Z}(x^{u_1}y^{v_1}\cdots x^{u_n}y^{v_n}).$$
\end{cor}
\begin{proof}
	It directly follows from Proposition \ref{prop: rel vertex} and Proposition \ref{prop: t-alt}.
\end{proof}

The scl spectrum of a group $G$ is the image of $\scl_G$ (as a function on $G$). Corollary \ref{cor: t-alt} implies that the spectrum of $\Z/M\Z*\Z/L\Z$ is a subset of the spectrum of $\BS(M,L)$. Due to our limited understanding of the scl spectrum, it is not clear if this is a proper subset and how big the difference is. We do know that the smallest positive elements in the two spectra exist and agree if $M$ and $L$ are odd; See \cite[Corollary 5.10]{CH} and \cite[Remark 3.6]{Chen:sclgap}.

Here are a few examples, where we write $T=t^{-1}$ and $A=a^{-1}$ for simplicity.
\begin{example}
	With the notation in Corollary \ref{cor: t-alt}, the product formula \cite[Theorem 2.93]{Cal:sclbook} implies
	$$\scl_{\BS(M,L)}(atAT)=\scl_{\Z/M*\Z/L}(xy^{-1})=\frac{1}{2}\left(1-\frac{1}{M}-\frac{1}{L}\right).$$
	This explains why \cite[Proposition 5.5]{CFL} resembles the product formula.
\end{example}
	
\begin{example}
	Similarly, using \cite[Proposition 5.6]{Chen:sclfp} instead of the product formula, we have
	$$\scl_{\BS(M,L)}(ataTAtAT)=\scl_{\Z/M*\Z/L}([x,y])=\frac{1}{2}-\frac{1}{\min(M,L)}.$$
\end{example}
	
\begin{example}
	Finally, for explicit $M$ and $L$, one can use Corollary \ref{cor: t-alt} and the computer program {\tt scallop} \cite{CW:scallop} to quickly compute scl of rather long $t$-alternating words. For example, 
	\begin{eqnarray*}
		&&\scl_{\BS(7,5)}(ataTataTatATatATataTatATatATatATataT)=\frac{123}{70}\\
		&&\scl_{\BS(5,11)}(ataTataTatATatATataTatATatATatATataTataTatAT)=\frac{102}{55}.
	\end{eqnarray*}
\end{example}


Proposition \ref{prop: t-alt} also provides a new perspective and a shorter proof for \cite[Proposition 4.4]{CW:IsomEnd} as follows.

\begin{prop}[Calegari--Walker \cite{CW:IsomEnd}]
	The homomorphism 
	$$h: F_2=\left<x,y\right>\to F_2=\left<a,t\right>$$ 
	given by $h(x)=a$ and $h(y)=tat^{-1}$ is an isometric embedding.
\end{prop}
\begin{proof}
	Consider the domain $F_2$ as a free product and the co-domain $F_2$ as a free HNN extension of $\Z$. Then the conclusion immediately follows from Proposition \ref{prop: t-alt} and Proposition \ref{prop: rel vertex}.
\end{proof}

\section{Asymptotic promotion}\label{sec: asym prom}
In the sequel, we consider graphs of groups $\mathcal{G}(\Gamma, \{G_v\}, \{G_e\})$ where
\begin{enumerate}
	\item $\scl_{G_v}\equiv0$ for each vertex $v$, and
	\item the images of edge groups in each vertex group are \emph{central} and \emph{mutually commensurable}.
\end{enumerate}

The first assumption can be weakened to $\scl_G(c)=0$ for any $c\in B_1^H(G_v)$ and any $v$, but it is usually hard to check without having $\scl_{G_v}\equiv0$.

\begin{example}\label{ex: groups}
	The following families of groups satisfy both assumptions above.
	\begin{enumerate}
		\item Any graphs of groups with vertex and edge groups isomorphic to $\Z$, since all non-trivial subgroups of $\Z$ are commensurable. These groups are also known as generalized Baumslag--Solitar groups.
		\item Amalgams of abelian groups, since the commensurability assumption is vacuous for degree-one vertices. This includes the groups studied in \cite{Susse}.
		\item Graphs of groups where vertex groups are isomorphic to the Heisenberg group $\mathcal{H}_3(\Z)$ or fundamental groups of irreducible $3$-manifolds with $\mathrm{Nil}$ geometry, and edge groups are isomorphic to $\Z$ and maps into the central $\Z$ subgroup of the vertex groups generated by a regular Seifert fiber (see \cite{Scott}). Note that these vertex groups are amenable (actually virtually solvable \cite[Theorem 4.7.8]{Thurston book}), and thus have trivial scl.
		\item Free products of groups with trivial scl. These are the groups studied in \cite{Chen:sclfp}.
	\end{enumerate}
\end{example}

As in the previous section, let $\ug$ be a finite collection of hyperbolic elements and $\uga$ be their tight loop representatives. We would like to compute scl of a rational chain $c=\sum r_i g_i$ relative to vertex groups. It comes down to understanding components that might appear in simple relative admissible surfaces. The essential difficulty is that there are too many possible components since the winding number of each turn could have infinitely many choices when some edge group is infinite. One can simply ignore the winding numbers of turns to estimate the Euler characteristic to get lower bounds of scl.  Clay--Forester--Louwsma \cite{CFL} and Susse \cite{Susse} show such a lower bound turns out to be equality in certain cases. In general, we cannot completely ignore the winding number.

Surprisingly, it turns out that recording winding numbers of turns mod certain \emph{finite index} subgroup of edge groups is sufficient (Lemma \ref{lemma: asym prom}) to asymptotically recover the surface by adjusting the winding numbers. This is the goal of this section and is the heart of the rationality theorem (Theorem \ref{thm: rational}). To this end, we investigate adjustment of winding numbers in terms of \emph{transition maps} and \emph{adjustment maps} which we will define.

For the moment, we will also assume $\Gamma$ to be \emph{locally finite} for convenience. We can always reduce the situation to this case using restriction of domain (Proposition \ref{prop: restriction of domain}).

For each vertex $v$, let $W_v$ be the subgroup generated by the images of adjacent edge groups. By our assumption and local finiteness, $W_v$ is \emph{central} in $G_v$ and each adjacent edge group is \emph{finite index} in $W_v$.

\begin{definition}
	A \emph{virtual isomorphism} $\phi:H\to H'$ is an isomorphism $\phi: H_0\to H'_0$ of finite index subgroups $H_0\le H$ and $H'_0\le H'$. The domain $\Dom\phi=H_0$ and the image $\Im\phi=H'_0$ are part of the data of $\phi$. Typically $\phi$ is not defined for elements outside $\Dom\phi$. In the case $H=H'$, we say $\phi$ is a \emph{virtual automorphism}.
	
	Two virtual isomorphisms $\phi:H_1\to H_2$ and $\psi: H_2\to H_3$ form a composition $\psi\phi$ with $\Dom \psi\phi=\phi^{-1}(\Im\phi\cap \Dom\psi)$ and $\Im \psi\phi=\psi(\Im\phi\cap \Dom\psi)$, which is again a virtual isomorphism. Each virtual isomorphism has an inverse in the obvious way. 
\end{definition}

A similar notion appears in \cite{Thomas}, which agrees with ours in the context of finitely generated abelian groups.

It follows from the commensurability assumption and local finiteness of $\Gamma$ that the inclusion $t_e=o_{\bar{e}}: G_e\to W_v$ is a virtual isomorphism for all edges $e$ with $t(e)=v$. For each edge $e$ connecting vertices $u=o(e)$ and $v=t(e)$ (possibly $u=v$), we form a \emph{transition map} $\tau_{e}:W_u\to W_v$ via $\tau_e\defeq -t_e\circ o_e^{-1}$, which is a virtual isomorphism. The negative sign makes sense since $W_u$ and $W_v$ are abelian, and the reason we add it will be clear when we define adjustment maps in Subsection \ref{subsec: disk-like}.

Moreover, for any oriented path $P=(v_0,e_1,v_1,\ldots,e_n,v_n)$ in $\Gamma$ with $o(e_i)=v_{i-1}$ and $t(e_i)=v_i$, we have a virtually isomorphic \emph{transition map} $\tau_P: W_{v_0}\to W_{v_n}$ defined as
$$\tau_P\defeq\tau_{e_n}\circ \cdots\circ \tau_{e_1}.$$

For any arc $a_v\in A_v$ on $\gamma_i\in\uga$, the loop $\gamma_i$ projects to an oriented cycle $P(a_v)=(v,e_1,v_1,\ldots,e_n,v)$ in $\Gamma$, and thus gives rise to a transition map $\tau_{P(a_v)}: W_v\to W_v$ as above. 

There is a stability result of virtual automorphisms that is important for our argument. 

\subsection{Stability of virtual automorphisms}\label{subsec: stability}
Let $\phi: H\to H$ be a virtual automorphism. Typically both $\Dom\phi^p$ and $\Im\phi^p$ will keep getting smaller as $p\to\infty$. However, when $H$ is an abelian group, the subgroup $\Dom\phi^p+\Im\phi^q$ generated by $\Dom\phi^p$ and $\Im\phi^q$ will stabilize to a finite index subgroup.

\begin{lemma}[Stability]\label{lemma: stability}
	Let $H$ be an abelian group with a virtual automorphism $\phi$. There is a \emph{finite index} subgroup $H_0$ of $H$ such that $H_0\subset\Dom\phi^p+\Im\phi^q$ for any $p,q\ge 0$.
\end{lemma}

We first look at the simplest example with $H=\Z$, where everything is quite explicit.

\begin{example}\label{ex: index compute}
	Let $X,Y$ be non-zero integers. Let $\phi:\Z\to\Z$ be the virtual automorphism given by $\phi(X)=Y$ with $\Dom\phi=X\Z$ and $\Im\phi=Y\Z$. Let $d=\gcd(|X|,|Y|)$, $x=|X|/d$ and $y=|Y|/d$. Then $\Dom\phi^2=\phi^{-1}(\Im\phi\cap\Dom\phi)=\phi^{-1}(dxy\Z)=dx^2\Z$ and $\Im\phi^2=\phi(\Im\phi\cap\Dom\phi)=\phi(dxy\Z)=dy^2\Z$.
	
	More generally, we have $\Dom\phi^p=dx^p\Z$ and $\Im\phi^q=dy^q\Z$, both keep getting smaller as $p,q\to\infty$. However, $\Dom\phi^p+\Im\phi^q=\gcd(dx^q,dy^q)\Z=d\Z$ for all $p,q\ge1$.
\end{example}

We prove Lemma \ref{lemma: stability} by computing the index. Denote the index of $B\le A$ by $|A:B|$. Recall the following basic identities, which we will use.

\begin{lemma}\label{lemma: basic index}
	Let $A$ be an abelian group with finite index subgroups $B$ and $C$.
	\begin{enumerate}
		\item If $C$ is a subgroup of $B$, then $|B:C|=|A:C|/|A:B|$.
		\item $$|A:B+C|=\frac{|A:B|}{|B+C:B|}=\frac{|A:B|}{|C:B\cap C|}.$$
		\item If $\phi$ is an injective homomorphism defined on $A$, then $|A:B|=|\phi(A):\phi(B)|$.
	\end{enumerate}
\end{lemma}

\begin{lemma}\label{lemma: index compute}
	Let $H$ be an abelian group with a virtual automorphism $\phi$. Let $I_p\defeq|H:\Im\phi^p|$.
	\begin{enumerate}
		\item We have
		$$|H:\Dom\phi^p+\Im\phi^q|=\frac{I_p I_q}{I_{p+q}}.$$
		\item There exist integers $r\ge1$ and $N_\phi\ge1$ such that $I_{q+1}/I_{q}=r$ for all $q\ge N_\phi$.
	\end{enumerate}
\end{lemma}
\begin{proof}
	Note that $\phi^p$ isomorphically maps $\Dom\phi^p\cap \Im\phi^q$ to $\Im\phi^{p+q}$ and $\Dom\phi^p$ to $\Im\phi^p$. By the formulas in Lemma \ref{lemma: basic index}, we have
	$$|H:\Dom\phi^p+\Im\phi^q|=\frac{|H:\Im\phi^q|}{|\Dom\phi^p:\Dom\phi^p\cap \Im\phi^q|}=\frac{I_p I_q}{I_{p+q}}.$$
	
	To prove the second assertion, let $p=1$ in the equation above. We have
	$$\frac{I_{q+1}}{I_{q}}=\frac{I_1}{|H:\Dom\phi+\Im\phi^q|}.$$
	Since the sequence of \emph{integers} $|H:\Dom\phi+\Im\phi^q|$ is increasing in $q$ with upper bound $|H:\Dom\phi|$, it must stabilize when $q\ge N_\phi$ for some $N_\phi\ge1$ and thus $I_{q+1}/I_{q}$ stabilizes to some $r$ when $q\ge N_\phi$.
\end{proof}

\begin{proof}[Proof of Lemma \ref{lemma: stability}]
	With notation as in Lemma \ref{lemma: index compute}, we have $I_{q}=r^{q-N_\phi}I_{N_\phi}$ for all $q\ge N_\phi$. Thus
	$$|H:\Dom\phi^p+\Im\phi^q|=\frac{I_p I_q}{I_{p+q}}=\frac{r^{p+q-2N_\phi}I^2_{N_\phi}}{r^{p+q-N_\phi}I_{N_\phi}}=\frac{I_{N_\phi}}{r^{N_\phi}},$$
	for any $p,q\ge N_\phi$. 
	
	Since $\Dom\phi^p+\Im\phi^q\subset\Dom\phi^{p'}+\Im\phi^{q'}$ whenever $p\ge p'$ and $q'\ge q$, the index computation above shows that there is an index $I_{N_\phi}/r^{N_\phi}$ subgroup $H_0$ of $H$ such that $\Dom\phi^p+\Im\phi^q=H_0$ when $p,q\ge N_\phi$. Therefore, for arbitrary $p,q\ge0$, 
	$$\Dom\phi^p+\Im\phi^q \supseteq\Dom\phi^{\max\{p,N_\phi\}}+\Im\phi^{\max\{q,N_\phi\}}=H_0.$$
\end{proof}

\subsection{Disk-like pieces}\label{subsec: disk-like}
Recall that any arc $a_v\in A_v$ on $\gamma_i\in\uga$, determines an oriented cycle $P(a_v)=(v,e_1,v_1,\ldots,e_n,v)$ in $\Gamma$, and consequently a transition map $\tau_{P(a_v)}$ which is a virtual automorphism of $W_v$.

Consider a simple relative admissible surface $S$ for $c$. Refer to a component of a subsurface $S_v$ as \emph{a piece}. Recall that $S$ is obtained by gluing pieces together along paired turns on the polygonal boundaries in a certain way that can be encoded by a graph $\Gamma_S$, where each vertex corresponds to a piece and each edge corresponds to two paired turns glued together in $S$. There is a graph homomorphism $\pi: \Gamma_S\to\Gamma$ taking a vertex $\hat{v}$ to $v$ if $\hat{v}$ corresponds to a piece $C$ in $S_v$. In this way, $S$ admits an induced structure of a graph of spaces with underlying graph $\Gamma_S$.

Recall that each piece $C$ in $S_v$ has a unique polygonal boundary, which as a loop in $X_v$ represents a conjugacy class $w(C)$ in $G_v$ called \emph{the winding number} of $C$. Recall that $W_v$ is the subgroup generated by $t_e(G_e)$ with $t(e)=v$, which is central, so it makes sense to say whether $w(C)$ lies in $W_v$.

\begin{definition}
	We say a piece $C$ in $S_v$ is a \emph{potential disk} if its winding number $w(C)$ lies in $W_v$.
\end{definition}

Since $S$ is in simple normal form, a piece $C$ is a disk if and only if the winding number $w(C)$ is trivial. Then potential disks are exactly the pieces that can be made into disks by adjusting winding numbers of turns on their polygonal boundaries. However, this typically cannot be done for all potential disks \emph{simultaneously}. Our remedy is to find a class of potential disks, called \emph{disk-like pieces}, that can be made into disks simultaneously in an asymptotic sense. Moreover, we will make sure that keeping track of a finite amount of information suffices to tell whether a piece is disk-like.

Recall that each oriented edge $\hat{e}$ of $\Gamma_S$ going from $\hat{v}$ to $\hat{u}$ represents a turn shared by the pieces represented by $\hat{v}$ and $\hat{u}$. Changing the winding number of the turn by $n\in G_e$ would adjust $w(\hat{v})$ and $w(\hat{u})$ by $o_e(n)$ and $-t_e(n)$ respectively, where $e$ is the projection of $\hat{e}$ in $\Gamma$. Here the negative sign is due to the opposite orientations on the turn induced by the two pieces, which is why we add a negative sign in the definition of transition maps.

For an oriented path $\hat{P}=(\hat{v}_0,\hat{e}_1,\hat{v}_1,\ldots,\hat{e}_n,\hat{v}_n)$ in $\Gamma_S$ projecting to a path $P$ in $\Gamma$, associate to $\hat{P}$ the adjustment map $\alpha(\hat{P})\defeq\tau_{P}$. Then for any $x\in \Dom \alpha(\hat{P})$, we can adjust the winding numbers of the turns represented by $\hat{e}_i$ ($1\le i\le n$) such that, for any $1\le i\le n-1$, the changes to the winding number of $\hat{v}_i$ contributed by the adjustment on $\hat{e}_i$ and $\hat{e}_{i+1}$ cancel each other, and the net result of the adjustment is
\begin{enumerate}
	\item the winding numbers of $\hat{v}_0$ and $\hat{v}_n$ increase by $x$ and $\alpha(\hat{P})(x)$ respectively;
	\item the winding number of $\hat{v}_i$ stays unchanged for all $1\le i\le n-1$.
\end{enumerate}
We say such an adjustment is supported on $\hat{v}_0$ and $\hat{v}_n$. 

\begin{definition}
	For each $a_v\in A_v$, fix a finite index subgroup $W(a_v)$ of $W_v$ such that 
	\begin{equation}
		W(a_v)\subset\Dom\tau_{P(a_v)}^p+\Im\tau_{P(a_v)}^q\quad \text{for any }p,q\ge 0.\label{eqn: W(a_v)}
	\end{equation}
	Such a $W(a_v)$ exists by Lemma \ref{lemma: stability}.
\end{definition}

Consider any piece $C$ with a copy of $a_v$ on its polygonal boundary. If $\gamma_i\in \uga$ is the loop contain $a_v$, then the component $B$ of $\bdry S$ which this copy of $a_v$ sits on is a finite cover of $\gamma_i$, say of degree $n$. Then $B$ successively passes through $n$ copies of $a_v$ contained in pieces $C_j$, $j=0,\ldots, n-1$, where $C_0=C$. See Figure \ref{fig: followbdry}. It might happen that $C_j=C_k$ for $j\neq k$, in which case its polygonal boundary contain the $j$-,$k$-th copies of $a_v$ on $B$ as distinct arcs.

\begin{figure}
	\labellist
	\small 
	\pinlabel $B$ at 10 10
	
	\pinlabel $a_v$ at -10 60
	\pinlabel $C_0=C$ at 26 75

	\pinlabel $C_1=C_4$ at 130 80
	\pinlabel $a_v$ at 130 103
	\pinlabel $a_v$ at 130 15
	
	\pinlabel $a_v$ at 280 120
	\pinlabel $C_3$ at 275 90
	
	\pinlabel $a_v$ at 280 -5
	\pinlabel $C_2$ at 287 30
	
	\endlabellist
	\centering
	\includegraphics[scale=1]{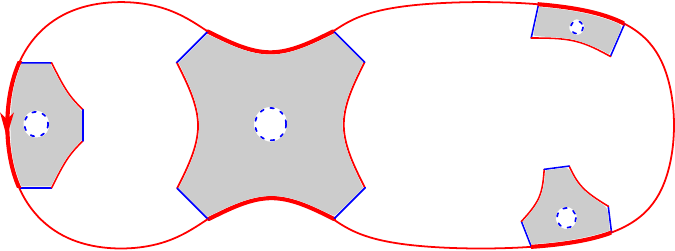}
	\caption{An example of a boundary component $B$ of $S$ with $n=5$ copies of $a_v$, where $C_1$ coincides with $C_4$, and where pieces sitting along $B$ not witnessing $a_v$ are omitted}\label{fig: followbdry}
\end{figure}
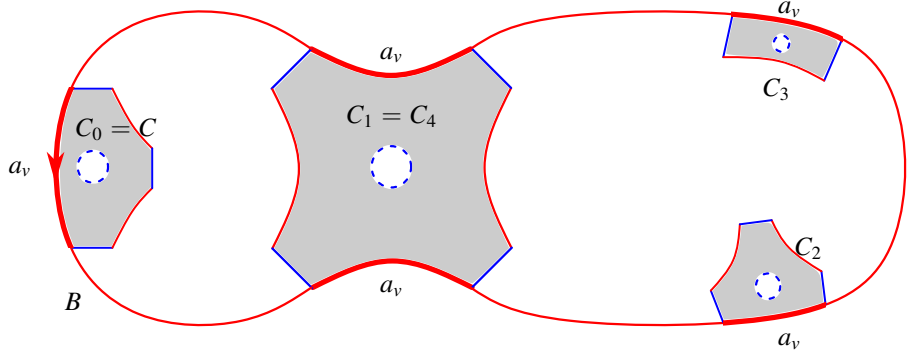

The boundary component $B$ visits a sequence of pieces and thus determines an oriented cycle $\omega$ in $\Gamma_S$, on which we have vertices $\hat{v}_j$ corresponding to $C_j$ for $j=0,\ldots, n-1$. Let $\omega_j$ and $\omega'_j$ be the subpaths on $\omega$ going from $\hat{v}_0$ to $\hat{v}_j$ in the positive and negative orientation respectively, such that under the projection $\pi: \Gamma_S\to \Gamma$, we have $\pi(\omega_j)=P(a_v)^j$ and $\pi(\omega'_j)=\overline{P(a_v)}^{n-j}$. Thus the adjustment maps are $\alpha(\omega_j)=\tau_{P(a_v)}^j$ and $\alpha(\omega'_j)=\tau_{P(a_v)}^{j-n}$.

\begin{lemma}\label{lemma: kill winding}
	With the above notation, let $C_j\neq C_0$ be a piece with $w(C_j)\in W(a_v)$. Then there is an adjustment of $S$ supported on $C_0$ and $C_j$ after which $C_j$ has trivial winding number.
\end{lemma}
\begin{proof}
	By definition, we have $w(C_j)\in \Dom\tau_{P(a_v)}^{n-j}+\Im\tau_{P(a_v)}^{j}=\Im \alpha(\omega'_j)+\Im \alpha(\omega_j)$. As a consequence, there exist $x,y\in W_v$ with $w(C_j)=\alpha(\omega'_j)(x)+\alpha(\omega_j)(y)$. This gives rise to an adjustment that eliminates $w(C_j)$ and adjusts $w(C_0)$ to $w(C_0)-x-y$ without changing all other winding numbers.
\end{proof}

Let $\widetilde{\Gamma}_S$ be a finite cover of $\Gamma_S$. This determines a cover $\widetilde{S}$ of $S$ in simple normal form. Let $\widetilde{B}$ be a lift of $B$ of degree $m$, which corresponds to a degree $m$ lift $\widetilde{\omega}$ of $\omega$ in $\widetilde{\Gamma}_S$. Denote the $m$ (distinct) lifts of $C$ by $\widetilde{C}_0,\ldots,\widetilde{C}_{m-1}$ as vertices on $\widetilde{\omega}$.

\begin{lemma}\label{lemma: lift and kill}
	With the above notation, suppose $w(C)\in W(a_v)$, then there is an adjustment of $\widetilde{S}$ supported on $\widetilde{C}_0,\ldots,\widetilde{C}_{m-1}$ after which $\widetilde{C}_k$ has trivial winding number for all $1\le k\le m-1$.
\end{lemma}
\begin{proof}
	For each $1\le k\le m-1$, apply Lemma \ref{lemma: kill winding} to the loop $\widetilde{\omega}$ to eliminate the winding number $w(\widetilde{C}_k)$ at the cost of changing  $w(\widetilde{C}_0)$.
\end{proof}

Once we normalize by the degree of admissible surfaces, the adjustment in Lemma \ref{lemma: lift and kill} has the effect of making $1-1/m$ portion of an annulus piece $C$ into disk pieces. This implies that $C$ can be asymptotically promoted to a disk as $m\to\infty$ without affecting other pieces if $w(C)\in W(a_v)$. In the exceptional case where $\omega$ is null-homotopic in $\Gamma_S$, we cannot find finite covers with $m\to\infty$ and need a different strategy to promote $C$.

\begin{lemma}\label{lemma: backtrack}
	The cycle $\omega$ representing the boundary component $B$ of $S$ in $\Gamma_S$ backtracks at a vertex $\hat{u}$ if and only if $\hat{u}$ has valence one. In this case, the piece $\hat{u}$ has only one turn on the polygonal boundary and is not a potential disk.
\end{lemma}
\begin{proof}
	Let $C'$ be a piece represented by a vertex $\hat{u}$ in $\Gamma_S$ that $\omega$ passes through. Note that the edges adjacent to $\hat{u}$ correspond to the turns on the polygonal boundary of $C'$ and thus have an induced cyclic order. See Figure \ref{fig: backtrackpiece}. Since $\omega$ represents the boundary component, the two edges that $\omega$ enters and leaves $\hat{u}$ are adjacent in the cyclic order. Thus $\omega$ backtracks at $\hat{u}$ if and only if $\hat{u}$ has valence one, in which case the polygonal boundary of $C'$ consists of one arc $a'_u\in A_u$ and one turn. Let $\hat{e}$ be the edge representing the gluing of this turn and let $e=\pi(\hat{e})$ so that $t(e)=\pi(\hat{u})$. Then $w(a'_u)\notin t_e(G_e)$ since $\uga$ consists of tight loops. It follows that $C'$ cannot be a potential disk.
\end{proof}

\begin{figure}
	\labellist
	\small 
	
	\endlabellist
	\centering
	\includegraphics[scale=1]{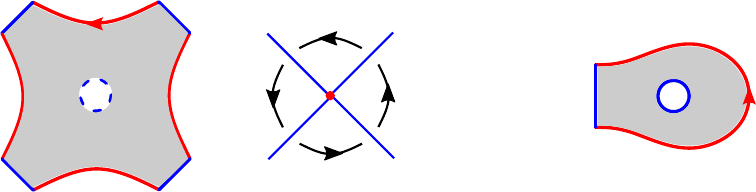}
	\caption{The cyclic order on the edges around a vertex in $\Gamma_S$ (left); A piece representing a vertex where the cycle $\omega$ backtracks (right)}\label{fig: backtrackpiece}
\end{figure}

\begin{lemma}\label{lemma: nullpush}
	With the notation above, let $B$ be the boundary component of $S$ passing through a copy of $a_v$ on a potential disk $C$. If the loop $\omega$ representing $B$ in $\Gamma_S$ is null-homotopic, then $P(a_v)$ is a null-homotopic loop in $\Gamma$ and $\tau_{P(a_v)}$ is the identity on its domain. In particular, we have $W(a_v)\subset\Dom\tau_{P(a_v)}$.
\end{lemma}
\begin{proof}
	The conclusions follow immediately from the fact that $P(a_v)$ is the image of $\omega$ under the graph homomorphism $\pi: \Gamma_S\to\Gamma$. We deduce $W(a_v)\subset\Dom\tau_{P(a_v)}$ from equation (\ref{eqn: W(a_v)}).
\end{proof}

\begin{lemma}\label{lemma: exceptional kill}
	With the notation above, suppose $\omega$ is null-homotopic and $w(C)\in W(a_v)$. Then there is an adjustment of $S$ supported on $C$ and a piece $C'$ that is not a potential disk, such that $w(C)$ becomes $0$ after the adjustment.
\end{lemma}
\begin{proof}
	Let $C'$ be a piece as in Lemma \ref{lemma: backtrack} representing a vertex $\hat{u}$ where $\omega$ backtracks. Let $\omega_0$ be the positively oriented subpath of $\omega$ going from $C$ to $C'$. Considering $\omega$ as a cycle based at $C$, we have $\Dom\alpha(\omega)\subset\Dom\alpha(\omega_0)$ for the adjustment maps. Thus $w(C)\in W(a_v)\subset \Dom\alpha(\omega_0)$ by Lemma \ref{lemma: nullpush} and the claimed adjustment exists.
\end{proof}

\begin{definition}\label{def: disk-like}
	Given the choices of $W(a_v)$ for all arcs $a_v\in A_v$, a piece $C$ of $S_v$ is \emph{a disk-like piece} if there is an arc $a_v\in A_v$ on the polygonal boundary of $C$ such that $w(C)\in W(a_v)$. Let $$\what{\chi}(C)\defeq \left\{\begin{array}{ll}
	1	&\text{if }$C$\text{ is disk-like},\\
	0	&\text{otherwise}
	\end{array}\right.$$
	be the \emph{over-counting Euler characteristic} that counts disk-like pieces as disks. Let $\what{\chi}_o(C)\defeq -\frac{1}{4}\#\text{corners}+\what{\chi}(C)$ be the \emph{over-counting orbifold Euler characteristic}. For a simple relative admissible surface $S$, define its \emph{over-counting Euler characteristic} as $$\what{\chi}(S)\defeq\sum \what{\chi}_o(C),$$ where the sum is taken over all pieces $C$. Equivalently, 
	$$\what{\chi}(S)=\chi(S)+\#\{C: \text{disk-like but not a disk}\}.$$
\end{definition}

Now we show that the over-counting is accurate via asymptotic promotion.
\begin{lemma}[Asymptotic Promotion]\label{lemma: asym prom}
	Let $G$ be a graph of groups $\mathcal{G}(\Gamma, \{G_v\}, \{G_e\})$ where
	\begin{enumerate}
		\item $\Gamma$ is locally finite,
		\item $\scl_{G_v}\equiv0$, and
		\item the images of edge groups in each vertex group are central and mutually commensurable.
	\end{enumerate}
	Let $\ug$ be a finite collection of hyperbolic elements of $G$. With the notion of disk-like pieces and $\what{\chi}$ above (which depends on the choices of $W(a_v)$), we have the following: For any simple relative admissible surface $S$ for a rational chain $c=\sum r_i g_i$ of degree $n$, and for any $\epsilon>0$, there is a simple relative admissible surface $S'$ of a certain degree $n'$ such that
	$$\frac{-\chi(S')}{2n'}\le \frac{-\what{\chi}(S)}{2n}+\epsilon.$$
\end{lemma}
\begin{proof}
	For each disk-like piece $C$, say a component of $S_v$, fix an arc $a_v(C)\in A_v$ on the polygonal boundary of $C$ such that $w(C)\in W(a_v(C))$. Let $\omega(C)$ be the oriented cycle in $\Gamma_S$ determined by the boundary component of $S$ containing the copy of $a_v(C)$ on $\bdry C$.
	
	For those pieces $C$ with null-homotopic $\omega(C)$, by applying the adjustment as in Lemma \ref{lemma: exceptional kill}, we assume $w(C)=0$ and $C$ is a genuine disk.
	
	For any given $\epsilon>0$, choose a large integer $N>\frac{1}{2n\epsilon}\#\{C: \text{disk-like and } w(C)\neq0\}$. For the finite collection of non-trivial loops $\Omega\defeq\{\omega(C): C\text{ disk-like and } w(C)\neq0\}$, since free groups are residually finite, there is a finite cover $\wtilde{\Gamma}_S$ of $\Gamma_S$ such that any lift of each $\omega\in \Omega$ has degree at least $N$. Let $M$ be the covering degree.
	
	Let $\wtilde{S}$ be the finite cover of $S$ corresponding to $\widetilde{\Gamma}_S$. Applying the adjustment in Lemma \ref{lemma: lift and kill} to each lift of each $\omega\in \Omega$, we observe that, for every disk-like piece $C$ with $w(C)\neq0$, at least $M(1-1/N)$ of its $M$ preimages in $\wtilde{S}$ after adjustment bounds a disk.
	
	Denote by $S'$ the surface obtained by the adjustment above from $\wtilde{S}$. Then $S'$ is simple relative admissible of degree $n'=Mn$. We have
	\begin{eqnarray*}
	\frac{-\chi(S')}{2n'}&\le&\frac{-\chi(\wtilde{S})-M(1-1/N)\#\{C: \text{disk-like and } w(C)\neq0\}}{2nM}\\
	&=&\frac{-\chi(S)-(1-1/N)\#\{C: \text{disk-like and } w(C)\neq0\}}{2n}\\
	&=&\frac{-\what{\chi}(S)}{2n}+\frac{\#\{C: \text{disk-like and } w(C)\neq0\}}{2nN}\\
	&<&\frac{-\what{\chi}(S)}{2n}+\epsilon.
	\end{eqnarray*}
\end{proof}



With Lemma \ref{lemma: asym prom} above, we can work with disk-like pieces and $\what{\chi}$ instead of genuine disks and $\chi$. The advantage is that deciding whether a piece is disk-like or not is equivalent to checking whether an element in a finite abelian group vanishes or not, which only requires keeping track of a finite amount of information.

\section{Determining scl by linear programming}\label{sec: scl by linprog}
With the help of Lemma \ref{lemma: asym prom}, there are several known methods of encoding to compute scl via linear programming \cite{Cal:freegroup PQL, BCF, CFL, Cal:sss, Wal:scylla}. In this section, we use an encoding similar to those in \cite{Cal:sss,Chen:sclfp} to optimize our rationality result. We will use the notation from Section \ref{sec: asym prom} and the setup in Lemma \ref{lemma: asym prom}. We further assume the graph $\Gamma$ to be \emph{finite} for our discussion until Theorem \ref{thm: rational}, where we use restriction of domain (Proposition \ref{prop: restriction of domain}).

Given the choices of the finite index subgroups $W(a_v)$ of $W_v$ satisfying (\ref{eqn: W(a_v)}) for all arcs $a_v\in A_v$ for any vertex $v$, let $D_v\defeq\cap_{a_v\in A_v}W(a_v)$ be the intersection, which is also finite index in $W_v$. For each edge $e$ of $\Gamma$, let $D_e\defeq o_{e}^{-1}D_{o(e)} \cap t_{e}^{-1}D_{t(e)}$ and $W_e\defeq G_e/D_e$. Then $W_e$ is a finite abelian group with induced homomorphisms $\overline{o_{e}}: W_e\to W_{o(e)}/D_{o(e)}$ and $\overline{t_{e}}: W_e\to W_{t(e)}/D_{t(e)}$.

We are going to consider all abstract pieces that potentially appear as a component of $S_v$ for some simple relative admissible surface $S$ for some rational chain supported on finitely many hyperbolic elements $\ug=\{g_i: i\in I\}$ represented by tight loops $\uga=\{\gamma_i: i\in I\}$. Thus we extend some previous definitions as follows.

\begin{definition}
	For each vertex $v$, \emph{a piece $C$ at $v$} is a surface with corners in the thickened vertex space $N(X_v)$ satisfying the following properties:
	\begin{enumerate}
		\item $C$ has a unique polygonal boundary where edges alternate between arcs in $A_v$ and turns connecting them, where a turn (with the induced orientation) going from $a_v$ to $a'_v$ is an arc supported on $X_e$ connecting the terminus point of $a_v$ and initial point of $a'_v$.
		\item $C$ is either a disk or an annulus depending on its winding number $w(C)$ as follows, where the winding number is the conjugacy class in $G_v$ represented by the polygonal boundary of $C$.
		\begin{enumerate}
			\item $C$ is a disk if $w(C)$ is trivial in $G_v$. The map in the interior of $C$ is a null-homotopy of its polygonal boundary in $N(X_v)$.
			\item $C$ is an annulus if its winding number $w(C)$ is non-trivial, where the other boundary has no corners and represents a loop in $X_v$ homotopic to the polygonal boundary in $N(X_v)$, and the homotopy gives the interior of $C$.
		\end{enumerate}
	\end{enumerate}
	As before, we define a piece $C$ at $v$ to be \emph{disk-like} if there is an arc $a_v\in A_v$ on the polygonal boundary of $C$ such that $w(C)\in W(a_v)$.
\end{definition}

For each vertex $v$ of $\Gamma$, let $T_v$ be the set of triples $(a_v,\bar{w},a'_v)$ where $a_v,a'_v\in A_v$ are arcs with terminus point of $a_v$ and initial point of $a'_v$ on a common edge space $X_e$, and $\bar{w}\in W_e$. Note that $T_v$ is a \emph{finite} set since there are finitely many pairs of $(a_v,a'_v)$ as above and $\bar{w}$ lies in a finite group given each such a pair. We use such a triple to record a turn from $a_v$ to $a'_v$ with winding number in the coset $\bar{w}$.

Consider a piece $C$ at vertex $v$ with $m$ turns. For each turn, say supported on $X_e$ and going from $a_v$ to $a'_v$ with winding number $w\in G_e$, the triple $(a_v,\bar{w},a'_v)$ is in $T_v$, where $\bar{w}$ is the image of $w$ in the quotient $W_e$. Regard this triple as a basis vector in $\R^{T_v}$ and let $x(C)\in \R^{T_v}$ be the sum of such triples over the $m$ turns of $C$. 

Let $\bdry: \R^{T_v}\to \R^{A_v}$ be the rational linear map given by $\bdry(a_v,\bar{w},a'_v)\defeq a'_v-a_v$ for all $(a_v,\bar{w},a'_v)\in T_v$. Then $\bdry(x(C))=0$ as the polygonal boundary of $C$ closes up.

There is a pairing on the set $\cup_v T_v$ similar to the pairing of turns (Figure \ref{fig: pair}). For each triple $(a_v,\bar{w},a'_v)\in T_v$ with $\bar{w}\in W_e$, there is a unique triple $(a_u,-\bar{w},a'_u)\in T_u$ such that $a_v$ and $a_u$ are followed by $a'_u$ and $a'_v$ in $\uga$ respectively, where $u$ is necessarily the vertex adjacent to $v$ via $e$. We say such two triples are \emph{paired}.

Denote the convex rational polyhedral cone $\R_{\ge0}^{T_v}\cap\ker\bdry$ by $\mathcal{C}_v$. For any simple relative admissible surface $S$ for any rational chain supported on $\uga$, let $x(S_v)=\sum x(C)$, where the sum is taken over all pieces $C$ of $S_v$. Then $x(S_v)$ is an integer point in $\mathcal{C}_v$ and paired turns are encoded by paired triples. Let $x(S)\in\prod_v \mathcal{C}_v$ be the element with $v$-coordinate $x(S_v)$.

Let $\mathcal{C}(\uga)$ be the subspace of $\prod_v \mathcal{C}_v$ consisting of points satisfying the \emph{gluing condition} which we now describe. This is the analog of (\ref{eqn: gluing condition}) on $\prod_v \mathcal{C}_v$. For each triple $(a_v,\bar{w},a'_v)\in T_v$, let $\#_{(a_v,\bar{w},a'_v)}$ be the linear function taking the $(a_v,\bar{w},a'_v)$ coordinate. We say $x\in\mathcal{C}(\uga)$ satisfies the gluing condition if $\#_{(a_v,\bar{w},a'_v)}(x)=\#_{(a_u,-\bar{w},a'_u)}(x)$ for all paired triples $(a_v,\bar{w},a'_v)$ and $(a_u,-\bar{w},a'_u)$. It follows that $\mathcal{C}(\uga)$ is a rational polyhedral cone.

For each $\gamma_i\in\uga$, fix an arc $a_i$ in some $A_v$ supported on $\gamma_i$ and let
$\#_{\gamma_i}=\sum \#_{(a_i,w,a'_v)}$, where the sum is taken over all triples in $T_v$ starting with $a_i$. For any $x\in\mathcal{C}(\uga)$, the gluing condition implies $\#_{\gamma_i}(x)$ is independent of the choice of arc $a_i$ on $\gamma_i$. Roughly speaking, the rational linear function $\#_{\gamma_i}(x)$ counts how many times $x$ winds around $\gamma_i$.

For $\bm{r}=(r_i)\in \Q^I_{\ge0}$ ($I$ is the index set of $\uga$), let $c(\bm{r})$ be the rational chain $\sum r_i g_i$, and let $\mathcal{C}(\bm{r})$ be the set of $x\in\mathcal{C}(\uga)$ satisfying the \emph{normalizing condition} $\#_{\gamma_i}(x)=r_i$ for all $i\in I$. Let $h(\bm{r})$ be the homology class of $c(\bm{r})$ under the projection $H_1(G;\R)\to H_1(\Gamma;\R)$. Then $h$ is a rational linear map.

\begin{lemma}\label{lemma: space nonempty}
	The set $\mathcal{C}(\bm{r})$ is nonempty if and only if $\bm{r}\in \ker h$, in which case, it is a compact rational polyhedron depending piecewise linearly on $\bm{r}$.
\end{lemma}
\begin{proof}
	The chain $c(\bm{r})$ bounds an admissible surface relative to vertex groups if and only if $h(\bm{r})=0$, according to the computation (\ref{eqn: homology of G}). Whenever a relative admissible surface exists, it can be simplified to one in simple normal form, which is encoded as a vector in $\mathcal{C}(\bm{r})$. Then in this case, each $\mathcal{C}(\bm{r})$ is the intersection of a rational polyhedral cone $\mathcal{C}(\uga)$ with a rational linear subspace $\cap_i\{x: \#_{\gamma_i}(x)=r_i\}$ that depends linearly on $\bm{r}$. Thus $\mathcal{C}(\bm{r})$ is a closed rational polyhedron depending piecewise linearly on $\bm{r}$. It is compact since the equations $\#_{\gamma_i}(x)=r_i$ impose upper bounds on all coordinates.
\end{proof}

\begin{definition}
	For each vertex $v$, an integer point $d\in\mathcal{C}_v$ is a \emph{disk-like vector} if $d=x(C)$ for some disk-like piece $C$ at $v$. Let $\mathcal{D}(v)$ be the set of disk-like vectors.
	
	For any $x\in \mathcal{C}_v$, we say $x=x'+\sum t_j d_j$ is an \emph{admissible expression} if $x'\in \mathcal{C}_v$, $d_j\in \mathcal{D}(v)$ and $t_j\ge0$. Define
	$$\kappa_v(x)\defeq \sup\left\{\sum t_j\ |\ x=x'+\sum t_j d_j \text{ is an admissible expression}\right\}.$$
\end{definition}

The key point of our encoding and the definitions of $D_v,D_e,W_e$ is to have enough information on the winding numbers of turns to tell whether a piece is disk-like.
\begin{lemma}\label{lemma: enough info}
	For each $\bar{w}\in W_e$, fix an arbitrary lift $w\in G_e$. Then any disk-like vector $d\in \mathcal{C}_v$ can be realized as a disk-like piece $C$ at $v$ such that every turn from $a_v$ to $a'_v$ on the polygonal boundary of $C$ representing a triple $(a_v,\bar{w}, a'_v)$ has winding number $w$.
\end{lemma}
\begin{proof}
	By definition, there is some disk-like piece $C_0$ at $v$ realizing the given disk-like vector $d$. Thus, for some $a_{v,0}\in A_v$, the winding number $w(C_0)$ lies in $W(a_{v,0})\subset W_v$, which is central in $G_v$. Suppose $C_0$ contains a turn from $a_v$ to $a'_v$ supported on the edge space $X_e$ representing a triple $(a_v,\bar{w}, a'_v)$ with actual winding number $w_0\in G_e$ and $e=e_{out}(a_v)$. Then $w-w_0\in D_e$ and changing the winding number of the turn from $w_0$ to $w$ would change the winding number $w(C_0)$ of the piece $C_0$ by $o_e(w-w_0)\in D_v$. Since $D_v\subset W(a_{v,0})$, such a change preserves the property of being disk-like. After finitely many such changes, we modify $C_0$ to a disk-like piece $C$ with the desired winding numbers of turns.
\end{proof}

Any admissible surface $S$ naturally provides an admissible expression for $x(S_v)$ by sorting out disk-like pieces among components of $S_v$. Hence $\kappa_v(x(S_v))$ is no less than the number of disk-like pieces in $S_v$.

For each vector $x\in \mathcal{C}_v\subset \R_{\ge0}^{T_v}$, denote its $\ell^1$-norm by $|x|$. Then $|\cdot|$ on $\mathcal{C}_v$ coincides with the linear function taking value $1$ on each basis vector. Thus $\sum_v |x(S_v)|$ is the total number of turns in $S$, which is twice the number of corners.

Recall from Definition \ref{def: disk-like} that $\what{\chi}(S)$ is the over-counting Euler characteristic that counts disk-like pieces as disks in a simple relative admissible surface $S$.

\begin{lemma}\label{lemma: overcount accurate}
	Fix any $\bm{r}$ as above. 
	\begin{enumerate}
		\item For any $S$ simple relative admissible for $c(\bm{r})$ of degree $n$, we have $x(S)/n \in \mathcal{C}(\bm{r})$ and
		$$\frac{-\what{\chi}(S)}{2n(S)}\ge\sum_v \frac{1}{4}|x(S_v)/n|-\sum_v \frac{1}{2}\kappa_v(x(S_v)/n).$$
		\item For any rational point $x=(x_v)\in \mathcal{C}(\bm{r})$ and any $\epsilon>0$, there is a simple relative admissible surface $S$ for $c(\bm{r})$ of a certain degree $n$ such that $x(S_v)/n=x_v$ and $$\frac{-\what{\chi}(S)}{2n}\le\sum_v \frac{1}{4}|x_v|-\sum _v\frac{1}{2}\kappa_v(x_v)+\epsilon.$$
	\end{enumerate}
\end{lemma}
\begin{proof}
	\begin{enumerate}
		\item It is easy to see $x(S)/n \in \mathcal{C}(\bm{r})$ from the definition. To obtain the inequality, recall that 
		$$-\what\chi(S)=-\sum_{C} \what{\chi}_o(C)=\frac{1}{4}\#\text{corners}-\sum_{C}\what{\chi}(C)=\frac{1}{2}\#\text{turns}-\#\text{disk-like pieces}.$$
		Since $\kappa_v(x(S_v))$ is no less than the number of disk-like pieces in $S_v$, the inequality follows.
		\item By finiteness of $\Gamma$, there are admissible expressions $x_v=x'_v+\sum_j t_{j,v}d_{j,v}$ with each $t_{j,v}\in\Q_{\ge0}$ such that $\sum_{j,v} t_{j,v}+2\epsilon>\sum_v \kappa_v(x_v)$. Note that each $x'_v$ is rational as each $t_{j,v}$ and $x_v$ are. Choose an integer $n$ so that each $nt_{j,v}$ is an integer and $nx'_v$ is an integer vector in $\mathcal{C}_v$. For each $\bar{w} \in W_e$, fix a lift $w\in G_e$. We can choose the lifts such that $-w$ is the lift of $-\bar{w}$. By Lemma \ref{lemma: enough info}, we can realize $nt_{j,v}d_{j,v}$ as the union of $nt_{j,v}$ pieces that are disk-like, such that each turn from $a_v$ to $a'_v$ representing a triple $(a_v,\bar{w},a'_v)$ has winding number $w$. We can also realize $nx'_v$ as the union of some other pieces at $v$ with turns satisfying the same property. Let $S_v$ be the disjoint union of these pieces at $v$. Then $x\in \mathcal{C}(\bm{r})$ and our choice of the lifts imply that the surface $\sqcup_v S_v$ satisfies the gluing condition (\ref{eqn: gluing condition}) and glues to a simple relative admissible surface $S$ for $c(\bm{r})$ of degree $n$ with $x(S_v)/n=x_v$. Noticing that the number of disk-like pieces in $S$ is no less than $n\Sigma_{j,v}t_{j,v}>n\Sigma_v \kappa_v(x_v)-2n\epsilon$, the estimate of $-\what{\chi}(S)/2n$ easily follows from a computation similar to the one in the first part.
	\end{enumerate}
\end{proof}

Let $\conv(E)$ denote the convex hull of a set $E$ in some vector space. Denote the Minkowski sum of two sets $E$ and $F$ by $E+F\defeq\{e+f\ |\ e\in E\text{ and }f\in F\}$. Note that $\conv(E+F)=\conv(E)+\conv(F)$.

The following lemma is the analog of \cite[Lemma 3.10]{Cal:sss} and has the same proof.
\begin{lemma}[Calegari \cite{Cal:sss}]\label{lemma: cite key}
	The function $\kappa_v$ on $\mathcal{C}_v$ is a non-negative concave homogeneous function which takes value $1$ exactly on the boundary of $\conv(\mathcal{D}(v))+\mathcal{C}_v$ in $\mathcal{C}_v$.
\end{lemma}

It is an important observation in \cite{Chen:sclfp} that both $\conv(\mathcal{D}(v))+\mathcal{C}_v$ and $\kappa_v$ are nice, no matter how complicated $\mathcal{D}(v)$ is.

\begin{lemma}[Chen \cite{Chen:sclfp}]\label{lemma: D'}
	There is a \emph{finite} subset $D'$ of $\mathcal{D}(v)$ such that $D'+\mathcal{C}_v=\mathcal{D}(v)+\mathcal{C}_v$. Consequently, the function $\kappa_v$ is the minimum of \emph{finitely} many rational linear functions.
\end{lemma}
\begin{proof}
	The first assertion follows from \cite[Lemma 4.7]{Chen:sclfp}. Then we have $$\conv(\mathcal{D}(v))+\mathcal{C}_v=\conv(\mathcal{D}(v)+\mathcal{C}_v)=\conv(D'+\mathcal{C}_v)=\conv(D')+\mathcal{C}_v.$$
	Note that $\conv(D')$ is a compact rational polyhedron since $D'$ is a finite set of integer points. Hence $\conv(D')+\mathcal{C}_v$ is a rational polyhedron as it is the sum of two such polyhedra (see the proof of \cite[Theorem 3.5]{Barvinok}). Then $\conv(\mathcal{D}(v))+\mathcal{C}_v=\mathcal{C}_v\cap(\cap_i\{f_i\ge1\})$ for a finite collection of rational linear functions $\{f_i\}$. Combining with Lemma \ref{lemma: cite key}, we have $\kappa_v(x)=\min_i{f_i(x)}$ for all $x\in\mathcal{C}_v$.
\end{proof}

\begin{remark}\label{rmk: realization by surfaces}
	As a consequence, for any rational $x_v\in \mathcal{C}_v$, the supremum in the  definition of $\kappa_v(x_v)$ is a maximum. More precisely, there is an expression $x_v=x'_v+\sum t_j d_j$ with $x'_v\in C_v$, $t_j\in \Q_{\ge0}$ and $d_j\in \mathcal{D}'$, such that $\kappa_v(x_v)=\sum t_j$.
	
	Then the second part of Lemma~\ref{lemma: overcount accurate} holds without $\epsilon$, using the same proof.
	
	Finally, note that \emph{if} all the edge groups are trivial, then disk-like pieces are exactly disks, so $\what{\chi}(S)=\chi(S)$. Hence in this case, for any rational point $x=(x_v)\in \mathcal{C}(\bm{r})$, there is a simple relative admissible surface $S$ for $c(\bm{r})$ of a certain degree $n$ such that $x(S_v)/n=x_v$ and $$\frac{-\chi(S)}{2n}=\sum_v \frac{1}{4}|x_v|-\sum _v\frac{1}{2}\kappa_v(x_v).$$
\end{remark}

\begin{lemma}\label{lemma: compute by linprog}
	The optimization $\min \sum_v \frac{1}{4}|x_v|-\sum_v \frac{1}{2}\kappa_v(x_v)$ among $x=(x_v)\in \mathcal{C}(\bm{r})$ can be computed via linear programming. The minimum is $\scl_{(G,\{G_v\})}(c(\bm{r}))$, which depends piecewise rationally linearly on $\bm{r}\in \ker h$ and is achieved at some rational point in $\mathcal{C}(\bm{r})$.
\end{lemma}
\begin{proof}
	Recall that $|x_v|$ is a rational linear function for $x_v\in \mathcal{C}_v$. Combining with Lemma \ref{lemma: D'}, for each vertex $v$, there are finitely many rational linear functions $f_{j,v}$ such that $|x_v|/4-\kappa_v(x_v)/2=\max_j f_{j,v}(x_v)$. By introducing slack variables $y=(y_v)$, the optimization is equivalent to minimizing $\sum_v y_v$ subject to $y_v\ge f_{j,v}(x_v)$ for all $j,v$ and $x=(x_v)\in \mathcal{C}(\bm{r})$, which is a rational linear programming problem in variables $(x,y)$. The minimum depends piecewise rationally linearly on $\bm{r}\in \ker h$ by Lemma \ref{lemma: space nonempty} and is achieved at a rational point.
	
	The minimum is $\scl_{(G,\{G_v\})}(c(\bm{r}))$ by Lemma \ref{lemma: overcount accurate} and Lemma \ref{lemma: asym prom}.
\end{proof}
\begin{remark}\label{rmk: const}
	The function $\sum_v |x_v|$ is actually a constant $\sum_i r_iA_i$ for $x=(x_v)\in\mathcal{C}(\bm{r})$, where $A_i$ is the number of arcs that the edge spaces cut $\gamma_i$ into. Thus for a fixed chain, the problem comes down to maximizing the number of disk-like pieces.
\end{remark}

Now we return to full generality without assuming $\Gamma$ to be finite or locally finite.
\begin{thm}[Rationality]\label{thm: rational}
	Let $G$ be a graph of groups $\mathcal{G}(\Gamma, \{G_v\}, \{G_e\})$ where
	\begin{enumerate}
		\item $\scl_{G_v}\equiv0$, and
		\item the images of edge groups in each vertex group are central and mutually commensurable.
	\end{enumerate}
	Then $\scl_G$ is piecewise rational linear, and $\scl_G(c)$ can be computed via linear programming for each rational chain $c\in B_1^H(G)$.
\end{thm}
\begin{proof}
	We compute $\scl_G(\sum r_i g_i)$ for an arbitrary finite set of element $\ug=\{g_1,\ldots,g_m\}\subset G$ with $r_i\in\Q_{>0}$ so that the chain $\sum r_i g_i$ is null-homologous in $G$. By Proposition \ref{prop: rel vertex}, we may assume each $g_i$ to be hyperbolic and consider $\scl_{(G,\{G_v\})}(c(\bm{r}))$ with $\bm{r}\in \ker h$ instead. By restriction of domain (Proposition \ref{prop: restriction of domain}), we further assume $\Gamma$ to be finite. Then the result follows from Lemma \ref{lemma: compute by linprog}.
\end{proof}
\begin{remark}
	Theorem \ref{thm: rational} holds with the weaker assumption $\scl_G(c_v)=0$ for all $c_v\in B_1^H(G_v)$ in place of $\scl_{G_v}\equiv0$ and all vertices $v$.
\end{remark}
\begin{remark}
	Let $G$ be a graph of groups where each edge group is $\Z$ and each vertex group $G_v$ is itself a graph of groups as in Theorem \ref{thm: rational} with vanishing $H_2(G_v;\R)$. Using the method in \cite{Cal:sshom}, it follows from Theorem \ref{thm: rational} that the Gromov--Thurston norm on $H_2(G;\R)$ has a rational polyhedral unit ball and can be computed via linear programming.
\end{remark}

It is more natural to state the rationality theorem in the context of relative scl, for which we do not need to assume the vertex groups to have vanishing scl.
\begin{thm}[Rationality for relative scl]
	Let $G$ be a graph of groups $\mathcal{G}(\Gamma, \{G_v\}, \{G_e\})$, where the images of edge groups in each vertex group are central and mutually commensurable. 
	Then $\scl_{(G,\{G_v\})}$ is piecewise rational linear, and $\scl_{(G,\{G_v\})}(c)$ can be computed via linear programming for each rational chain $c\in B_1^H(G,\mathcal{G})$. 
	
	Moreover, if all the edge groups are trivial, then for each rational chain $c\in B_1^H(G,\mathcal{G})$ there is a simple relative admissible surface $S$ that is extremal, i.e.
	$$\frac{-\chi(S)}{2n(S)}=\scl_{(G,\{G_v\})}(c).$$
\end{thm}
\begin{proof}
	For the first part, the proof is the same as that of Theorem~\ref{thm: rational}, by noting that the restriction of domain works for relative scl without the vanishing assumption.
	
	For the second part, the rational point achieving the optimal value of the linear programming can be realized as an extremal simple relative admissible surface by Remark~\ref{rmk: realization by surfaces} when edge groups are trivial.
\end{proof}

\begin{cor}
	For a free product $G=\star_\lambda G_\lambda$, any rational chain $c\in G$ has a simple relative admissible surface $S$ that achieves $\scl_{(G,\{G_\lambda\}_\lambda)}(c)$.
\end{cor}

\section{Scl in Baumslag--Solitar groups}\label{sec: sclBS}
\subsection{Basic setups}\label{subsec: sclBS setup}
In this section we focus on scl in Baumslag--Solitar groups $G=\BS(M,L)=\left<a,t\ |\ a^M =ta^Lt^{-1}\right>$ with integers $M,L\neq 0$. We are not interested in the case where $|M|=1$ or $|L|=1$ since $\BS(M,L)$ is solvable and $\scl_{\BS(M,L)}\equiv0$ in such cases. Most of the results and tools are applicable for any graphs of groups with abelian vertex groups, but we will not pursue such generalizations. Let $d\defeq\gcd(|M|,|L|)$, $m\defeq M/d$ and $\ell\defeq L/d$. Denote by $h:G\to\Z$ the homomorphism given by $h(a)=0$ and $h(t)=1$. An element $g$ is \emph{$t$-balanced} if $h(g)=0$.

We denote the only vertex and edge by $v$ and $\{e,\bar{e}\}$ respectively, where $e$ is oriented to represent the generator $t$. See Figure \ref{fig: BS}. Notation from Section \ref{sec: scl by linprog} will be used.

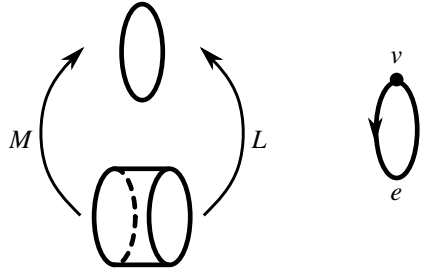
\begin{figure}
	\labellist
	\small 
	\pinlabel $M$ at -10 70
	\pinlabel $L$ at 120 70
	\pinlabel $v$ at 195 115
	\pinlabel $e$ at 195 40
	
	\endlabellist
	\centering
	\includegraphics[scale=0.7]{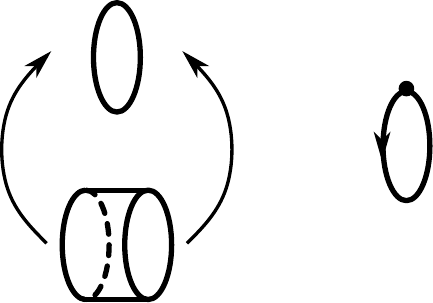}
	\caption{The graph of spaces associated to $\BS(M,L)$ and the underlying graph.}\label{fig: BS}
\end{figure}

Let $g=a^{p_1}t^{\epsilon_1}\ldots a^{p_n}t^{\epsilon_n}$ be a cyclically reduced word, where $\epsilon_i=\pm 1$ for all $i$. Then $g$ is represented by a tight loop $\gamma$ in $X_G$ cut into $n$ arcs $A_v=\{a_i\ |\ 1\le i\le n\}$ where $a_i$ has winding number $w(a_i)=p_i$. Note that we have the transition map $\tau_e: |M|\Z\to|L|\Z$ with $\tau_e(x)=-Lx/M$. For each $1\le i\le n$, let $\mu_i\defeq\max_{0\le k\le n} \sum_{j=1}^{k}\epsilon_{i+j}$ and $\lambda_i\defeq-\min_{0\le k\le n} \sum_{j=1}^{k}\epsilon_{i+j}$, where indices are taken mod $n$ and the summation is $0$ when $k=0$. It is straightforward to see that $\Dom\tau_{P(a_i)}=dm^{\mu_i}\ell^{\lambda_i}\Z$, $\Im\tau_{P(a_i)}=dm^{\mu_i-h(g)}\ell^{\lambda_i+h(g)}\Z$, and $\tau_{P(a_i)}(x)=(-1)^{h(g)} \ell^{h(g)}x/m^{h(g)}$. Example \ref{ex: index compute} (with $X=dm^{\mu_i}\ell^{\lambda_i}$ and $Y=\pm dm^{\mu_i-h(g)}\ell^{\lambda_i+h(g)}$) shows that, letting
\begin{equation}\label{eqn: choose W_0}
W_0(a_i)\defeq\left\{\begin{array}{cc}
dm^{\mu_i-|h(g)|}\ell^{\lambda_i}\Z & \text{if } h(g)\ge0,\\
dm^{\mu_i}\ell^{\lambda_i-|h(g)|}\Z & \text{if } h(g)\le0,
\end{array}\right.
\end{equation}
we have $W_0(a_i)\subset \Dom\tau_{P(a_i)}^p+\Im\tau_{P(a_i)}^q$ for all $p,q\ge0$.

In the sequel, we will use the following two different setups.
\begin{enumerate}
	\item \textbf{Setup 1}: Let $W(a_i)=W_0(a_i)$ for all $i$.
	\item \textbf{Setup 2}: Let $W(a_i)=\cap_j W_0(a_j)$ for all $i$.
\end{enumerate}
In both setups, the group $D_v$ defined as $\cap_i W(a_i)$ in Section \ref{sec: scl by linprog} equals $\cap_i W_0(a_i)$, and thus the choice of setups does not affect our encoding or the space $\mathcal{C}_v$. The only difference is that we have more disk-like pieces or vectors in Setup 1 than Setup 2. We will explicitly state our choice of setups whenever the discussion depends on it.

For an explicit formula of $D_v$, let $\lambda\defeq\max_i \lambda_i$ and $\mu\defeq\max_i \mu_i$. Since $h(g)=\sum \epsilon_i$, it is easy to observe that $\mu-|h(g)|=\lambda\ge0$ when $h(g)\ge0$ and $\lambda-|h(g)|=\mu\ge0$ when $h(g)\le0$. In any case, using the formula (\ref{eqn: choose W_0}) we have 
$$
D_v=dm^{\rho(g)}\ell^{\rho(g)}\Z,
$$
where $\rho(g)\defeq \min(\mu,\lambda)$, which we call the \emph{complexity} of $g$.

When $h(g)=0$, this can be easily seen geometrically. The infinite cyclic cover $\wtilde{X}_G$ of $X_G$ corresponding to $\ker h$ has a $\Z$-action by translation with fundamental domains projecting homeomorphically to the thickened vertex space $N(X_v)$. Since $h(g)=0$, the tight loop $\gamma$ representing $g$ lifts to a loop $\tilde{\gamma}$ on $\wtilde{X}_G$, and $\rho(g)+1$ is the number of fundamental domains that $\tilde{\gamma}$ intersects. In particular, when $h(g)=0$, the element $g$ is $t$-alternating if and only if the complexity $\rho(g)=1$. 

More generally, for a chain $c=\sum r_i g_i$ with each $r_i\neq0$, define its complexity $\rho(c)\defeq\max_i \rho(g_i)$. Then 
\begin{equation}\label{eqn: choose D_v}
D_v=dm^{\rho(c)}\ell^{\rho(c)}\Z,
\end{equation}
and $\rho(c)$ controls the amount of information we need to encode. Denote $|D_v|\defeq d|m|^{\rho(c)}|\ell|^{\rho(c)}$.

Using the notation from Section \ref{sec: scl by linprog}, we have $D_e=m^{\rho(c)}\ell^{\rho(c)}\Z$ and $W_e=\Z/D_e$ in both setups. The fact that $D_e$ does not depend on $d$ is important in Theorem \ref{thm: surgery}.

To better understand integer points in $\mathcal{C}_v$ and disk-like vectors, consider a directed graph $Y$ with vertex set $A_v$, where each oriented edge from $a_i$ to $a_j$ corresponds to a triple $(a_i,\bar{w},a_j)\in T_v$. See Proposition \ref{prop: eg3} and Figure \ref{fig: Y} for an example. Then each vector $x\in\mathcal{C}_v$ assigns non-negative weights to edges in $Y$. 

Define the support $\supp(x)$ to be the subgraph containing edges with positive weights. Then $\supp(x)$ is a union of positively oriented cycles in $Y$. Note that an integer point $x\in\mathcal{C}_v$ can be written as $x(C)$ for some piece $C$ if and only if $\supp(x)$ is connected. If two pieces $C$ and $C'$ are encoded by the same vector $x$, then their winding numbers $w(C)$ and $w(C')$ are congruent mod $D_v$ since the vertex group is abelian. Moreover, fixing a lift $\wtilde{w}\in \Z$ of each $\bar{w}\in W_e$, we can compute this winding number $w(x)$ as a linear function on $\R^{T_v}$ determined by
\begin{equation}\label{eqn: w}
w(a_i,\bar{w},a_j)\defeq w(a_j)+\iota(\wtilde{w})
\end{equation}
where $\iota=t_e$ if $e_{out}(a_i)=\bar{e}$ (ie $a_i$ leaves $v$ by following $\bar{e}$, see Figure \ref{fig: pair}) and $\iota=o_e$ if $e_{out}(a_i)=e$. Then $w(x)$ depends on the choice of lifts but $w(x)$ mod $D_v$ does not. In particular, it makes sense to discuss whether $w(x)\in W(a_i)$ since $D_v\subset W(a_i)$.

Then in both setups, an integer point $x\in\mathcal{C}_v$ is disk-like if and only if $\supp(x)$ is connected and $w(x)\in W(a_i)$ for some $a_i$ in $\supp(x)$. Since $W(a_i)=D_v$ for all $i$ in Setup 2, the criterion is simply that $\supp(x)$ is connected and $w(x)\in D_v$. In contrast, in Setup 1, we have more disk-like pieces, which makes it easier to construct simple relative admissible surfaces using disk-like pieces.

We use $\mathcal{C}(c)$ instead of $\mathcal{C}(\bm{r})$ to denote the polyhedron encoding normalized simple relative admissible surfaces for $c$, as we will not consider families of chains with varying $\bm{r}$.

\subsection{Extremal surfaces}
The goal of this subsection is to obtain a criterion for the existence of extremal surfaces in the case of Baumslag--Solitar groups. We accomplish this by strengthening results in Section \ref{sec: asym prom} to analyze when the asymptotic approximation can terminate at a finite stage. We will use the notion of transition maps and adjustment maps from Section \ref{sec: asym prom}.


Let $c=\sum r_i g_i$ be a rational chain as in the previous subsection where $r_i\in \Q_{>0}$ and each $g_i\in \ug$ is a hyperbolic element represented by a tight loop $\gamma_i\in\uga$. Throughout this subsection, we use Setup 2 where $W(a)=D_v$ for all $a\in A_v$. 

\begin{lemma}\label{lemma: all disk-like}
	Let $S$ be a simple relative admissible surface of degree $n$ with 
	$$-\frac{\what{\chi}(S)}{2n}=\scl_{\BS(M,L)}(c).$$
	Then $S$ consists of disk-like pieces only.
\end{lemma}
\begin{proof}
	Suppose there is a piece $C$ not disk-like. Take a piece $\wtilde{C}$ where the boundary is a degree $|D_v|$ cover of $\bdry C$. Then $\wtilde{C}$ is disk-like. Moreover, $\wtilde{C}$ and $|D_v|$ copies of every piece in $S$ other than $C$ together satisfy the gluing condition and form a new admissible surface $\wtilde{S}$ of degree $n|D_v|$, such that
	$$-\frac{\what{\chi}(\wtilde{S})}{2n|D_v|}=-\frac{\what{\chi}(S)}{2n}-\frac{1}{2n|D_v|}<\scl_{\BS(M,L)}(c).$$
	This is absurd (in view of Lemma \ref{lemma: asym prom}) since $S$ is optimal. 
\end{proof}

The optimal solution to the linear programming in Lemma \ref{lemma: compute by linprog} to compute $\scl_{\BS(M,L)}(c)$ provides a surface $S$ satisfying the assumption in Lemma \ref{lemma: all disk-like}. To obtain an extremal surface we need to adjust disk-like pieces into genuine disks by trivializing their winding numbers. 

This can be easily understood when $M=L$. The winding number $w(C)$ is divisible by $|D_v|=|M|=|L|$ since $C$ is disk-like, which can be eliminated and added to $w(C')$ by an adjustment where $C'$ is any nearby piece. Eventually we can concentrate all the winding numbers at a single piece (assuming $S$ is connected). The last winding number is simply the sum of all original winding numbers which remains invariant in the adjustment process. Thus we can make $S$ into an extremal surface if and only if this sum vanishes. A similar argument works when $M=-L$ where one needs to consider an ``alternating'' sum instead.

In what follows we assume $M\neq \pm L$, which is the hard case. If we were allowed to treat winding numbers as rational numbers, the same process as above can be done, and the problem comes down to the vanishing of an obstruction number. The main bulk of this subsection is to use suitable finite covers and stability (Lemma \ref{lemma: stability}) to make our situation as good as working over $\Q$.

In the sequel, we always suppose $S$ to be a simple relative admissible surface consisting of disk-like pieces only. Recall that there is a finite graph $\Gamma_S$ encoding how pieces glue up to form $S$, which has a graph homomorphism $\pi$ to the underlying graph $\Gamma$. In the case of Baumslag--Solitar groups, $\Gamma$ is a circle with one edge $e$ (shown in Figure \ref{fig: BS}), inducing an orientation on $\Gamma_S$. Recall from Section \ref{sec: asym prom} that every nontrivial oriented path $P$ in $\Gamma_S$ has an adjustment map $\alpha(P)$, which is of the form $\alpha(P)(x)=(-\ell/m)^{h(P)} x$ for every $x\in \Dom\alpha(P)\subset d\Z$, where 
$$ h(P)=\#(\text{positively oriented edges on P}) - \#(\text{negatively oriented edges on P}).$$

Similarly, for each oriented cycle $\omega$ in $\Gamma_S$, define $h(\omega)$ to be $h(P)$, where we consider $\omega$ as an oriented path $P$ by choosing an arbitrary base point.

\begin{definition}
	An oriented cycle $\omega$ in $\Gamma_S$ is \emph{imbalanced} if the number $h(\omega)$ above is non-zero (so that $(-\ell/m)^{h(\omega)}\neq 1$ since $M\neq \pm L$). Otherwise, we say $\omega$ is \emph{balanced}.
\end{definition}

We investigate whether all disk-like pieces can be made into genuine disks in each component $\Sigma$ of $S$. Denote the corresponding component of $\Gamma_S$ by $\Gamma_\Sigma$. We say $\Sigma$ is \emph{disk-only} if it consists of genuine disks.

The following lemmas show that imbalanced cycles are useful to eliminate winding numbers of pieces.

\begin{lemma}\label{lemma: boom via imba}
	Suppose $\Gamma_\Sigma$ contains an imbalanced oriented cycle $\omega$, on which sits a piece $C$ as the base point. If $w(C)\in \Dom \alpha(\omega)^p+\Im \alpha(\omega)^q$ for all $p,q\ge0$, then for any $N\in \Z_+$, there is an adjustment of $\Sigma$ supported on $C$, after which $w(C)$ is divisible by $dm^N\ell^N$.
\end{lemma}
\begin{proof}
	Without loss of generality, assume $h(\omega)>0$. Then $\Dom \alpha(\omega)^p\subset dm^p\Z$ and $\Im \alpha(\omega)^q\subset d\ell^q\Z$. By the assumption, there are $a,b\in \Z$ and $u\in\Dom \alpha(\omega)^{2N}$, $v\in\Im \alpha(\omega)^{2N}$, such that $w(C)=au+bv$. Then $\alpha(\omega)^{N}(au)\in \Dom \alpha(\omega)^{N}\cap \Im \alpha(\omega)^{N}$ is divisible by $dm^N\ell^N$, so is $\alpha(\omega)^{-N}(bv)$.
\end{proof}

\begin{lemma}\label{lemma: lift balanced and merge}
	Suppose $\Gamma_\Sigma$ contains a balanced oriented cycle $\omega$, on which sits a piece $C$ as the base point. Let $\wtilde{S}$ be a finite cover of $S$ and $\wtilde{\omega}$ be a degree $k$ lift of $\omega$, where the preimages of $C$ are denoted as $\wtilde{C}_1,\ldots,\wtilde{C}_k$. If $w(C)\in \Dom \alpha(\omega)^p+\Im \alpha(\omega)^q$ for all $p,q\ge0$, then there is an adjustment of $\wtilde{\Sigma}$ supported on $\{\wtilde{C}_1,\ldots,\wtilde{C}_k\}$, after which $w(\wtilde{C}_1)=kw(C)$ and $w(\wtilde{C}_j)=0$ for all $k>1$.
\end{lemma}
\begin{proof}
	We eliminate $w(\wtilde{C}_j)$ for $j>1$ as in the proof of Lemma \ref{lemma: lift and kill}. The fact that $\omega$ is balanced implies $w(\wtilde{C}_1)=kw(C)$ after the adjustment.
\end{proof}

\begin{lemma}\label{lemma: kill last one}
	Suppose a component $\Sigma$ contains two distinct embedded cycles $\omega$ and $\omega_{im}$ where $\omega_{im}$ is imbalanced. If all pieces of $\Sigma$ are genuine disks except for one piece $C$, then there is a finite cover of $\Sigma$ where all pieces are genuine disks after a suitable adjustment, assuming $w(C)$ is divisible by $dm^N\ell^N$ for some $N\in\Z_+$ depending on $\Sigma$.
\end{lemma}
\begin{proof}
	It suffices to prove the case where $C$ lies on the cycle $\omega$. For the general case, we can increase $N$ by the diameter of $\Gamma_\Sigma$, and move $w(C)$ to a piece $C'$ on $\omega$ by an adjustment supported on the two pieces, after which $w(C')$ is divisible by $dm^N\ell^N$ and $w(C)=0$. When $C$ lies on $\omega$, fix an embedded oriented path $P$ from $C$ to $\omega_{im}$. Since $\omega$ and $\omega_{im}$ are embedded and distinct, for any $k\in \Z_+$, there is a degree $k$ normal cover $\wtilde{\Sigma}$ of $\Sigma$ containing cycles $\wtilde{\omega}_{im,j}$ and $\wtilde{\omega}$, $1\le j\le k$, where $\wtilde{\omega}$ covers $\omega$ with degree $k$ and each $\wtilde{\omega}_{im,j}$ projects homeomorphically to $\omega_{im}$. See Figure \ref{fig: liftkill}. Denote the lifts of $C$ according to the cyclic order on $\wtilde{\omega}$ by $\wtilde{C}_j$, $1\le j\le k$. Let $\wtilde{P}_j$ be the lift of $P$ connecting $\wtilde{C}_j$ and $\wtilde{\omega}_{im,j}$.
	
	\begin{figure}
		\labellist
		\small 
		\pinlabel $P$ at -5 130
		\pinlabel $C$ at 7 100
		\pinlabel $\omega_{im}$ at 10 180
		\pinlabel $\omega$ at 110 100
		
		\pinlabel $\wtilde{C}_1$ at 185 105
		\pinlabel $\wtilde{P}_1$ at 170 130
		\pinlabel $\wtilde{\omega}_{im,1}$ at 188 190
		
		\pinlabel $\wtilde{C}_2$ at 290 5
		\pinlabel $\wtilde{P}_2$ at 270 -5
		\pinlabel $\wtilde{\omega}_{im,2}$ at 195 10
		
		\pinlabel $\wtilde{C}_3$ at 385 105
		\pinlabel $\wtilde{P}_3$ at 400 75
		\pinlabel $\wtilde{\omega}_{im,3}$ at 405 40
		
		\pinlabel $\wtilde{C}_4$ at 280 210
		\pinlabel $\wtilde{P}_4$ at 300 220
		\pinlabel $\wtilde{\omega}_{im,4}$ at 350 220
		
		\pinlabel $\wtilde{\omega}$ at 195 70
		
		\endlabellist
		\centering
		\includegraphics[scale=0.8]{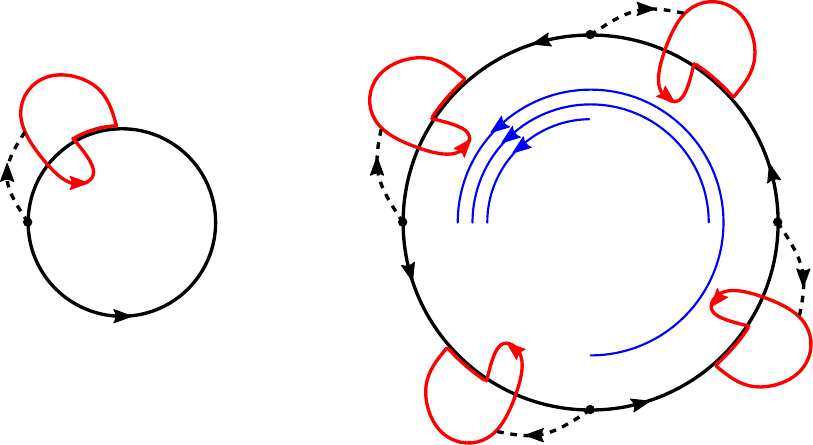}
		\vspace{5pt}
		\caption{The cover unwrapping $\omega$ when $k=4$ with the blue concentric arcs indicating our adjustment moving winding numbers of $\wtilde{C}_j$ together in the case where $\omega$ is balanced}\label{fig: liftkill}
	\end{figure}
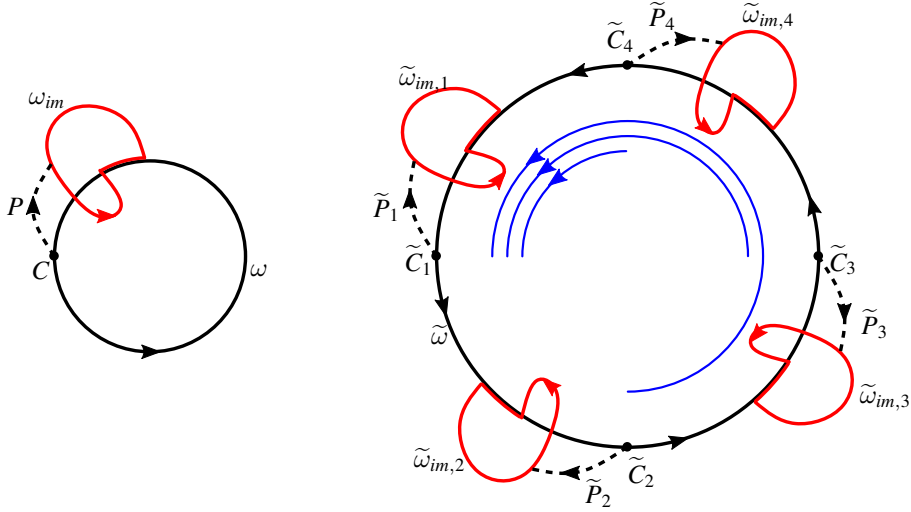
	
	Fix an integer $k_0\in\Dom\alpha(P\omega_{im}\overline{P})$ of the form $dm^{N_0}\ell^{N_0}$ with $N_0\in \Z_+$, where $\overline{P}$ is $P$ reversed and $P\omega_{im}\overline{P}$ is the concatenation. Note that any $\wtilde{C}_j$ with $w(\wtilde{C}_j)$ divisible by $w_0\defeq [(-\ell/m)^{h(\omega_{im})}-1] k_0\neq0$ can be made into a genuine disk without affecting any other pieces on $\wtilde{\Sigma}$ since
	$$w_0=[(-\ell/m)^{h(\omega_{im})}-1]k_0=\alpha(P\omega_{im}\overline{P})k_0-k_0=\alpha(\wtilde{P}_j\wtilde{\omega}_{im,j}\overline{\wtilde{P}}_j)k_0-k_0.$$
	
	Our strategy is to choose a suitable covering degree $k$ and make adjustment so that each $w(\wtilde{C}_j)$ is divisible by $w_0$.
	
	If $\omega$ is balanced, let $k=|w_0|$. By stability (Lemma \ref{lemma: stability}), we can apply the adjustment in Lemma \ref{lemma: lift balanced and merge} by choosing $N$ large. After the adjustment, the only non-trivial winding number in $\wtilde{\Sigma}$ is $w(\wtilde{C}_1)=kw(C)$, which is divisible by $w_0$ and thus can be eliminated.
	
	If $\omega$ is imbalanced, let $k=|h(\omega_{im})w_0|$ and $s=\mathrm{sign}(h(\omega))$. We divide $\{\wtilde{C}_1,\ldots,\wtilde{C}_k\}$ into $|h(\omega_{im})|$ groups to carry out the adjustment. For each $1\le i\le |h(\omega_{im})|$, there is an adjustment supported on $\{\wtilde{C}_{i+jsh(\omega_{im})}\ |\ 1\le j\le |w_0|\}$ (subscripts are taken mod $k$) along subpaths of $\wtilde{\omega}$, such that $w(\wtilde{C}_{i+jsh(\omega_{im})})=0$ for all $j$ except
	\begin{equation}\label{eqn: compute wind}
	w(\wtilde{C}_i)=\sum_{j=0}^{|w_0|-1} \alpha(\omega)^{sjh(\omega_{im})}w(C)=|w_0|w(C)+\sum_{j=0}^{|w_0|-1}[\alpha(\omega)^{sjh(\omega_{im})}w(C)-w(C)].
	\end{equation}
	This can be done provided $w(C)\in \Dom\alpha(\omega)^{sjh(\omega_{im})}$ for all $1\le j\le |w_0|$ by choosing $N$ large. Make $N$ further larger so that $w(C)\in m^{N_0}\ell^{N_0}\Dom\alpha(\omega)^{sjh(\omega_{im})}$ for all $1\le j\le |w_0|$. Then $(x-1)^\lambda w(C)$ is divisible by $w_0=(x-1)dm^{N_0}\ell^{N_0}$ for all $1\le \lambda\le |w_0 h(\omega)|$, where $x=(-\ell/m)^{h(\omega_{im})}$. Therefore, for any $j\ge0$,
	$$\alpha(\omega)^{sjh(\omega_{im})}w(C)-w(C)=[x^{j|h(\omega)|}-1]w(C)=\sum_{\lambda=1}^{j|h(w)|}
	\begin{pmatrix}
	j|h(\omega)|\\
	\lambda
	\end{pmatrix}
	(x-1)^{\lambda}w(C)$$
	is an integer multiple of $w_0$ since each term in the summation is. Thus for $N$ large, each term in the last summation of equation (\ref{eqn: compute wind}) is divisible by $w_0$. Hence $w(\wtilde{C}_i)$ is also divisible by $w_0$ and can be eliminated.
\end{proof}

\begin{definition}\label{def: reduced chains}
	Two hyperbolic elements $g$ and $h$ are \emph{pseudo-inverses} if $\scl_{\BS(M,L)}(g+h)=0$. A rational chain $c=\sum r_i g_i$ with $r_i\in \Q_{>0}$ and each $g_i\in\ug$ hyperbolic is \emph{reduced} if no $g_i^p$ and $g_j^q$ are pseudo-inverses for any $p,q\in \Z_+$ and any $g_i,g_j\in \ug$ (possibly $i=j$).
\end{definition}

\begin{lemma}\label{lemma: reduced chains}
	Let $g,h\in \BS(M,L)$ be hyperbolic elements with positive $\scl$. Let $\rho=\rho(g+h)$ be the complexity and $|D_v|=d|m|^{\rho}|\ell|^{\rho}$ as in (\ref{eqn: choose D_v}). Then the following are equivalent:
	\begin{enumerate}
		\item $g$ and $h$ are pseudo-inverses;\label{item: definition}
		\item For some $k\in \Z_+$ and $N\in \Z$, there is a thrice punctured sphere $S$ in $\BS(M,L)$ in simple normal form with $\Gamma_S$ being a cycle such that $S$ bounds conjugacy classes $g^k$, $h^k$ and $a^{N|D_v|}$ (see Figure \ref{fig: pseudoinv});\label{item: geometric}
		\item For some $k\in \Z_+$ and $N\in \Z$, some cyclic conjugates of $g^k$ and $h^k$ can be written as cyclically reduced words
		$$a^{n_1}t^{\epsilon_1}\ldots a^{n_s}t^{\epsilon_s}, \quad \text{and}\quad t^{-\epsilon_s}a^{-n_s}\ldots a^{-n_2}t^{\epsilon_1}a^{-n_1+N|D_v|},$$
		respectively, where all $\epsilon_{i}=\pm 1$.\label{item: combinatorial}
	\end{enumerate} 
\end{lemma}
\begin{figure}
	\labellist
	\small 
	\pinlabel $h^k$ at 0 130
	\pinlabel $g^k$ at 45 110
	\pinlabel $a^{N|D_v|}$ at 115 15
	\endlabellist
	\centering
	\includegraphics[scale=0.7]{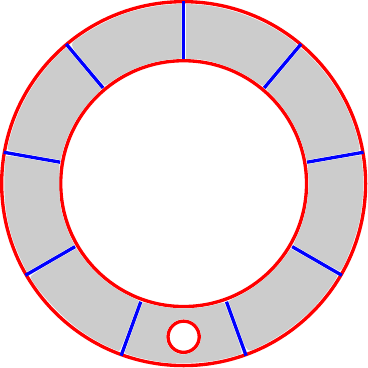}
	\vspace{5pt}
	\caption{Pseudo-inverses $g$ and $h$ have powers almost cobounding an annulus.}\label{fig: pseudoinv}
\end{figure}
\begin{proof}
	Obviously (\ref{item: geometric}) and (\ref{item: combinatorial}) are equivalent. Now suppose (\ref{item: geometric}) holds. Notice that the only non-disk piece in the thrice punctured sphere $S$   is disk-like. Thus $\what{\chi}(S)=0$ and hence $\scl(g+h)=\scl(g^k+h^k)/k=0$ by Lemma \ref{lemma: asym prom}. Therefore (\ref{item: definition}) holds.
	
	Conversely, suppose $g$ and $h$ are pseudo-inverses. By Theorem \ref{thm: rational}, the linear programming produces a simple relative admissible surface $S$ with $\what{\chi}(S)=0$. By Lemma \ref{lemma: all disk-like}, all pieces of $S$ are disk-like, thus $\chi(\Gamma_S)=\what{\chi}(S)=0$. Since each vertex of $\Gamma_S$ has valence at least $2$ (Lemma \ref{lemma: backtrack}), it follows that each component of $\Gamma_S$ is a cycle. Since both $g$ and $h$ have positive $\scl$ and $\scl(g+h)=0$, each component must bound $g^k+h^k$ for some $k\in\Z$ relative to the vertex group. By Lemma \ref{lemma: boom via imba} and Lemma \ref{lemma: lift balanced and merge}, after taking a suitable finite cover of a component and an adjustment, all but at most one disk-like piece become genuine disks, as depicted in Figure \ref{fig: pseudoinv}. Thus (\ref{item: geometric}) follows.
\end{proof}

The last part of the proof without much change implies the following characterization of chains with trivial scl in $\BS(M,L)$.
\begin{prop}\label{prop: reduced scl>0}
	Let $c=\sum r_i g_i$ be a rational chain with $r_i\in \Q_{>0}$ and each $g_i\in\ug$ hyperbolic. Then $\scl_{\BS(M,L)}(c)=0$ if and only if $c$ can be rewritten as $\sum r'_j (h_j+h'_j)$ where $r'_j\in \Q_{>0}$, $h_j$ and  $h'_j$ are pseudo-inverses and each of the form $g_i^k$ for some $g_i\in\ug$ and $k\in\Z_+$. In particular, reduced chains have positive $\scl$.
\end{prop}

For simplicity, we only consider reduced chains in the rest of this subsection.

We first characterize the existence of imbalanced cycles. In the sequel, let $\wtilde{\Gamma}$ be the universal cover of $\Gamma$, considered as a bi-infinite line with vertex set $\Z$ and oriented edges $k\to k+1$ for each $k\in \Z$.
\begin{lemma}\label{lemma: char imba cycles}
	For a component $\Sigma$ of $S$, the following are equivalent:
	\begin{enumerate}
		\item $\Gamma_\Sigma$ contains an imbalanced cycle;
		\item $\Gamma_\Sigma$ contains an embedded imbalanced cycle;
		\item There is no graph homomorphism $\phi:\Gamma_\Sigma\to \wtilde{\Gamma}$ preserving edge orientations.
	\end{enumerate}
\end{lemma}
\begin{proof}
	The existence of such a map $\phi$ implies that each cycle is balanced. Breaking an imbalanced cycle into embedded ones, we obtain at least one embedded imbalanced cycle. Finally, we can construct $\phi$ inductively by first mapping a chosen vertex to an arbitrary vertex of $\wtilde{\Gamma}$ and then extending the map to adjacent vertices preserving edge orientations. Such a construction ends up with a well-defined $\phi$ if all embedded cycles in $\Gamma_\Sigma$ are balanced.
\end{proof}

We show a component $\Sigma$ of $S$ can be made disk-only after taking a finite cover if it contains an imbalanced cycle $\omega$. This is achieved by taking suitable covers and adjustment so that we land in the situation of Lemma \ref{lemma: kill last one}. We need the following lemma on finite covers of graphs.

\begin{lemma}\label{lemma: lift with controlled degree}
	Let $\omega_1,\ldots, \omega_n$ be non-trivial cycles in a connected graph $\Lambda$. Then for any $N\in \Z_+$, there is a connected finite normal cover $\wtilde{\Lambda}\to \Lambda$ such that each lift $\wtilde{\omega}_i\to \omega_i$ has degree divisible by $N$, for any $1\le i\le n$.
\end{lemma}
\begin{proof}
	Let $w_1=x_{i_k}^{e_k}\ldots x_{i_1}^{e_1}$ be a cyclically reduced word representing the conjugacy class of $\omega_1$ in the free group $\pi_1(\Lambda)$ with generators $\{x_1,\ldots x_r\}$, where each $e_j=\pm 1$. Construct permutations $\varphi(x_1),\ldots,\varphi(x_r)$ on the set $\mathcal{E}=\{1,\ldots, Nk\}$ such that
	$$\varphi(x_{i_j})(qk+j)=\left\{\begin{array}{rl}
	qk+j+1& \text{if }e_j=1\\
	qk+j-1& \text{if }e_j=-1\\
	\end{array}\right.,$$
	for $q=0,\ldots,N-1$ and $j=1,\ldots,k$, where elements in $\mathcal{E}$ are represented mod $Nk$. For each generator $x_i$, the conditions imposed on each permutation $\varphi(x_i)$ above are compatible (well-defined and injective). Thus such permutations exist and induce an action $\varphi$ of $\pi_1(\Lambda)$ on $\mathcal{E}$ by permutations. By construction, we have $\varphi(w_1)(qk+1)=(q+1)k+1$ for all $q=0,\ldots, N-1$, which gives an orbit of length $N$. Let $N_1$ be the kernel of $\varphi$, and then $w_1$ has order divisible by $N$ in the quotient $\pi_1(\Lambda)/N_1$. Similarly construct normal subgroups $N_i$ for each cycle $\omega_i$. Then the normal cover corresponding to $\cap_i N_i$ has the desired property.
\end{proof}

Recall that each boundary component $B$ of $S$ provides an oriented non-backtracking (by Lemma \ref{lemma: backtrack}) cycle $\omega_B \subset \Gamma_S$ with $h(\omega_B)=kh(g_i)$ if $B$ wraps $k$ times around $\gamma_i$. In particular, $\omega_B$ is balanced if and only if $g_i$ is $t$-balanced. Since Setup 2 is used, for any disk-like piece $C$ on $\omega_B$, we have $w(C)\in \Dom \alpha(\omega_B)^p+\Im \alpha(\omega_B)^q$ for all $p,q\ge0$.

\begin{lemma}\label{lemma: lift and concentrate, imba case}
	If a component $\Sigma$ of $S$ contains an imbalanced cycle $\omega$, then for any $N\in \Z_+$, there is an adjustment on a connected finite cover $\wtilde{\Sigma}$ of $\Sigma$, after which all pieces of $\wtilde{\Sigma}$ have winding number $0$, except for one piece $\wtilde{C}$ which has $w(\wtilde{C})$ divisible by $dm^N\ell^N$ and sits on a lift of $\omega$.
\end{lemma}
\begin{proof}
	Fix a piece $C$ on $\omega$. For any other piece $C'$, choose a path $P(C')$ on $\Gamma_\Sigma$ from $C'$ to $C$. Then there is some $N(C')\in\Z_+$ such that $n(C')=d|m|^{N(C')}|\ell|^{N(C')}$ is in the domain of $\alpha(P(C'))$ and its image lies in $D_v$. Choose also a boundary component $B(C')$ passing through $C'$.
	
	Now for any $C'$ with $\omega_{B(C')}$ imbalanced, by stability and Lemma \ref{lemma: boom via imba}, we may assume $w(C')$ to be divisible by $n(C')$. Consider the finite collection of all (non-backtracking) cycles $\omega_{B(C')}$ that are \emph{balanced} with corresponding integers $n(C')$. Applying Lemma \ref{lemma: lift with controlled degree}, we obtain a cover $\wtilde{\Sigma}$ of $\Sigma$ corresponding to a finite normal cover $\Gamma_{\wtilde{\Sigma}}\to\Gamma_\Sigma$. Then for any $C'$ with $\omega_{B(C')}$ balanced, by Lemma \ref{lemma: lift balanced and merge}, up to an adjustment, the winding numbers of the preimages of $C'$ on each lift $\wtilde{\omega}_{B(C')}$ of $\omega_{B(C')}$ concentrate on one piece $\wtilde{C}'$ with $w(\wtilde{C}')$ divisible by $n(C')$. Hence for any piece $C'$ other than $C$, every preimage $\wtilde{C'}$ on $\wtilde{\Sigma}$ has $w(\wtilde{C'})$ divisible by $n(C')$. Now a lift $\wtilde{P(C')}$ of $P(C)$ connects $\wtilde{C}'$ and some preimage $\wtilde{C}$ of $C$, by our choice of $n(C')$, the winding number at $\wtilde{C}'$ can be eliminated by an adjustment at the cost of changing $w(\wtilde{C})$ by a multiple of $|D_v|$.
	
	In this way, all pieces of $\wtilde{\Sigma}$ are either genuine disks or a preimage of $C$ that is disk-like. Since each lift of the cycle $\omega$ must be imbalanced, further applying Lemma \ref{lemma: boom via imba}, we may assume preimages of $C$ to have large winding numbers that can be merged into a single preimage $\wtilde{C}$ with $w(\wtilde{C})$ divisible by $dm^N\ell^N$ after an adjustment.
\end{proof}

\begin{lemma}\label{lemma: ext surf imba case}
	For $c$ reduced, if a component $\Sigma$ of $S$ contains an imbalanced cycle $\omega$, then a connected finite cover of $\Sigma$ is disk-only after a suitable adjustment.
\end{lemma}
\begin{proof}
	Let $\wtilde{\Sigma}$ be the connected cover as in Lemma \ref{lemma: lift and concentrate, imba case}, which almost consists of genuine disks except for one piece $\wtilde{C}$ with $w(\wtilde{C})$ divisible by $dm^N\ell^N$, where $N$ can be chosen arbitrarily large without changing the surface $\wtilde{\Sigma}$ by Lemma \ref{lemma: boom via imba}, since $\wtilde{C}$ sits on a lift of $\omega$ which is imbalanced. By Lemma \ref{lemma: char imba cycles}, $\wtilde{\Sigma}$ contains an embedded imbalanced cycle $\omega_{im}$. Since $c$ is reduced, we have $\chi(\Gamma_S)=\what{\chi}(S)<0$ by Proposition \ref{prop: reduced scl>0}, so there must be an embedded cycle distinct from $\omega_{im}$, and thus the conclusion follows from Lemma \ref{lemma: kill last one}.
\end{proof}

\begin{cor}
	If no $g_i\in \ug$ is $t$-balanced, then the reduced chain $c$ has an extremal surface.
\end{cor}
\begin{proof}
	Let $S$ be a simple relative admissible surface obtained from the optimal solution to the linear programming problem as in Theorem \ref{thm: rational}. It consists of disk-like pieces by Lemma \ref{lemma: all disk-like}. Every component $\Sigma$ of $S$ contains imbalanced loop since no $g_i\in \ug$ is $t$-balanced, hence a connected finite cover of $\Sigma$ can be made disk-only by Lemma \ref{lemma: ext surf imba case}. Taking the union of suitable copies of these finite covers of components produces an extremal surface for $c$.
\end{proof}

Now we consider a component $\Sigma$ without any imbalanced cycle. By Lemma \ref{lemma: char imba cycles}, there is an orientation preserving homomorphism $\phi:\Gamma_\Sigma\to \wtilde{\Gamma}$. Any two such homomorphisms differ by a translation on $\wtilde{\Gamma}$. 

Given $\phi$, for each $k\in \Z$, let $s_k(\Sigma)$ be the sum of winding numbers of arcs on the boundary of all pieces $C$ in $\Sigma$ with $\phi(C)=k$. The quantity $s_k(\Sigma)$ is $0$ for all but finitely many $k\in \Z$ since the graph $\Gamma_\Sigma$ is finite. Let $s(\Sigma)$ be the sum of $s_k(\Sigma)(m/\ell)^k$ over all $k\in \Z$.

\begin{lemma}\label{lemma: properties of s}
	The number $s(\Sigma)$ defined above satisfies the following properties:
	\begin{enumerate}
		\item It is invariant under adjustment on $\Sigma$; \label{item: first prop of s}
		\item $s(\Sigma)=0$ if $\Sigma$ can be made disk-only by an adjustment; and \label{item: second prop of s}
		\item Different choices of $\phi$ affect $s(\Sigma)$ by non-zero multiplicative scalars. \label{item: third prop of s}
	\end{enumerate}
\end{lemma}
\begin{proof}
	Suppose an edge $e$ on $\Gamma_\Sigma$ represents the gluing of two pieces $C=o(e)$ and $C'=t(e)$ on $\Sigma$ along paired turns. Then we have $\phi(C')=\phi(C)+1$. If we increase the winding number of the turn by $1$ viewed from $C$, then $w(C)$ would increase by $o_e(1)=M$ and $w(C')$ would decrease (due to opposite orientation) by $t_e(1)=L$. Thus $s(\Sigma)$ increases by $M(m/\ell)^{\phi(C)}-L(m/\ell)^{\phi(C)+1}=0$. This proves (\ref{item: first prop of s}), from which (\ref{item: second prop of s}) immediately follows. (\ref{item: third prop of s}) is obvious since a different choice of $\phi$ differs by a constant $k'$, which changes $s(\Sigma)$ by a scalar $(m/\ell)^{k'}$.
\end{proof}

\begin{lemma}\label{lemma: ext surf no imba case}
	Suppose there is an orientation preserving homomorphism $\phi:\Gamma_\Sigma\to \wtilde{\Gamma}$. With the notation above, a connected finite cover of $\Sigma$ can be made disk-only after a suitable adjustment if $s(\Sigma)=0$.
\end{lemma}
\begin{proof}	
	Since $\Gamma_\Sigma$ is finite and connected, the range of $\phi$ is $[N_{\min},N_{\max}]\cap\Z$ for some integers $N_{\min}$ and $N_{\max}$. Let $N\defeq N_{\max}-N_{\min}$. Then it is easy to see that, for any piece $C$ with $w(C)$ divisible by $n\defeq dm^N \ell^N$ and any path $P$ from $C$ to some $C'$, we have $w(C)\in\Dom\alpha(P)$ and $\alpha(P)w(C)=w(C)(-\ell/m)^{\phi(C')-\phi(C)}$. 
	
	For any $C$, choose a boundary component $B(C)$ passing through $C$. By Lemma \ref{lemma: lift with controlled degree}, there is a finite cover $\wtilde{\Sigma}$ of $\Sigma$ corresponding to a connected finite normal cover of $\Gamma_\Sigma$ where each lift $\wtilde{\omega}_{B(C)}$ of $\omega_{B(C)}$ has degree divisible by $n$, for any piece $C$ in $\Sigma$. Hence by Lemma \ref{lemma: lift balanced and merge}, after an adjustment every piece $\wtilde{C}$ of $\wtilde{\Sigma}$ has winding number $w(\wtilde{C})$ divisible by $n$. Composing the covering projection with $\phi$, we have an orientation preserving homomorphism $\tilde{\phi}:\Gamma_{\wtilde{\Sigma}} \to \wtilde{\Gamma}$ with the same range as $\phi$. Thus the observation made in the first paragraph also applies to $\Gamma_{\wtilde{\Sigma}}$, and all winding numbers can be merged into a single piece $\wtilde{C}_*$ of $\wtilde{\Sigma}$ up to an adjustment. Note that $s(\wtilde{\Sigma})$ is simply $s(\Sigma)$ multiplied by the covering degree, so it must also vanish. It follows that $\wtilde{C}_*$ is a genuine disk since all other pieces are.
\end{proof}

Note that $g\in \BS(M,L)=\langle a,t\in\Z\ |\ a^M=ta^L t^{-1}\rangle$ is $t$-balanced if and only if it lies in $\ker h$, which has presentation $\langle a_k, k\in\Z\ |\ a_k^M=a_{k+1}^L\rangle$ with inclusion into $\BS(M,L)$ given by $a_k\mapsto t^k a t^{-k}$. Hence each $t$-balanced $g$ can be written as $\prod_j a_{k_j}^{u_j}$, where $k_j,u_j\in \Z$. Let $s(g)\defeq\sum_j u_j(m/\ell)^{k_j}$. Note that whether $s(g)=0$ only depends on the conjugacy class of $g$ since $s(wgw^{-1})=s(g)(m/\ell)^{h(w)}$ for any $w\in \BS(M,L)$.

\begin{cor}\label{cor: extremal surfaces}
	Let $M\neq \pm L$. For the reduced chain $c$ above, if every $t$-balanced $g_i\in \ug$ satisfies $s(g_i)=0$, then an extremal surface for $c$ exists.
\end{cor}
\begin{proof}
	Let $S$ be a relative admissible surface built out of the optimal solution to the linear programming in Theorem \ref{thm: rational}. Components containing imbalanced cycles can be made extremal after taking a finite cover by Lemma \ref{lemma: ext surf imba case}. If $\Sigma$ is a component without imbalanced loop, fix some $\phi$ provided by Lemma \ref{lemma: char imba cycles}, note that $s(\Sigma)$ can be computed by a sum over boundary components $B$, where arcs on $B$ contribute $ks(g_i)(m/\ell)^n$ for some $n\in\Z$ depending on $\phi$ if $B$ represents $g_i^k$, where $g_i$ is necessarily $t$-balanced and $s(g_i)=0$ by our assumption. Thus $s(\Sigma)=0$. By Lemma \ref{lemma: ext surf no imba case}, a finite cover of $\Sigma$ can be made extremal.
\end{proof}
\begin{remark}
	Following the comment at the beginning of this subsection, it is easy to prove similar results for the exceptional cases. 
	\begin{enumerate}
		\item If $M=L$, the number $s(g)\defeq\sum u_j$ is well defined for any $g=\prod_j a^{u_j} t^{\epsilon_j}$ ($\epsilon_j=\pm 1$). The reduced chain $c$ bounds an extremal surface if $s(g_i)=0$ for all $g_i\in \ug$. 
		\item If $M=-L$, the number $s(g)\defeq\sum (-1)^j u_j$ is well defined for $g=\prod_j a^{u_j} t^{\epsilon_j}$ ($\epsilon_j=\pm 1$) if $h(g)$ is even. The reduced chain $c$ bounds an extremal surface if either $h(g_i)$ is odd or $s(g_i)=0$ for all $i\in \ug$.
	\end{enumerate}
\end{remark}

\begin{remark}
	Based on this result, the method in \cite{Cal:sshom} produces (closed) surface subgroups in certain graphs of groups.
\end{remark}

\begin{cor}
	Let $G$ be a graph of groups with each vertex group isomorphic to $\BS(M,L)$ for $M\neq \pm L$ and edge groups isomorphic to $\Z$. Assume for each vertex group the inclusions of adjacent edge groups do not form pseudo-inverses except the identity, where each $t$-balanced word $g$ has $s(g)=0$. Then any non-trivial class in $H_2(G;\Q)$ has an integer multiple represented by a $\pi_1$-injective surface of genus at least $1$.
\end{cor}

Now we upgrade Lemma \ref{lemma: ext surf imba case} and Lemma \ref{lemma: ext surf no imba case} to characterize the existence of extremal surfaces for $c$ in term of data canonically obtained from the result of the linear programming problem computing $\scl_{\BS(M,L)}(c)$. Such data is expressed as branched surfaces which we describe as follows.

Recall from Lemma \ref{lemma: compute by linprog} that $\scl_{\BS(M,L)}(c)$ is computed as the minimum of a convex piecewise rational linear function on a compact rational polyhedron $\mathcal{C}(c)$. Thus the set of points $x\in \mathcal{C}(c)$ achieving the optimal value is a nonempty (finite-sided) convex rational polyhedron $\mathcal{P}_{opt}$. For each rational point $x\in \mathcal{P}_{opt}$, by Lemma \ref{lemma: all disk-like} and Lemma \ref{lemma: D'}, there are \emph{finitely} many disk-like pieces $\{C_j\}$ and weights $wt_j\in \Q_{>0}$ such that
\begin{equation}\label{eqn: weights of pieces}
	\sum_j wt_j x(C_j)=x \quad\text{and}\quad \sum_j wt_j=\kappa_v(x).
\end{equation}
Let $\{C_j\}_{j\in J}$ be the finite set of disk-like pieces that possibly appear in (\ref{eqn: weights of pieces}) for some $x\in\mathcal{P}_{opt}$. Then the set of weights satisfying equation (\ref{eqn: weights of pieces}) for some $x\in\mathcal{P}_{opt}$ form a convex rational polyhedron $\mathcal{Q}_w$ in $\R_{\ge0}^J$.

For each weight $wt\in \mathcal{Q}_w$, we obtain a simple branched surface $S$ from the disk-like pieces $C_j$ with positive weights: this is done by gluing each turn to all its paired turns as different branches. Then each piece $C_j$ embeds into $S$, and branching only occurs at gluing loci. See Figure \ref{fig: branchsurf} for an example.

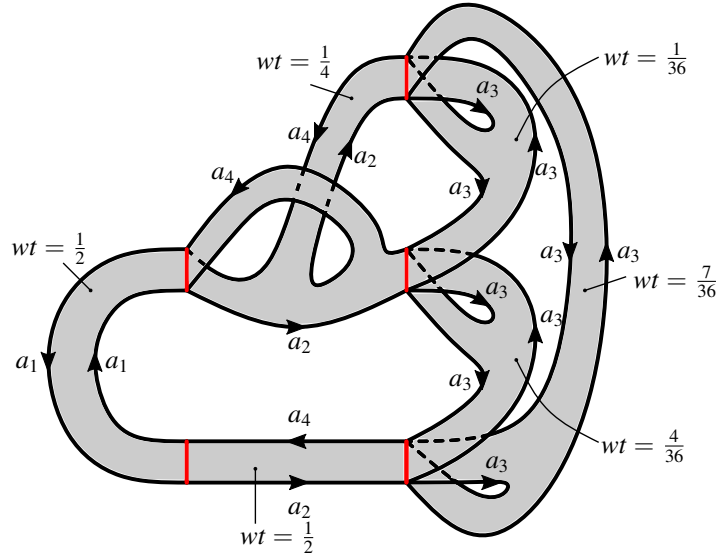
\begin{figure}
	\labellist
	\small 
	\pinlabel $a_1$ at -7 95
	\pinlabel $a_1$ at 42 95
	\pinlabel $wt=\frac{1}{2}$ at 5 158
	
	\pinlabel $a_2$ at 178 207
	\pinlabel $a_2$ at 142 105
	\pinlabel $a_4$ at 142 220
	\pinlabel $a_4$ at 100 198
	\pinlabel $wt=\frac{1}{4}$ at 139 260
	
	\pinlabel $a_4$ at 142 65
	\pinlabel $a_2$ at 142 15
	\pinlabel $wt=\frac{1}{2}$ at 130 0
	
	\pinlabel $a_3$ at 230 190
	\pinlabel $a_3$ at 277 200
	\pinlabel $a_3$ at 245 245
	\pinlabel $wt=\frac{1}{36}$ [bl] at 305 250
	
	\pinlabel $a_3$ at 230 85
	\pinlabel $a_3$ at 280 120
	\pinlabel $a_3$ at 250 135
	\pinlabel $wt=\frac{4}{36}$ [lu] at 330 50
	
	\pinlabel $a_3$ at 280 155
	\pinlabel $a_3$ at 322 155
	\pinlabel $a_3$ at 250 40
	\pinlabel $wt=\frac{7}{36}$ [l] at 320 140
	\endlabellist
	\centering
	\includegraphics[scale=0.7]{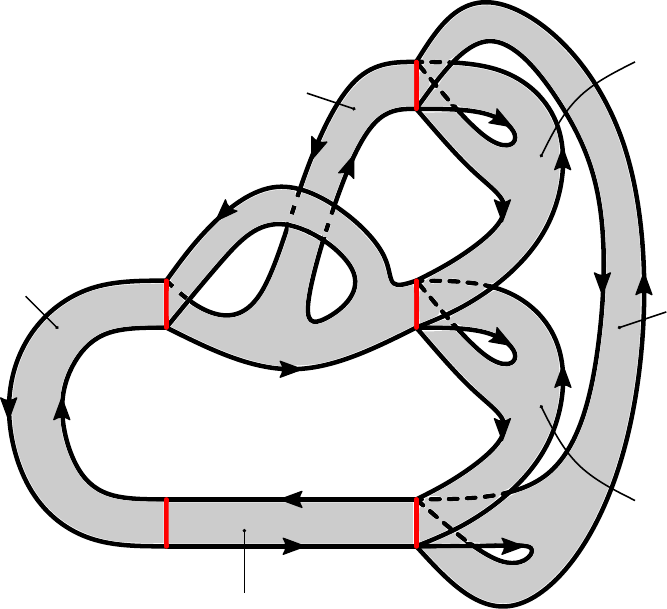}
	\caption{This is a balanced branched surface consisting of $6$ pieces with gluing loci colored red. It carries the surface admissible for $[a,t^2]$ constructed in the proof of Proposition \ref{prop: eg3}.}\label{fig: branchsurf}
\end{figure}

Such a branched surface only depends on the open face $\mathcal{F}$ of $\mathcal{Q}_w$ in which $wt$ lies, so we denote it by $S(\mathcal{F})$.

Each gluing locus has a co-orientation induced by the orientation of $\Gamma$. There is a projection $\pi:S(\mathcal{F}) \to \Gamma$ taking each piece to the unique vertex of $\Gamma$ and each gluing locus to the edge $e$ of $\Gamma$ so that the co-orientation agrees with the orientation of $e$.

$S(\mathcal{F})$ may not be connected. A component $\Sigma(\mathcal{F})$ of $S(\mathcal{F})$ is \emph{balanced} if $\pi$ lifts to $\wtilde{\Gamma}$. Explicitly, $\Sigma(\mathcal{F})$ is balanced if every loop $\omega$ in it intersects an even number of gluing loci, exactly half of which have co-orientations consistent with the orientation of $\omega$. For instance, the branched surface in Figure \ref{fig: branchsurf} is balanced.

For a balanced component $\Sigma(\mathcal{F})$, let $\phi$ be a lift of $\pi$. Given a weight $wt\in\mathcal{F}$, let $s(\Sigma(\mathcal{F}),wt)$ be the sum of $wt_j s_j(m/\ell)^{\phi(C_j)}$ over all $j\in J$, where $s_j$ is the sum of winding numbers of arcs on the boundary of $C_j$. Whether $s(\Sigma(\mathcal{F}),wt)$ vanishes does not depend on the choice of $\phi$.

Finally we are in a position to state our criterion for the existence of extremal surfaces in terms of such branched surfaces.
\begin{thm}[extremal surfaces]\label{thm: extremal surfaces}
	Let $M\neq \pm L$. A reduced rational chain $c$ has an extremal surface if and only if there is some $wt$ in an open face $\mathcal{F}$ of $\mathcal{Q}_w$, such that each balanced component $\Sigma(\mathcal{F})$ of $S(\mathcal{F})$ has $s(\Sigma(\mathcal{F}),wt)=0$.
\end{thm}

\begin{remark}
	This criterion can be checked by an algorithm. For each open face $\mathcal{F}$, each equation $s(\Sigma(\mathcal{F}),wt)=0$ is rational linear in $wt$. Thus the existence of $wt$ in a given $\mathcal{F}$ satisfying the criterion is a linear programming (feasibility) problem, and there are only finitely many open faces $\mathcal{F}$ of $\mathcal{Q}_w$ to enumerate.
\end{remark}

To prove Theorem \ref{thm: extremal surfaces}, in view of Lemma \ref{lemma: ext surf imba case} and Lemma \ref{lemma: ext surf no imba case}, it suffices to produce surfaces without any branching satisfying analogous conditions.

Given a rational weight $wt\in \mathcal{Q}_w$, a surface $\Sigma$ in simple normal form is \emph{carried by} $\Sigma(\mathcal{F})$ with weight $wt$ if there is some $N\in\Z_+$ such that $\Sigma$ is made of $Nwt_j$ copies of $C_j$ for each $j\in J$. The extremal surfaces $S$ that we look for consist of components $\Sigma$, each carried by some $\Sigma(\mathcal{F})$. Note that if $\Sigma(\mathcal{F})$ is balanced, then there is an orientation preserving homomorphism $\phi:\Gamma_\Sigma\to \wtilde{\Gamma}$ sending each copy of $C_j$ to $\phi(C_j)$ where $\phi$ is a lift of $\pi$. In this case, we have $s(\Sigma)=Ns(\Sigma(\mathcal{F}),wt)$.

We can actually find $\Sigma$  that is connected and has the same balanced properties as $\Sigma(\mathcal{F})$. This is Lemma \ref{lemma: carry connected surf} below. To prove it, recall that an undirected graph $\Lambda$ is strongly connected (or $2$-edge-connected) if every edge is non-separating. A strongly connected component (SCC for short) of $\Lambda$ is a maximal strongly connected subgraph. Any connected graph $\Lambda$ uniquely decomposes as a tree of its SCCs. A SCC is trivial if it consists of a single vertex.

\begin{lemma}\label{lemma: cover with nonsep edges}
	Suppose $\Lambda$ is a connected graph where each vertex has valence at least $2$. Then $\Lambda$ has a strongly connected double cover $\wtilde{\Lambda}$.
\end{lemma}
\begin{proof}
	Express $\Lambda$ as a tree of its SCCs. Our assumption implies that each leaf of the tree must be a non-trivial SCC. Let $C_1,\ldots, C_k$ be the non-trivial SCCs of $\Lambda$. Then their complement is a union of embedded trees $T_1,\ldots, T_n$ in $\Lambda$, where each leaf of a tree $T_i$ is glued to a vertex in some $C_j$. Take a connected double cover $\wtilde{C}_i$ of each non-trivial SCC $C_i$ and take two copies $\wtilde{T}_i^{(1)},\wtilde{T}_i^{(2)}$ of each embedded tree $T_i$. They can be assembled into a double cover $\wtilde{\Lambda}$ of $\Lambda$: For each leaf $v$ on $T_i$ that is glued to a vertex $u$ on some $C_j$, match the two leaves $\tilde{v}^{(1)}$, $\tilde{v}^{(2)}$ on $\wtilde{T}_i^{(1)},\wtilde{T}_i^{(2)}$ corresponding to $v$ with the two preimages $\tilde{u}^{(1)}$, $\tilde{u}^{(2)}$ of $u$ on $\wtilde{C}_j$. See Figure \ref{fig: graphscc} for an illustration. Then the graph $\wtilde{\Lambda}$ obtained is strongly connected.
\end{proof}

\begin{figure}
	\labellist
	\small \hair 2pt
	\pinlabel $2:1$ at 265 115
	\pinlabel $\wtilde{\Lambda}$ at 120 -15
	\pinlabel $\Lambda$ at 400 -15
	\endlabellist
	\centering
	\includegraphics[scale=0.7]{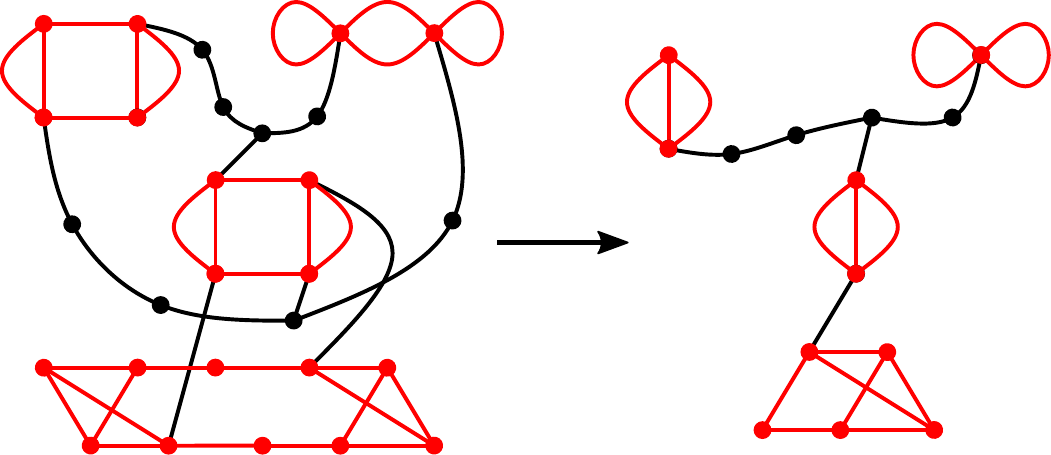}
	\vspace{10pt}
	\caption{The graph $\Lambda$ has its non-trivial SCCs colored red. The graph $\wtilde{\Lambda}$ is a strongly connected double cover of $\Lambda$.}\label{fig: graphscc}
\end{figure}

\begin{lemma}\label{lemma: carry connected surf}
	For each rational weight $wt$ in an open face $\mathcal{F}$, each component $\Sigma(\mathcal{F})$ carries a connected surface $\Sigma$ with weight $wt$. Moreover, we may choose $\Sigma$ to contain an imbalanced loop unless $\Sigma(\mathcal{F})$ is balanced.
\end{lemma}
\begin{proof}
	Let $\Sigma$ be any surface carried by $\Sigma(\mathcal{F})$ with the given weight. First we show that for any given path $P$ of pieces in $\Sigma(\mathcal{F})$, we can modify $\Sigma$ to contain a lift of $P$ without increasing the number of components of $\Sigma$. We accomplish this by lifting longer subpaths on $P$. At each step, suppose our lift ends with a piece $\wtilde{C}_1$ corresponding to $C_1$ on $P$ and let $\wtilde{C}_2$ be an arbitrarily chosen lift of the piece $C_2$ next to $C_1$ on $P$ extending the subpath. For $i=1,2$, let $\wtilde{C}'_i$ be the piece on $\Sigma$ glued to $\wtilde{C}_i$ along the turn corresponding to the turn on $\Sigma(\mathcal{F})$ along which $C_1$ and $C_2$ are glued. Now we modify $\Sigma$ by gluing $\wtilde{C}_1$ to $\wtilde{C}_2$ and $\wtilde{C}'_2$ to $\wtilde{C}_1'$ instead. Figure \ref{fig: rewire} illustrates this at the level of graphs. Such a rewiring operation extends our lift. Moreover, if at least one of the two edges representing the original gluing in $\Gamma_{\Sigma}$ is non-separating, then the number of components in $\Sigma$ does not increase, and it decreases if in addition the two edges lie in different components of $\Sigma$. Finally note that we may assume both edges to be non-separating up to taking a double cover of $\Sigma$ by Lemma \ref{lemma: backtrack} and Lemma \ref{lemma: cover with nonsep edges}.
	
	Now let $\Sigma$ be a surface carried by $\Sigma(\mathcal{F})$ with the given weight that has the minimal number of components. Suppose $\Sigma$ is disconnected and let $\wtilde{C}_0$ and $\wtilde{C}_2$ be two pieces in different components of $\Sigma$. Connect the corresponding pieces $C_0$ and $C_2$ in $\Sigma(\mathcal{F})$ by a path $P$, and let $C_1$ be the piece next to $C_2$ on $P$. Modify $\Sigma$ to lift $P$ by the process above until we obtain a lift of the subpath from $C_0$ to $C_1$, and proceed the last step by choosing $\wtilde{C}_2$ to be the lift of $C_2$. This decreases the number of components of $\Sigma$ by at least one since the components containing $\wtilde{C}_0$ and $\wtilde{C}_2$ are merged at a certain stage of this process. This contradicts the choice of $\Sigma$ and thus $\Sigma$ must be connected.
	
	For $\Sigma$ connected, since $\Sigma(\mathcal{F})$ contains an imbalanced cycle unless it is balanced, we can obtain a copy of such a cycle in $\Sigma$ applying the rewiring operations above.
\end{proof}
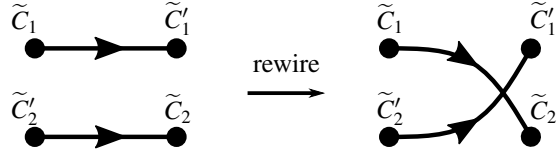
\begin{figure}
	\labellist
	\small \hair 2pt
	\pinlabel $\wtilde{C}'_1$ at 65 55
	\pinlabel $\wtilde{C}_1$ at 0 55
	\pinlabel $\wtilde{C}_2$ at 65 17
	\pinlabel $\wtilde{C}'_2$ at 0 17
	
	\pinlabel rewire at 111 35
	
	\pinlabel $\wtilde{C}'_1$ at 220 55
	\pinlabel $\wtilde{C}_1$ at 155 55
	\pinlabel $\wtilde{C}_2$ at 220 17
	\pinlabel $\wtilde{C}'_2$ at 155 17
	\endlabellist
	\centering
	\includegraphics[scale=0.9]{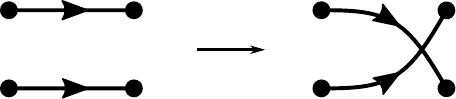}
	\vspace{10pt}
	\caption{The two edges on the left represent pairings of turns of the same type. The rewiring switches the end points of the edges, ie the gluing of pieces.}\label{fig: rewire}
\end{figure}

\begin{proof}[Proof of Theorem \ref{thm: extremal surfaces}]
	Suppose $S$ is an extremal surface. Then it is carried by some $S(\mathcal{F})$ with weight $wt\in\mathcal{Q}_w$, where $\mathcal{F}$ is the open face containing $wt$. Since each piece of $S$ is a genuine disk, each component $\Sigma$ carried by a balanced component $\Sigma(\mathcal{F})$ has $s(\Sigma)=0$. We conclude that $s(\Sigma(\mathcal{F}),wt)=0$ for any balanced component $\Sigma(\mathcal{F})$ since the sum of $s(\Sigma)$ over all components $\Sigma$ of $S$ carried by $\Sigma(\mathcal{F})$ is a positive multiple of $s(\Sigma(\mathcal{F}),wt)$.
	
	Conversely, suppose there is a weight $wt$ in an open face $\mathcal{F}$ satisfying the criterion. Since each equation $s(\Sigma(\mathcal{F}),wt)=0$ is rational linear in $wt$, by rationality of $\mathcal{F}$, we may assume $wt$ to be rational. By Lemma \ref{lemma: carry connected surf}, each component $\Sigma(\mathcal{F})$ carries a connected surface $\Sigma$ with the given weight, such that $\Gamma_\Sigma$ contains no imbalanced loop if and only if $\Sigma(\mathcal{F})$ is balanced. Then each $\Sigma$ corresponding to a balanced $\Sigma(\mathcal{F})$ has $s(\Sigma)=0$ since $s(\Sigma(\mathcal{F}),wt)=0$. By Lemma \ref{lemma: ext surf imba case} and Lemma \ref{lemma: ext surf no imba case}, an extremal surface exists.	
\end{proof}	


\subsection{Lower bounds from duality}\label{subsec: lower bound}
To obtain explicit formulas for scl, especially when we consider chains in $\BS(M,L)$ with parameters or with varying $M$ and $L$, it is often too complicated to work out the linear programming problems. Proving a sharp lower bound is usually the main difficulty. The classical approach using Bavard duality relies on finding extremal quasimorphisms, which is quite hard in practice.

A new method using the idea of linear programming duality is used in \cite{Chen:sclgap} to obtain lower bounds in free products. It actually applies to graphs of groups and proves the so-called spectral gap properties with sharp estimates, which is discussed in detail in a paper with Nicolaus Heuer \cite{CH}. Here we simply describe this method in our setting. Fix any one of the two setups which determines the notion of disk-like vectors as in Subsection \ref{subsec: sclBS setup}.

To each turn $(a_i,\bar{w},a_j)$ we assign a \emph{non-negative cost} $q_{i,\bar{w},j}$. This defines a linear cost function $q$ on $\mathcal{C}_v$. In particular, the cost of a piece is the sum of the costs of the turns on its polygonal boundary.

Recall from Lemma \ref{lemma: compute by linprog} and Remark \ref{rmk: const} that computing scl is equivalent to maximizing the function $\kappa_v$ counting the (normalized) number of disk-like pieces on $\mathcal{C}(c)$, where $c=\sum r_i g_i$. Hence giving upper bounds of $\kappa_v$ produces lower bounds of scl.

\begin{lemma}\label{lemma: dual}
	If $q(d)\ge1 $ for any disk-like vector $d$, then $\kappa_v(x)\le q(x)$ for any $x\in \mathcal{C}_v$.
\end{lemma}
\begin{proof}
	For any admissible expression $x=x'+\sum t_i d_i$ with $d_i$ disk-like, $t_i\ge 0$, and $x'\in \mathcal{C}_v$, we have $q(x)=q(x')+\sum t_iq(d_i)\ge \sum_i t_i$. Thus $q(x)\ge \kappa_v(x)$.
\end{proof}

For each vector $x\in \mathcal{C}(c)$, we write it as $\sum t_{i,\bar{w},j}(a_i,\bar{w},a_j)$ where $t_{i,\bar{w},j}$ is the coordinate corresponding to the basis element $(a_i,\bar{w},a_j)\in T_v$. We think of $t_{i,\bar{w},j}$ as the normalized number of turns of type $(a_i,\bar{w},a_j)$. Recall from Section \ref{sec: scl by linprog} that the gluing condition requires $t_{i,\bar{w},j}=t_{i',\bar{w}',j'}$ if $(a_i,\bar{w},a_j)$ and $(a_{i'},\bar{w}',a_{j'})$ are paired triples, and the normalizing condition implies $\sum_{\bar{w},j} t_{i,\bar{w},j}=r_k$ if $a_i\subset\gamma_k$ and $\sum_{i,\bar{w}} t_{i,\bar{w},j}=r_k$ if $a_j\subset\gamma_k$.

If we can choose the costs so that
\begin{enumerate}
	\item $q(d)\ge1 $ for any disk-like vector $d$, and
	\item for any $x\in \mathcal{C}(c)$ expressed in the form above, $q(x)=\sum q_{i,\bar{w},j}t_{i,\bar{w},j}$ is equal to or bounded above by a constant $K$ on $\mathcal{C}(c)$ by the gluing and normalizing conditions,
\end{enumerate}
then Lemma \ref{lemma: dual} implies $\kappa_v\le K$ on $\mathcal{C}(c)$. Combining with Lemma \ref{lemma: compute by linprog} and Remark \ref{rmk: const}, this gives a lower bound of scl.

\subsection{Examples with explicit formulas}
In this subsection, we compute three examples of complexities $\rho=0,1,2$ that are not $t$-alternating words. As we will see, the computations get more complicated as the complexity increases.

\begin{prop}\label{prop: eg1}
	For $d=\gcd(|M|,|L|)$, we have 
	$$\scl_{\BS(M,L)}(a^k t^2+2t^{-1})=\frac{1}{2}-\frac{\gcd(|k|,d)}{2d}.$$
\end{prop}
\begin{proof}
	Let $n_k=\frac{d}{\gcd(|k|,d)}$, which is the order of $[k]$ in $\Z/d\Z$. We use the notation introduced in Subsection \ref{subsec: sclBS setup}.
	
	Let $\gamma_1$ be the tight loop representing $g_1=a^k t^2$ consisting of two arcs $a_1$ and $a_2$ with winding numbers $k$ and $0$ respectively. Let $\gamma_2$ be the tight loop representing $g_2=t^{-1}$ consisting of a single arc $a_3$ with winding number $0$. The three arcs are depicted in Figure \ref{fig: arcs_ex1}. By formula (\ref{eqn: choose W_0}), we have $W_0(a_1)=W_0(a_2)=W_0(a_3)=d\Z$. Then Setup 1 and Setup 2 coincide, which we use. We have $\rho(g_1+2g_2)=\rho(g_1)=\rho(g_2)=0$ and $D_v=d\Z$. As a consequence, we have $W_e=\{1\}$. That is, we ignore the winding numbers of turns and use a pair $(a_i,a_j)$ instead of a triple $(a_i,\bar{w},a_j)$ to represent a turn.
	
	\begin{figure}
		\labellist
		\small \hair 2pt

		\pinlabel $a_2$ at -15 75
		\pinlabel $a_1$ at -15 55
		\pinlabel $a_3$ at -15 95
		
		\pinlabel $a_1$ at 290 75
		\pinlabel $a_2$ at 290 55
		\pinlabel $a_3$ at 290 95
		\endlabellist
		\centering
		\includegraphics[scale=0.5]{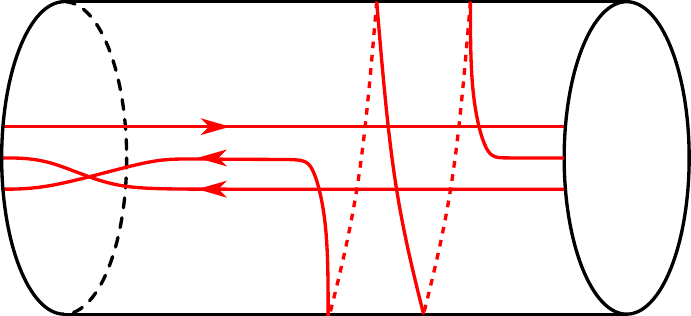}
		\caption{The three arcs of $a^kt^2+2t^{-1}$ in the thickened vertex space $N(X_v)$ with $k=2$.}\label{fig: arcs_ex1}
	\end{figure}
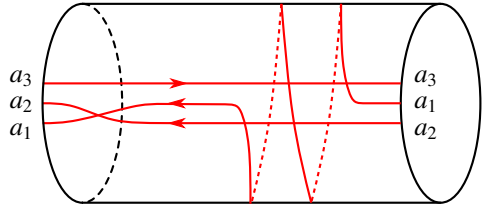
	
	We have a turn $(a_1,a_3)$ paired with $(a_3,a_2)$, and a turn $(a_2,a_3)$ paired with $(a_3,a_1)$. The defining equation $\partial=0$ implies that $\mathcal{C}_v$ consists of vectors of the form $\xi(x,y)=x(a_1,a_3)+x(a_3,a_1)+y(a_2,a_3)+y(a_3,a_2)$ with $(x,y)\in\R_{\ge0}^2$, which has winding number $kx$ mod $d$. Thus such a vector is disk-like if and only if $(x,y)\neq(0,0)\in\Z_{\ge0}^2$ and $kx\in d\Z$. This describes the set $\mathcal{D}(v)$ of disk-like vectors, from which we get $\mathcal{D}(v)+\mathcal{C}_v=\{\xi(n_k,0),\xi(0,1)\}+\mathcal{C}_v$ (See Figure \ref{fig: ex1}). It follows that $\kappa_v(\xi(x,y))=x/n_k+y$. For the chain $c=g_1+2g_2$, the normalizing condition requires $x=y=1$, so $\xi(1,1)$ is the only vector in $\mathcal{C}(c)$. Thus by Lemma \ref{lemma: compute by linprog} and Remark \ref{rmk: const}, we have
	$$\scl(c)=1-\frac{1}{2}\kappa_v(\xi(1,1))=\frac{1}{2}-\frac{1}{2n_k}.$$
\end{proof}

\begin{prop}\label{prop: eg2}
	For all $|M,|L|\ge2$, we have
	$$\scl_{\BS(M,L)}(atta^{-1}t^{-1}+t^{-1})\le\frac{1}{2}-\frac{1}{4|M|}-\frac{1}{4|L|}.$$
	The equality holds if $d=\gcd(|M|,|L|)$ satisfies $d\ge \frac{|M|+|L|}{2\min\{|M|,|L|\}}$.
\end{prop}
\begin{proof}
	We have $3$ arcs $a_1, a_2, a_3$ on the loop $\gamma_1$ representing $atta^{-1}t^{-1}$ with winding numbers $1,0,-1$ respectively, and have another arc $a_4$ on the other loop $\gamma_2$ with winding number $0$. It easily follows from the definitions in Subsection \ref{subsec: sclBS setup} that $\rho(\gamma_1)=1$, $\rho(\gamma_2)=0$, so our chain $c$ has complexity $\rho(c)=1$. 
	
	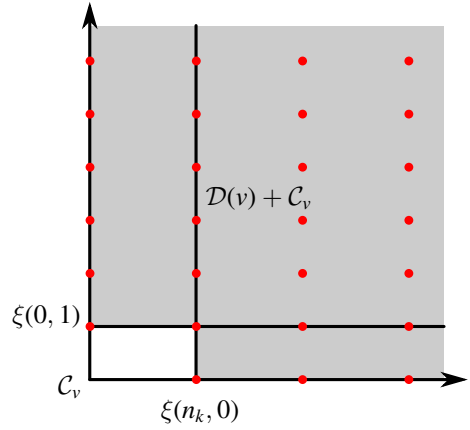
\begin{figure}
		\labellist
		\small \hair 2pt
		\pinlabel $\mathcal{C}_v$ at -5 0
		\pinlabel $\xi(n_k,0)$ at 50 -10
		\pinlabel $\xi(0,1)$ at -15 30	
		\pinlabel $\mathcal{D}(v)+\mathcal{C}_v$ at 75 80
		\endlabellist
		\centering
		\includegraphics[scale=0.9]{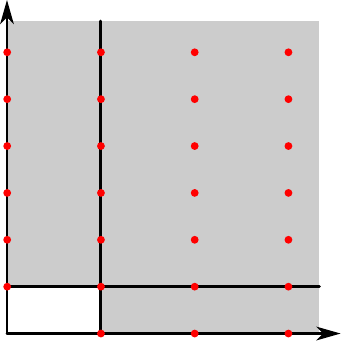}
		\vspace{10pt}
		\caption{The red dots are disk-like vectors in $\mathcal{D}(v)$. The gray region is the set $\mathcal{D}(v)+\mathcal{C}_v$.}\label{fig: ex1}
	\end{figure}

	We use Setup 1 in this proof. By formula (\ref{eqn: choose W_0}), we have $W_0(a_1)=dm\Z$, $W_0(a_2)=W_0(a_4)=d\Z$, $W_0(a_3)=d\ell\Z$. Hence $D_v=dm\ell\Z$ and $W_e=\Z/m\ell \Z$. Turns are paired up as follows:
	\begin{align*}
	(a_1,\bar{w},a_1)\quad \leftrightarrow \quad (a_3,-\bar{w},a_2) \quad\quad\quad\quad & (a_1,\bar{w}, a_4) \quad \leftrightarrow \quad (a_4,-\bar{w}, a_2)\\
	(a_2,\bar{w},a_4)\quad \leftrightarrow \quad (a_4,-\bar{w},a_3) \quad\quad\quad\quad & (a_2,\bar{w}, a_1) \quad \leftrightarrow \quad (a_3,-\bar{w}, a_3)
	\end{align*}
	
	These are the only turns, so each piece $C$ falls into exactly one of the following three types.
	\begin{enumerate}[(i)]
		\item The polygonal boundary contains both $a_2$ and $a_4$. It is disk-like if and only if $w(C)$ is divisible by $d$ since $W(a_2)=d\Z$. Changing the winding numbers of turns will change $w(C)$ by a multiple of $|M|$ or $|L|$, both divisible by $d$.\label{enum: first type}
		\item The polygonal boundary contains $a_1$ only. Then the winding number $w(C)\equiv k\mod |M|$ if there are $k$ copies of $a_1$ on the boundary, which does not depend on the winding numbers of turns. Thus it is disk-like if and only if $w(C)\in |M|\Z=W(a_1)$, ie $k$ is divisible by $|M|$.
		\item The polygonal boundary contains $a_3$ only. Similar to the previous case, it is disk-like if and only if the number of copies of $a_3$ on the boundary is divisible by $|L|$.
	\end{enumerate}
	In summary, whether a piece is disk-like does not depend on the winding numbers of turns on its polygonal boundary. Thus we simply assume all turns to have winding numbers $0$ in what follows.
	
	We prove the upper bound by constructing a simple relative admissible surface $S$ of degree $|2ML|$ consisting of the following disk-like pieces described by the turns on the polygonal boundaries (also see Figure \ref{fig: ex2}).
	\begin{enumerate}
		\item $(a_2,0,a_4)+(a_4,0,a_2)$, take $|ML|$ copies of this piece;
		\item $(a_2,0,a_1)+(a_1,0,a_4)+(a_4,0,a_3)+(a_3,0,a_2)$, take $|ML|$ copies of this piece;
		\item $|M|(a_1,0,a_1)$, take $|L|$ copies of this piece;
		\item $|L|(a_3,0,a_3)$, take $|M|$ copies of this piece.
	\end{enumerate}
	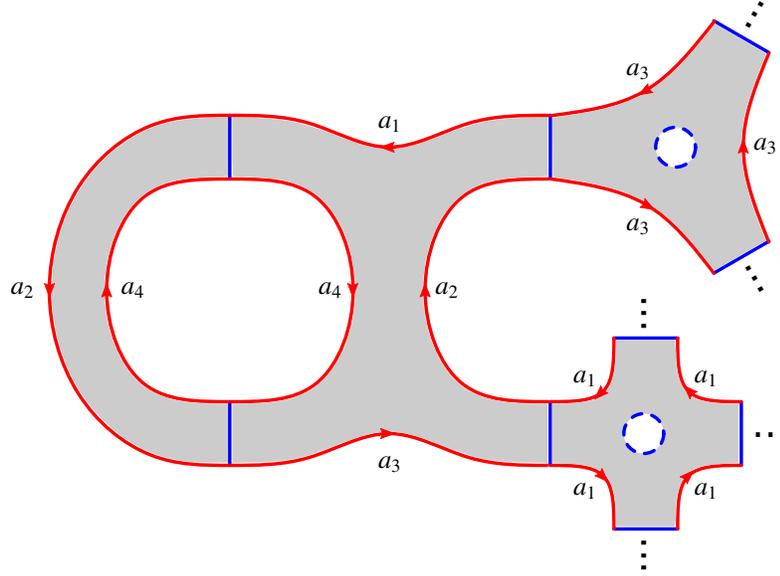
\begin{figure}
		\labellist
		\small \hair 2pt
		\pinlabel $a_1$ at 255 90
		\pinlabel $a_1$ at 312 90
		\pinlabel $a_1$ at 255 37
		\pinlabel $a_1$ at 312 37
		\pinlabel $a_1$ at 163 210
		\pinlabel $a_2$ at -10 132
		\pinlabel $a_2$ at 190 132
		\pinlabel $a_3$ at 163 50
		\pinlabel $a_3$ at 280 235
		\pinlabel $a_3$ at 280 162
		\pinlabel $a_3$ at 340 200
		\pinlabel $a_4$ at 42 132
		\pinlabel $a_4$ at 135 132
		\endlabellist
		\centering
		\includegraphics[scale=0.8]{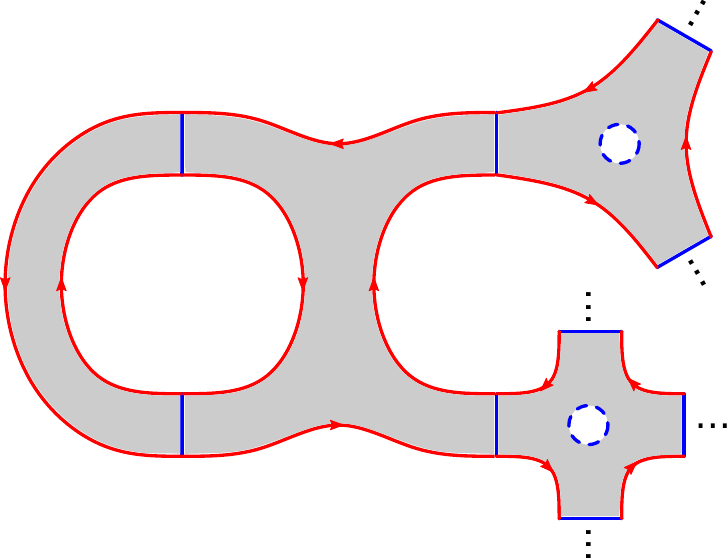}
		\caption{Part of a relative admissible surface involving the four disk-like pieces constructed to give the upper bound, illustrating the case $M=4$ and $L=3$}\label{fig: ex2}
	\end{figure}
	It is easy to check that these are disk-like pieces and that the gluing conditions hold. With notation as in Lemma \ref{lemma: asym prom}, we get $\what\chi(S)=-2|ML|+|M|+|L|$ and obtain the upper bound
	$$\scl_{\BS(M,L)}(atta^{-1}t^{-1}+t^{-1})\le \frac{1}{2}-\frac{1}{4|M|}-\frac{1}{4|L|}.$$
	
	To obtain the lower bound, we use the duality method introduced in Subsection \ref{subsec: lower bound}. Assign a cost $q_{ij}$ independent of the winding number $\bar{w}$ to each turn $(a_i,\bar{w}, a_j)$ as in the following matrix $Q=(q_{ij})$, where $*$ appears if there are no such turns.
	$$Q=\begin{pmatrix}
	\frac{1}{|M|} &* &* &1-\left(\frac{1}{|M|}+\frac{1}{|L|}\right)\\
	\frac{1}{|M|} &* &* &1-\frac{1}{2}\left(\frac{1}{|M|}+\frac{1}{|L|}\right)\\
	* &\frac{1}{|L|} &\frac{1}{|L|} &*\\
	* &\frac{1}{2}\left(\frac{1}{|M|}+\frac{1}{|L|}\right) &0 &*\\
	\end{pmatrix}$$
	We first check that every disk-like piece costs at least $1$ when $d\ge \frac{|M|+|L|}{2\min\{|M|,|L|\}}$. According to the classification of pieces above, only those of type (\ref{enum: first type}) requires some attention. Let $C$ be such a piece, which must contain a turn ending at $a_4$. Assume $|M|\le |L|$.
	\begin{enumerate}
		\item Suppose $C$ contains a turn from $a_1$ to $a_4$. 
		\begin{enumerate}
			\item If this turn is followed by another from $a_4$ to $a_3$, then we must also have a turn from $a_3$ to $a_2$ (to leave $a_3$) and another from $a_2$ to $a_1$ so that the boundary closes up. In this case, the cost is at least $q_{14}+q_{43}+q_{32}+q_{21}=1$.
			\item If this turn is followed by another from $a_4$ to $a_2$ instead, then the cost is at least $q_{14}+q_{42}+\min(q_{21},q_{24})\ge 1$ since $2\le |M|\le |L|$.
		\end{enumerate} 
	
		\item Suppose $C$ does not contain a turn from $a_1$ to $a_4$. Then $C$ does not visit $a_1$ and must contain a turn from $a_2$ to $a_4$. If $C$ also contains a turn from $a_4$ to $a_2$, then the cost will be at least $q_{24}+q_{42}=1$. Otherwise, $C$ is encoded as $n_1(a_2,0,a_4)+n_1(a_4,0,a_3)+n_2(a_3,0,a_3)+n_1(a_3,0,a_2)$ with integers $n_1\ge1$, $n_2\ge0$, and costs $n_1(q_{24}+q_{43}+q_{32})+n_2q_{33}$. Then $C$ has winding number $w(C)\equiv -(n_1+n_2)\mod d$. For $C$ to be disk-like, we have $n_1+n_2\ge d$. Note that $q_{24}\ge 1/2\ge q_{33}$ since $|M|,|L|\ge 2$. Therefore, the cost
		$$n_1(q_{24}+q_{43}+q_{32})+n_2q_{33}\ge q_{24}+q_{43}+q_{32}+(d-1)q_{33}=1+\frac{d-\frac{1}{2}}{|L|}-\frac{1}{2|M|}\ge 1$$
		since $d\ge\frac{|M|+|L|}{2|M|}$.
	\end{enumerate}
	The other case $|M|\ge |L|$ is similar.
	
	Now let $t_{ij}=\sum_{\bar{w}} t_{i,\bar{w},j}$, where $t_{i,\bar{w},j}$ is the normalized number of the turn $(a_i,\bar{w}, a_j)$.
	Then we obtain $t_{14}=t_{42}$ from the gluing conditions, which implies the total cost
	$$\sum_{i,j}q_{ij}t_{ij}=\left[1-\frac{1}{2}\left(\frac{1}{|M|}+\frac{1}{|L|}\right)\right](t_{14}+t_{24})+\frac{1}{|M|}(t_{11}+t_{21})+\frac{1}{|L|}(t_{32}+t_{33}).$$
	The normalizing conditions imply $t_{14}+t_{24}=t_{11}+t_{21}=t_{32}+t_{33}=1$ and thus 
	$$\sum_{i,j} q_{ij}t_{ij}=1+\frac{1}{2}\left(\frac{1}{|M|}+\frac{1}{|L|}\right),$$
	which is an upper bound of $\kappa_v(x)$ for all $x\in \mathcal{C}(c)$ by Lemma \ref{lemma: dual}. Hence
	$$\scl_{\BS(M,L)}(atta^{-1}t^{-1}+t^{-1})\ge \frac{1}{2}-\frac{1}{4|M|}-\frac{1}{4|L|}$$	
	by Lemma \ref{lemma: compute by linprog} and Remark \ref{rmk: const}.
\end{proof}
\begin{remark}\label{rem: eg2 convergence}
	A slightly weaker lower bound 
	$$\scl_{\BS(M,L)}(atta^{-1}t^{-1}+t^{-1})\ge \frac{1}{2}-\frac{1}{2\min\{|M|,|L|\}}.$$
	holds for all $|M|,|L|\ge2$, which can be proved in a similar way with much simpler computations using cost matrix
	$$Q=\begin{pmatrix}
	\frac{1}{\min\{|M|,|L|\}} &* &* &0\\
	0&* &* &0\\
	* &0 &\frac{1}{\min\{|M|,|L|\}} &*\\
	* &1 &1 &*\\
	\end{pmatrix}.$$
	This bound is sharp when $d=1$ by constructing admissible surfaces in simple normal form.
\end{remark}

In contrast to the two examples above, the winding numbers of turns cannot be ignored in the following example, which has higher complexity.
\begin{prop}\label{prop: eg3}
	$$\scl_{\BS(2,3)}([a,t^2])=\frac{5}{24}.$$
\end{prop}
\begin{proof}
	We have four arcs $a_1,\ldots,a_4$ with winding numbers $1,0,-1,0$ respectively. We use Setup 1 and compute $W(a_1)=4\Z$, $W(a_2)=W(a_4)=6\Z$, and $W(a_3)=9\Z$. Then $D_v=D_e=36\Z$, $W_e=\Z/36\Z$ and the complexity $\rho(c)=2$. The turns are paired up as follows.
	\begin{align*}
	(a_1,\bar{w},a_1)\ \  \leftrightarrow \ \ (a_4,-\bar{w},a_2)\\
	(a_2,\bar{w},a_4) \ \ \leftrightarrow \ \  (a_3,-\bar{w},a_3)\\
	(a_2,\bar{w},a_1)\ \ \leftrightarrow \ \ (a_4,-\bar{w},a_3)\\
	(a_1,\bar{w},a_4)\ \ \leftrightarrow \ \ (a_3,-\bar{w},a_2)
	\end{align*}
	
	To get $5/24$ as an upper bound, we present a simple relative admissible surface $S$ of degree $36$ consisting of the following disk-like pieces described by the turns on the polygonal boundaries. By the computation above and our orientation on $e$, the maps $\overline{o_e},\overline{t_e}: \Z/36\Z\to \Z/36\Z$ are given by $\overline{o_e}(\bar{w})=2\bar{w}$ and $\overline{t_e}(\bar{w})=3\bar{w}$.
	
	\begin{enumerate}
		\item $(a_1,0,a_1)+(a_1,1,a_1)$, which is disk-like since its winding number $2w(a_1)+\overline{o_e}(0+1)\in 4+D_v$ lies in $W(a_1)=4\Z$. This is the leftmost piece in Figure \ref{fig: branchsurf}. Take $18$ copies of this piece;
		\item $(a_2,0,a_4)+(a_4,0,a_2)$, which is disk-like since its winding number $w(a_2)+w(a_4)+\overline{o_e}(0)+\overline{t_e}(0)\equiv 0\mod 36$ lies in $W(a_2)=6\Z$. This is depicted in the bottom-middle of Figure \ref{fig: branchsurf}. Take $18$ copies of this piece.
		\item $(a_2,1,a_4)+(a_2,2,a_4)+2(a_4,-1,a_2)$, which is disk-like since its winding number $2w(a_2)+2w(a_4)+\overline{o_e}(1+2)+2\overline{t_e}(-1)\equiv 0\mod 36$ lies in $W(a_2)=6\Z$. This is depicted in the upper-middle of Figure \ref{fig: branchsurf}. Take $9$ copies of this piece.
		\item $(a_3,0,a_3)+2(a_3,-1,a_3)$, which is disk-like since its winding number $3w(a_3)+\overline{t_e}(2(-1))\equiv -9\mod 36$ lies in $W(a_3)=9\Z$. This is the piece on the right of Figure \ref{fig: branchsurf} with weight $4/36$. Take $4$ copies of this piece.
		\item $(a_3,-2,a_3)+2(a_3,0,a_3)$, which is disk-like since its winding number $3w(a_3)+\overline{t_e}(-2)\equiv -9\mod 36$ lies in $W(a_3)=9\Z$. This is the piece on the right of Figure \ref{fig: branchsurf} with weight $7/36$. Take $7$ copies of this piece.
		\item $(a_3,-1,a_3)+2(a_3,-2,a_3)$, which is disk-like since its winding number $3w(a_3)+\overline{t_e}(-1-2-2)\equiv -18\mod 36$ lies in $W(a_3)=9\Z$. This is the piece on the right of Figure \ref{fig: branchsurf} with weight $1/36$. Take $1$ copy of this piece.
	\end{enumerate}
	It is easy to see that the gluing condition is satisfied and the surface $S$ obtained is carried by the branched surface in Figure \ref{fig: branchsurf}. Note that $S$ is relative admissible of degree $36$ with $\what\chi(S)/36=2-57/36=5/12$, and thus
	$$\scl_{\BS(2,3)}([a,t^2])\le \frac{5}{24}.$$
	
	To establish the lower bound, assign cost $q_{i,\bar{w},j}$ to the turn $(a_i,\bar{w},a_j)$ with 
	$$q_{1,\bar{w},1}=\left\{
	\begin{array}{ll}
	1/4 &\text{if } \bar{w} \text{ is even },\\
	3/4 &\text{otherwise},
	\end{array}\right.
	\quad q_{4,\bar{w},2}=\left\{
	\begin{array}{ll}
	1 &\text{if } \bar{w} \text{ is even},\\
	1/2 &\text{otherwise},
	\end{array}\right.$$
	and for any $\bar{w}$
	\[
	\begin{array}{llllll}
	q_{2,\bar{w},4}=0,		&q_{3,\bar{w},3}=1/3,	&q_{2,\bar{w},1}=1/3,	&q_{4,\bar{w},3}=0, &q_{1,\bar{w},4}=1/4,	&q_{3,\bar{w},2}=1.\\
	\end{array}
	\]
	We check each disk-like piece $C$ costs at least $1$. 
	\begin{figure}
		\labellist
		\small \hair 2pt
		\pinlabel $a_1$ at 70 75
		\pinlabel $a_2$ at 160 135
		\pinlabel $a_3$ at 250 75
		\pinlabel $a_4$ at 160 -7
		\endlabellist
		\vspace{3 pt}
		\centering
		\includegraphics[scale=0.8]{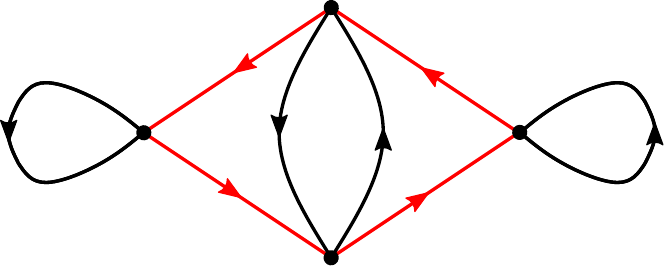}
		\caption{The graph $Y$ indicating turns between arcs, except that every single edge represents $|W_e|=36$ multiedges}\label{fig: Y}
	\end{figure}
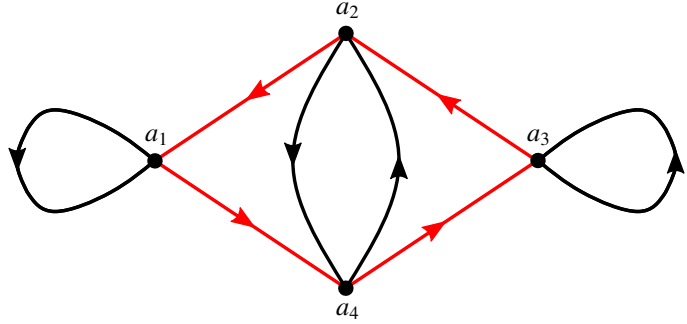
	\begin{enumerate}
		\item Suppose $C$ does not contain any turns of forms $(a_2,\bar{w},a_1)$, $(a_4,\bar{w},a_3)$, $(a_1,\bar{w},a_4)$, or $(a_3,\bar{w},a_2)$. That is, the red turns in Figure \ref{fig: Y} are excluded. Then there are three cases:
		\begin{enumerate}
			\item The boundary only contains $a_1$. By formula (\ref{eqn: w}), $w(a_1,\bar{w},a_1)=w(a_1)+\overline{o_e}(\bar{w})=1+2\bar{w}\in \Z/36\Z$. Thus we have $w(C)\equiv n_0+3n_1\mod 4$, where $n_0$ (resp. $n_1$) is the total number of turns $(a_1,\bar{w},a_1)$ on $C$ with $\bar{w}$ even (resp. odd). For $C$ to be disk-like, we have $w(C)\in W(a_1)=4\Z$. Thus the cost is $(n_0+3n_1)/4=k$ for some integer $k\ge1$.
			\item The boundary only contains $a_2$ and $a_4$. Then $w(C)\equiv 2n_0+n_1\mod 2$, where $n_0$ (resp. $n_1$) is the total number of turns $(a_4,\bar{w},a_2)$ on $C$ with $\bar{w}$ even (resp. odd). For $C$ to be disk-like, we need $w(C)\in W(a_2)=6\Z$. Hence $2n_0+n_1$ must be even and the cost $(2n_0+n_1)/2=k$ for some integer $k\ge1$.
			\item The boundary only contains $a_3$. Then $w(C)\equiv n\mod 3$, where $n$ is the total number of turn on $C$. For $C$ to be disk-like, we need $w(C)\in W(a_3)=9\Z$. Hence $n$ must be divisible by $3$ and the cost $n/3=k$ for some integer $k\ge1$.
		\end{enumerate}
		\item Now suppose $C$ is a disk-like piece containing at least one of the turns $(a_2,\bar{w},a_1)$, $(a_4,\bar{w},a_3)$, $(a_1,\bar{w},a_4)$ or $(a_3,\bar{w},a_2)$. 
		\begin{enumerate}
			\item If $C$ contains $a_3$ on the boundary, then it must have a turn $(a_3,\bar{w},a_2)$ which already has cost $q_{3,\bar{w},2}=1$.
			\item If $C$ does not contain $a_3$, then it includes $3$ turns of the forms $(a_2,\bar{w},a_1)$, $(a_1,\bar{w}',a_4)$ and $(a_4,\bar{w}'',a_2)$ respectively. In this case, the cost of $C$ is at least $$q_{2,\bar{w},1}+q_{1,\bar{w}',4}+q_{4,\bar{w}'',2}\ge \frac{1}{3}+\frac{1}{4}+\frac{1}{2}>1.$$
		\end{enumerate}
	\end{enumerate}
	
	In summary, any disk-like piece has cost at least $1$. Let $t_{i,\bar{w},j}$ be the normalized number of turns $(a_i,\bar{w},a_j)$ in a vector $x\in \mathcal{C}(c)$. Then the gluing conditions imply $t_{1,\bar{w},1}=t_{4,-\bar{w},2}$, $t_{1,\bar{w},4}=t_{3,-\bar{w},2}$ and $t_{2,\bar{w},1}=t_{4,-\bar{w},3}$. Therefore, the total cost
	\begin{eqnarray*}
		\sum_{i,\bar{w},j} q_{i,\bar{w},j}t_{i,\bar{w},j}&=&\left(\frac{1}{4}+1\right)\left(\sum_{\bar{w}\text{ even}}t_{1,\bar{w},1}+\sum_{\bar{w}}t_{1,\bar{w},4}\right)+\left(\frac{3}{4}+\frac{1}{2}\right)\sum_{\bar{w}\text{ odd}}t_{1,\bar{w},1}\\
		&+&\frac{1}{3}\sum_{\bar{w}}(t_{3,\bar{w},3}+t_{4,\bar{w},3})\\
		&=&\frac{5}{4}\sum_{\bar{w}}(t_{1,\bar{w},1}+t_{1,\bar{w},4})+\frac{1}{3}\sum_{\bar{w}}(t_{3,\bar{w},3}+t_{4,\bar{w},3}).
	\end{eqnarray*}
	Combining with the normalizing condition, we have
	$$\kappa_v(x)\le\sum_{i,\bar{w},j} q_{i,\bar{w},j}t_{i,\bar{w},j}=\frac{5}{4}+\frac{1}{3},$$
	by the duality method. Hence by Lemma \ref{lemma: compute by linprog}, we have $\scl_{\BS(2,3)}([a,t^2])\ge \frac{5}{24}$.
\end{proof}

If $M$ and $L$ are coprime, one can see that any reduced word of the form $g=a^{u_1}ta^{u_2}t\cdots a^{u_n}ta^{v_1}Ta^{v_2}T\ldots a^{v_n}T$ can be rewritten as a reduced word $g=a^{u}t^na^{v}T^n$, where $T=t^{-1}$. A trick using the Chinese remainder theorem shows $\scl_{\BS(M,L)}(g)=\scl_{\BS(M,L)}([a,t^n])$ for any such $g$. Thus it would be interesting to know how the sequence $\scl_{\BS(M,L)}([a,t^n])$ behaves as $n\to \infty$, where the complexity increases unboundedly. For example, does $\scl_{\BS(M,L)}([a,t^n])$ converges to some limit? If so, how fast does it converge?

One can lift $[a,t^n]$ to the infinite cyclic cover corresponding to $\ker h$, which is an infinite amalgam with presentation $\wtilde{G}(M,L)=\langle a_k, k\in\Z\ |\ a_k^M=a_{k+1}^L\rangle$. Then $g_n=a_0 a_n^{-1}$ is a lift of $[a,t^n]$ for each $n$. One can use techniques similar to the computations above to show that $\scl_{(\wtilde{G}(M,L),\langle a_0\rangle)}(g_n)$ converges \emph{exponentially} in $n$ to $1/2$. It is not clear whether a similar convergence holds for $\scl_{\BS(M,L)}([a,t^n])$. An exponential convergence seems unusual for scl of a family of words with linear word length growth, eg in free groups.

\subsection{Convergence theorem}
Throughout this subsection, let $F_2=\langle a,t\rangle$ and consider $\BS(M,L)$ as the quotient by imposing the relation $a^M=ta^L t^{-1}$. We say a family of $\BS(M,L)$ is a \emph{surgery family} if $d=\gcd(|M|,|L|)\to\infty$.

We observe from Remark \ref{rem: eg2 convergence} that 
$$\scl_{\BS(M,L)}(atta^{-1}t^{-1}+t^{-1})\to \scl_{F_{2}}(atta^{-1}t^{-1}+t^{-1}) =1/2$$ 
as $|M|,|L|\to \infty$. In contrast, Proposition \ref{prop: eg1} shows that for certain chains like $c=at^2+2t^{-1}$, the convergence depends on how $|M|,|L|$ go to infinity, governed by how $d=\gcd(|M|,|L|)$ behaves. The purpose of this subsection is to study this convergence phenomenon for arbitrary chains in surgery families.

To describe the convergence in detail, we say a sequence $\{x_d\}$ is \emph{quasirational} in $d$ if there exists $\pi\in\Z_+$ and a rational function $\phi_p\in \Q(n)$ for each $0\le p<\pi$ such that $x_{n\pi+p}=\phi_p(n)$.

\begin{thm}\label{thm: surgery}
	For any chain $c\in B_1^H(F_2)$, let $\bar{c}$ be its image in $\BS(M,L)$. If $d=\gcd(M,L)\to \infty$, then $\scl_{\BS(M,L)}(\bar{c})$ converges to $\scl_{F_2}(c)$. Moreover, for $M,L$ fixed, the sequence $\scl_{\BS(dM,dL)}(\bar{c})$ is eventually quasirational in $d$ and converges to $\scl_{F_2}(c)$.
\end{thm}

Here $F_2$ is considered as the free HNN extension, that is, the graph of groups with a single vertex group $\Z$ and a single trivial edge group. Fix a rational chain $c=\sum r_ig_i\in B_1^H(F_2)$ with each $g_i$ hyperbolic represented by a tight loop $\gamma_i$. We can apply Theorem \ref{thm: rational} to compute $\scl_{F_2}(c)$, where disk-like pieces are the same as genuine disks since the edge group is trivial. Thus we notice that a piece in $\BS(M,L)$ becomes a disk-like piece in $F_2$ by ignoring winding numbers of turns if and only if the winding numbers of the arcs on the polygonal boundary sum to $0$. We refer to such pieces as \emph{stable} disk-like pieces.

When $d=\gcd(M,L)$ is sufficiently large, the cyclically reduced expression of $g_i$ is also a cyclically reduced in $\BS(M,L)$. More precisely, with
$$N\defeq\max_{a\in A_v} |w(a)|,$$
where $A_v$ is the set of arcs obtained by cutting $\uga=\{\gamma_i\}$, each $g_i$ is reduced when $d>N$. In particular, the image $\bar{g_i}$ of each $g_i$ in $\BS(M,L)$ is hyperbolic with a constant complexity $\rho(\bar{g_i})$ when $d>N$. 

With such large $d$, recall from Lemma \ref{lemma: compute by linprog} that $\scl_{\BS(M,L)}(\bar{c})$ is computed as an optimization on the space $\mathcal{C}(\bar{c})$ where we essentially maximize $\kappa_v$. Let $\kappa_v^*$ be the variant of $\kappa_v$ that counts the maximal number of \emph{stable} disk-like pieces. Then the optimization in Lemma \ref{lemma: compute by linprog} with $\kappa_v(x)$ replaced by $\kappa_v^*(x)$ computes $\scl_{F_2}(c)$ instead.

We prove the first assertion in Theorem \ref{thm: surgery}.
\begin{lemma}\label{lemma: limit is F_2}
	With notation as above, $\scl_{\BS(M,L)}(\bar{c})$ converges to $\scl_{F_2}(c)$ as $d\to \infty$.
\end{lemma}
\begin{proof}
	By monotonicity of scl, it suffices to show for any $\epsilon>0$, we have $\scl_{\BS(M,L)}(\bar{c})\ge \scl_{F_2}(c)-\epsilon$ for any $d$ sufficiently large. Let $S$ be any simple relative admissible surface for $\bar{c}$ in $\BS(M,L)$. With the integer $N$ defined above, note that each disk-like piece $C$ of $S$ is either stable or contains at least $k$ arcs if $d\ge kN$. In fact, the sum of winding numbers of arcs on $\bdry C$ is divisible by $d$ since both $w(C)$ (by equation (\ref{eqn: choose D_v})) and the contribution of each turn to $w(C)$ are divisible by $d$. Thus the sum is either $0$, in which case $C$ is stable, or has absolute value at least $d\ge kN$, where we must have at least $k$ arcs.
	
	Let $n>0$ be the degree of $S$. Then the number of stable disk-like pieces in $S$ is no more than $n\kappa_v^*(x)$ where $x=x(S)/n\in \mathcal{C}(c)$. By the observation above, the number of disk-like pieces in $S$ that are not stable cannot exceed $(nN/d)\sum r_iA_i$, where $A_i$ is the number of arcs that the edge space cut $\gamma_i$ into. Since $N$ and $\sum r_iA_i$ are fixed, for any $d$ sufficiently large, the total number of disk-like pieces in $S$ is no more than $n(\kappa_v^*(x)+\epsilon)$, from which the desired estimate follows.
\end{proof}

Now we fix coprime integers $m,\ell\neq0$ and focus on $\BS(dm,d\ell)$. To ease the notation, we assume $m,\ell>0$. With $d>N$, each $\bar{g}_i$ is reduced, thus $D_v=dm^{\rho(\bar{c})}\ell^{\rho(\bar{c})}\Z$ is linear in $d$, and $D_e=m^{\rho(\bar{c})}\ell^{\rho(\bar{c})}\Z$ does not depend on $d$. Therefore, the spaces $\mathcal{C}_v$ and $\mathcal{C}(\bar{c})$ are both eventually independent of $d$. What does depend on $d$ is the set of disk-like vectors and the function $\kappa_v$. Denote $\mathcal{D}(v)^{(d)}=\mathcal{D}(v)$ and $\kappa_v^{(d)}=\kappa_v$ to emphasize the dependence, which we investigate to prove the second assertion in Theorem \ref{thm: surgery}.

Throughout this subsection, we use Setup 2, where the criterion for disk-like vectors is simpler. Recall from Subsection \ref{subsec: sclBS setup} that each vector $x\in \mathcal{C}_v$ assigns non-negative weights to edges in the graph $Y$. An integer vector $x\in \mathcal{C}_v$ is disk-like if and only if its winding number $w(x)\in D_v$ and $\supp(x)\subset Y$ is connected.

\begin{lemma}\label{lemma: extremal rays}
	The extremal rays of the polyhedral cone $\mathcal{C}_v$ correspond to embedded cycles in $Y$. For each $x\in \mathcal{C}_v$, its support $\supp(x)$ only depends on the unique open face of $\mathcal{C}_v$ containing $x$.
\end{lemma}
\begin{proof}
	For each $x\in \mathcal{C}_v$, the requirement $\bdry x=0$ implies that $x$ corresponds to a positive weighted sum of finitely many cycles $\omega_i$ in $Y$. Moreover, we may assume each $\omega_i$ to be embedded since each cycle is a sum of embedded ones. Thus $x\neq0$ is on a extremal ray if and only if $x$ corresponds to $\alpha \omega$ for some $\alpha>0$ and an embedded cycle $\omega$. It follows that vectors in the same open face have the same support.
\end{proof}

We define the support of an open face $\mathcal{F}$ of $\mathcal{C}_v$ to be $\supp(x)$ for any $x\in \mathcal{F}$.

To see how the set of disk-like vectors depend on $d$, similar to the proof of \cite[Theorem 4.11]{Susse}, we consider the following embedding 
\begin{eqnarray*}
	\mathcal{E}_d: \R^{T_v} &\to& \R^{T_v}\times \R \\
	x&\mapsto & (x,w(x)/|D_v|)
\end{eqnarray*}
where $|D_v|=dm^{\rho(c)}\ell^{\rho(c)}$ and $w$ is defined as in the equation (\ref{eqn: w}) by fixing a lift $\tilde{w}\in \Z$ for each $\bar{w}\in W_e$.

Note that $w(x)$ is the sum of two parts. The first part is the (weighted) sum of winding numbers of arcs $a_i$ in $x$, which is independent of $d$. The second part is the sum of winding numbers contributed by turns, which is linear in $d$ since both $o_e(\tilde{w})=dm\tilde{w}$ and $t_e(\tilde{w})=d\ell\tilde{w}$ are. In particular, $w(x)$ is affine linear in $d$ for fixed $x$.

Observe that $\mathcal{E}_d(x)$ is an integer point if and only if $x$ itself is and $w(x)\in D_v$. Thus, for each open face $\mathcal{F}$ of $\mathcal{C}_v$ with connected support, $\mathcal{E}_d$ gives a bijection between the set of disk-like vectors in $\mathcal{F}$ and the set of integer points in $\mathcal{E}_d(\mathcal{F})$.

We are now in a position to apply the following result of Calegari--Walker \cite[Corollary 3.7]{CW:Inthull}. To state it, we need the following definition.

\begin{definition}
	A finite subset $S(d)$ of $\Z^k$ depending on $d$ is \emph{QIQ} if there exists some positive integer $\pi$ such that for each $0\le p<\pi$, the coordinates of elements in $S(n\pi+d)$ are integral polynomials in $n$. 
\end{definition}

\begin{lemma}[Calegari--Walker \cite{CW:Inthull}]\label{lemma: CW}
	Let $V(d)$ be a cone with integral extremal vectors affine linear in $d$. Then the vertex set of the integer hull (open or closed) of $V(n\pi+a)-0$ is QIQ.
\end{lemma}

\begin{lemma}\label{lemma: vertex QIQ}
	The vertex set of $\conv(\mathcal{D}(v)^{(d)})+\mathcal{C}_v$ is eventually QIQ.
\end{lemma}
\begin{proof}
	Only consider large $d$ so that $\mathcal{C}_v$ does not depend on $d$. Fix an arbitrary open face $\mathcal{F}$ of $\mathcal{C}_v$ with connected support. Let $\{\omega_i\}_{i\in I}$ be finitely many embedded cycles in $Y$ representing the extremal rays on the boundary of $\mathcal{F}$. Denote the corresponding integer points in $\mathcal{C}_v$ also by $\omega_i$, $i\in I$. Then $\mathcal{E}_d(|D_v|\omega_i)=(|D_v|\omega_i,w(\omega_i))$ is an integer point depending affine linearly on $d$. Thus the vertex set of the convex hull of integer points in $\mathcal{E}_d(\mathcal{F})$ is QIQ by Lemma \ref{lemma: CW}. As we noticed earlier, these integer points are exactly the image of disk-like vectors in $\mathcal{F}$ under $\mathcal{E}_d$. Taking projection to the first coordinate shows that the vertex set of $\conv(\mathcal{D}(v)^{(d)}\cap \mathcal{F})$ is QIQ.
	
	Since each disk-like vector lies in one of the finitely many open faces of $\mathcal{C}_v$, we conclude that the vertex set of $\conv(\mathcal{D}(v)^{(d)})$ is also QIQ. The conclusion follows since each vertex of $\conv(\mathcal{D}(v)^{(d)})+\mathcal{C}_v$ is a vertex of $\conv(\mathcal{D}(v)^{(d)})$.
\end{proof}

\begin{proof}[Proof of Theorem \ref{thm: surgery}]
	We have proved the first assertion in Lemma \ref{lemma: limit is F_2}. Now adopt the notation above to prove the second assertion for $\BS(dm,d\ell)$.
	
	For $d$ sufficiently large, Lemma \ref{lemma: vertex QIQ} and the proof of Lemma \ref{lemma: D'} imply that, there is some positive integer $\pi$ so that for each $0\le p<\pi$, there are finitely many linear functions $\{f_i^{(n)}\}_{i\in I_p}$ where coefficients are rational functions in $n$, such that $\kappa_v^{n\pi+p}(x)=\min_{i\in I_p} f_i^{n}(x)$. Then $\scl_{\BS(dm,d\ell)}(\bar{c})$ is eventually quasirational by the linear programming problem formulated in the proof of Lemma \ref{lemma: compute by linprog}.
\end{proof}

\subsection{Implementation}
A direct implementation of the method introduced in Section \ref{sec: scl by linprog} to compute scl would probably result in an algorithm with run time doubly exponential on the word length (see \cite[Subsection 4.5]{Cal:sss}). However, in the case of Baumslag--Solitar groups, we have an algorithm computing $\scl_{\BS(M,L)}(c)$ with run time polynomial in the length of $c$ if the complexity $\rho(c)$ is fixed. This uses the idea developed by Walker \cite{Wal:scylla} to efficiently compute scl in free products of cyclic groups.

We use Setup 2 throughout out this subsection. Note that any piece has a multiple that is disk-like since $D_v$ is finite index in $G_v=\Z$. Thus we will restrict our attention to disk-like pieces only. To compute scl, it comes down to maximizing the number of disk-like pieces by finding the best combination. The idea is to cut disk-like pieces into small building blocks to reduced the complexity of enumeration.


We first transform winding numbers into lengths. For each piece $C$, we think of each arc (resp. turn) on the polygonal boundary as a segment consisting of $k$ unit intervals, where $k\in \{1,\ldots,|D_v|\}$ is the unique number congruent mod $|D_v|$ to the contribution of this arc (resp. turn) to $w(C)$. Then the length of $C$, ie the total number of unit intervals on its boundary, is divisible by $|D_v|$ since $C$ is disk-like.

It is easy to enumerate the possible types of unit intervals that appear in this way. For each arc $a_i\in A_v$, the number of intervals on it is the length of $a_i$, which does not exceed $|D_v|$. A similar bound holds for each turn from $a_i$ to $a_j$, and there are no more than $|D_v|$ different turns given $a_i,a_j$. Note that $|A_v|\le |c|$, where $|c|$ is the sum of word lengths of words involved in the chain $c$, so the total number $N$ of types of unit intervals is bounded above by $2|c|^2|D_v|^2$.

Suppose a (disk-like) piece $C$ has length $k|D_v|$ with $k\ge 1$. Choose $k$ consecutive unit intervals and form $(k-1)$ cuts to divide $C$ into $k$ \emph{blocks} so that each has length $|D_v|$ and each contains one of the $k$ unit intervals. See Figure \ref{fig: blocks} for an illustration. Then each block $B$ contains $s$ cuts with $s\in \{0,1,2\}$. The case $s=0$ appears only when the block itself is a disk-like piece of length $|D_v|$. In any case, let the orbifold Euler characteristic of a block $B$ be $\chi_o(B)=1-s/2$. Then $\sum\chi_o(B)=1$ where the summation is taken over the $k$ blocks $B$ obtained by cutting $C$.

\begin{figure}
	\centering
	\includegraphics[scale=01]{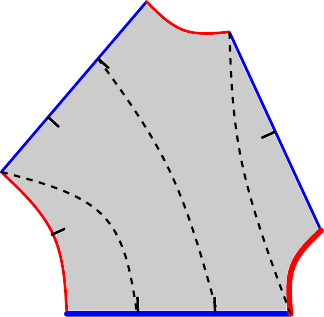}
	\caption{When $|D_v|=3$, a piece of total length $12$ can be cut into $k=4$ blocks each of length $|D_v|=3$, where the $4$ chosen consecutive unit intervals are thickened. The two blocks in the middle each contains $2$ genuine cuts, while the other two each contains $1$ genuine cut.}\label{fig: blocks}
\end{figure}

For each block obtained from $C$, label the $|D_v|$ unit intervals on it cyclically by $\{1,\ldots,|D_v|\}$, where the one labeled by $1$ is one of the $k$ chosen intervals on $C$.

For the $|D_v|$ unit intervals on a block $B$, two nearby intervals are either separated by a cut or not. If there is no cut in between, we think of it as a \emph{dummy cut}. Then each block consists of $|D_v|$ cuts, $s\in \{0,1,2\}$ of which are \emph{genuine}. Moreover, one of the two intervals connected by a genuine cut must be labeled by $1$. Each cut $q$ contributes to $\chi_o(B)$ by
$$\chi_o(q)\defeq\left\{
\begin{array}{cl}
1/|D_v|&\text{ if }q\text{ is dummy},\\
1/|D_v|-1/2 &\text{ if }q\text{ is genuine}.
\end{array}\right.$$

Now we can form a linear programming problem to compute the maximum of $\kappa_v(x)$ over $x\in \mathcal{C}(c)$, from which we can obtain $\scl(c)$. Each variable $x_q$ of the linear programming problem is the (normalized) number of a cut $q$ (dummy or genuine) that possibly appear. The objective function to be maximized is $\sum_q x_q\chi_o(q)$, which is linear in $x_q$.

Here each cut $q$ is encoded as a $5$-tuple
$$(I_b,I_k,I_{k+1},k,g),$$
where the binary variable $g=TRUE$ if the cut is genuine, $k\in\{1,\ldots, |D_v|\}$ (indices taken mod $|D_v|$) indicates that the cut connects the $k$-th and $(k+1)$-th intervals on a block, and $I_b,I_k,I_{k+1}$ represent the type of the intervals labeled by $1,k,k+1$ on the block respectively. This tuple satisfies the following restrictions:
\begin{enumerate}
	\item $g=FALSE$ if $k\neq 1$ or $|D_v|$;
	\item $I_k$ and $I_{k+1}$ must be \emph{consecutive} if $g=FALSE$: 
	If $I_k$ is the last interval on an arc $a_i$, then $I_{k+1}$ must be the first interval on a turn from $a_i$ to some $a_j$, otherwise $I_{k+1}$ must be the unique interval following $I_k$;
	\item $I_k=I_b$ if $k=1$; and
	\item $I_{k+1}=I_b$ if $k=|D_v|$;
\end{enumerate}

Besides the gluing and normalizing conditions, the linear programming problem here contains two more types of constraints.
\begin{enumerate}
	\item Boundary conditions that make sure cuts close up as blocks. This can be formulated to $\bdry_B(x)=0$ where $\bdry_B$ is defined as 
	$$\bdry_B(I_b,I_k,I_{k+1},k,g)\defeq(I_b,I_{k+1},k+1)-(I_b,I_{k},k)$$
	on basis elements and extends to a linear map to the abstract vector space with basis $\{(I_1,I_2,k): 1\le k\le |D_v|, I_1,I_2\text{ are types of intervals}\}$.
	\item Gluing conditions that make sure genuine cuts glue up with each other. This can be formulated similarly as the gluing conditions for turns. We say two pairs $(I,J)$ and $(I',J')$ of types of intervals can be glued if $I,J'$ are consecutive and $I',J$ are consecutive. We say $(I,J)$ and $(I',J')$ are equivalent if they both can be glued with some $(I'',J'')$. Then observe from the definition of being consecutive that if $(I,J)$ and $(I',J')$ are equivalent and $(I,J)$ can be glued with $(I'',J'')$, then $(I',J')$ can be glued with $(I'',J'')$ as well. Thus it makes sense to say whether two equivalence classes can be glued.
	
	Let $V$ be the vector space spanned by equivalence classes of pairs. Let $\#_P$ and $\#_N$ be linear maps to $V$ that vanish on $(I_b,I_k,I_{k+1},k,g)$ unless $g=TRUE$. When $g=TRUE$, if $k=1$, define $\#_P(I_b,I_k,I_{k+1},k,g)$ to be the image of $(I_k,I_{k+1})$ and define $\#_N(I_b,I_k,I_{k+1},k,g)\defeq 0$; Otherwise $k=|D_v|$ and we define $\#_N(I_b,I_k,I_{k+1},k,g)$ to be the image of $(I_k,I_{k+1})$ and define $\#_P(I_b,I_k,I_{k+1},k,g)\defeq0$. Then the gluing conditions are that the $\overline{(I,J)}$ coordinate of $\#_P(x)$ equals the $\overline{(I',J')}$ coordinate of $\#_Q(x)$ whenever the equivalence classes $\overline{(I,J)}$ and $\overline{(I',J')}$ can be glued.
\end{enumerate}

It is easy to show that the number of such tuples satisfying the restrictions is no more than $5|c|^4|D_v|^5=5|c|^4 d^5 |m|^{5\rho(c)} |\ell|^{5\rho(c)}$, which is polynomial in $|c|,d,m,\ell$ if the complexity $\rho(c)$ is bounded. This bounds the number of variables in the linear programming problem. Similar polynomial bounds on the number of constraints can be obtained. Fixing $\rho(c)$, this produces an algorithm that computes $\scl_{\BS(M,L)}(c)$ with time complexity polynomial in $|c|,|M|,|L|$ since linear programming problems can be solved by algorithms whose worst-case run time is polynomial \cite{LP-original}. Unfortunately, our algorithm does not run efficiently in practice, for example, $d^5 |m|^{5\rho(c)} |\ell|^{5\rho(c)}\approx6\times 10^7$ is already terribly large when $M=2$, $L=3$, and the complexity $\rho(c)=2$.

%
%
%

\end{document}